\documentclass[reqno,english,11pt]{amsart}
\usepackage[margin=1in]{geometry}
\usepackage[latin9]{inputenc}
\usepackage{mathrsfs}
\usepackage{bm}
\usepackage{amsbsy}
\usepackage{amstext}
\usepackage{amsthm}
\usepackage{amssymb}
\usepackage{hyperref}
\usepackage{xcolor}

\usepackage{stmaryrd}
\makeatletter

\newcommand{\lyxmathsym}[1]{\ifmmode\begingroup\def\b@ld{bold}
	\text{\ifx\math@version\b@ld\bfseries\fi#1}\endgroup\else#1\fi}

\numberwithin{equation}{section}
\numberwithin{figure}{section}
\theoremstyle{plain}
\newtheorem{thm}{\protect\theoremname}[section]
\theoremstyle{plain}
\newtheorem{cor}[thm]{\protect\corollaryname}
\theoremstyle{definition}
\newtheorem{defn}[thm]{\protect\definitionname}
\theoremstyle{plain}
\newtheorem{prop}[thm]{\protect\propositionname}
\theoremstyle{remark}
\newtheorem{rem}[thm]{\protect\remarkname}
\theoremstyle{plain}
\newtheorem{lem}[thm]{\protect\lemmaname}

\makeatother

\usepackage{babel}
\providecommand{\definitionname}{Definition}
\providecommand{\lemmaname}{Lemma}
\providecommand{\propositionname}{Proposition}
\providecommand{\corollaryname}{Corollary}
\providecommand{\remarkname}{Remark}
\providecommand{\theoremname}{Theorem}

\def\bR {\mathbb{R}}

\def\cD {\mathcal{D}}

\def\cY {\mathcal{Y}}

\def\grad {{\nabla}}

\def\la {\langle}
\def\ra {\rangle}

\newcommand{\bs}[1]{\boldsymbol{#1}}

\renewcommand{\ker}{\operatorname{ker}}

\newcommand{\eee}{\mathrm e}

\begin{document}
	\title[Strichartz estimates for Klein-Gordon]{Strichartz estimates for Klein-Gordon Equations \\with Moving Potentials}
	\author{Gong Chen and Jacek Jendrej}
	\begin{abstract}
		We study linear Klein-Gordon equations with moving potentials motivated by the stability analysis of traveling waves and multi-solitons.
		In this paper, Strichartz estimates, local energy decay and the scattering theory for these models are established.	 The results and estimates obtained in this paper will be used to study the interaction of solitons and the stability
		of multi-solitons of nonlinear Klein-Gordon equations.
	\end{abstract}
	
	\address[Chen]{School of Mathematics, Georgia Institute of Technology, Atlanta, GA 30332-0160, USA }
	\email{gc@math.gatech.edu}
	\address[Jendrej]{CNRS \& LAGA, Universit\'e Sorbonne Paris Nord, UMR 7539, 99 av J.-B.~Cl\'ement, 93430 Villetaneuse, France }
	\email{jendrej@math.univ-paris13.fr}
	\date{\today}
	
	\maketitle
	\setcounter{tocdepth}{1}
	
	\tableofcontents
	\section{Introduction}
	
	
	We consider linear Klein--Gordon equations
	with a finite number of moving potentials,
	\begin{equation}
	\partial_{t}^{2}u=\Delta u-u-\sum_{j=1}^{J}(V_{j})_{\beta_{j}(t)}(\cdot-y_{j}(t))u+F,
	\end{equation}see \eqref{eq:maineq} for the detailed setting. These models are motivated
	by the study of the stability analysis of traveling waves,  and multi-solitons which which are referred to  superpositions of a finite number
of Lorentz-transformed solitons, moving with distinct speeds. in the energy space.  In this paper, we establish local energy decay estimate and Strichartz estimates for these models. In our forthcoming work, Chen-Jendrej \cite{CJ2}, these estimates will be used to study nonlinear multi-soliton problems.
	\subsection{Motivation}
	Consider the nonlinear Klein-Gordon equation 
	\begin{equation}
	\partial_{t}^{2}\psi-\Delta \psi+\psi-\psi^{p}=0,\,\left(t,x\right)\in\mathbb{R}^{1+d}.\label{eq:nkg}
	\end{equation}
	Eliminating the time dependence, one can find stationary solutions  to the equation above which solve
	\begin{equation}
	-\Delta Q+Q-Q^{p}=0\label{eq:nQ}
	\end{equation}
	and decay exponentially.  In particular, there exists a unique radial positive ground
	state with the least energy
	\[
	E(Q):=\int_{\mathbb{R}^d}\frac{|\nabla Q|^2 +Q^2}{2}-\frac{Q^{p+1}}{p+1}\,dx
	\] among all non-zero solutions to the elliptic problem \eqref{eq:nQ}. We refer to Nakanishi-Schlag \cite{NSch,NSch1,NSch3} for more details.

	
	The nonlinear Klein-Gordon equation \eqref{eq:nkg} is a wave type equation, so one indispensable
	tool to study it is the Lorentz boost. Let $\beta\in\mathbb{R}^{d}$, $|\beta|<1$, be a velocity vector.
	For a function $\phi:\mathbb{R}^{d}\to\mathbb{R}^{d}$, the Lorentz
	boost of $\phi$ with respect to $\beta$ is given by
	\[
	\phi_{\beta}(x):=\phi(\Lambda_{\beta}x),\quad\Lambda_{\beta}x:=x+(\gamma-1)\frac{(\beta\cdot x)\beta}{|\beta|^{2}},\quad\gamma:=\frac{1}{\sqrt{1-|\beta|^{2}}}.
	\]
	With these notations, the Lorentz transformation is given by
	\[
	(t',x')=\big(\gamma(t-\beta\cdot x),\ \Lambda_{\beta}x-\gamma\beta t\big)=\big(\gamma(t-\beta\cdot x),\ \Lambda_{\beta}(x-\beta t)\big).
	\]
	It is crucial that for each $\beta\in\mathbb{R}^{d},\,\left|\beta\right|\in[0,1)$, if
	$u$ is a solution of \eqref{eq:nkg} then $u_{\beta}\left(x-\beta t\right)$
	is also a solution.
	

	
	Applying Lorentz transforms and the translational symmetry, we
	can obtain a family of traveling waves $Q_{\beta}\left(x-\beta t+x_{0}\right)$
	to \eqref{eq:nkg} from the stationary solution $Q$.  Using these traveling waves as building blocks, one can construct a
	pure multi-soliton to \eqref{eq:nkg} in the following sense:
	\begin{equation}
	\psi\rightarrow\sum_{j=1}^{J}\sigma_j Q_{\beta_{j}}\left(x-\beta_{j}t+y_j\right)\,\,\text{as}\,\,t\rightarrow\infty,\,\sigma_j\in\{\pm 1\},\label{eq:asyMuti}
	\end{equation}
	see C\^ote-Mu\~noz \cite{CMu}. Similar pure multi-soliton  can be constructed using other type of stationary solutions to \eqref{eq:nQ}, see Bellazzini-Ghimenti-Le Coz \cite{BGL} and C\^ote-Martel \cite{CMart}.
	To understand the long-time dynamics and soliton resolution of \eqref{eq:nkg},
	two crucial problems are to study the asymptotic stability of the
	multi-soliton structure
	\begin{equation}
	R\left(t,x\right)=\sum_{j=1}^{J}\sigma_j Q_{\beta_{j}}\left(x-\beta_{j}t+y_j\right)\label{eq:multi}
	\end{equation}
	and to give a classification of pure multi-solitons, i.e.,
	solutions satisfying \eqref{eq:asyMuti}. The fundamental tool to analyze
	the these problems is  Strichartz estimates for the underlining
	linear models.

	Roughly speaking, Strichartz estimates are dispersive and smoothing phenomena 
	on the $L_{x}^{p}$ space scale of linear dispersive flows. 
	Such linear estimates are
	of crucial importance in the analysis of nonlinear problems. On short
	time scales, dispersive bounds are a useful stepping stone toward
	a nonlinear local existence theory, and on long time scales they enable
	analysis of asymptotic properties and scattering behavior, see for
	example Muscalu-Schlag \cite{MS} and Tao \cite{Tao}. Also see Keel-Tao \cite{KT} for more
	details on the subject background and the historical development.


	Focusing on the stability analysis, starting from the one-soliton setting, superficially,
	we linearize the equation \eqref{eq:nkg} around $Q_{\beta}\left(x-\beta t\right)$
	with the decomposition $\psi=u+Q_{\beta}\left(x-\beta t\right)$ and
	obtain the equation for the error term
	\begin{equation}
	\partial_{t}^{2}u=\Delta u-u-V_{\beta}(\cdot-\beta t)u+F\left(u,Q_{\beta}\right)\label{eq:exlin}
	\end{equation}
	where $V=pQ_{\beta}^{p-1}$ and $F$ denotes higher order nonlinear
	terms in $u$. Due the symmetries of the equation, the linearized operator above naturally has some zero modes. To analyze the asymptotic behavior of $u$, we have
	to make the error term $u$ satisfy some orthogonality conditions.
	Then one has to invoke the modulation method and make the Lorentz
	parameter $\beta$ and an additional shift time-dependent. Therefore,
	the underlining model is given by
	\begin{equation}
	\partial_{t}^{2}u=\Delta u-V{}_{\beta(t)}(\cdot-y(t))u+F.\label{eq:maineq-2}
	\end{equation}
	Now the parameters $\dot{y}(t)-\beta(t)$ and $\dot{\beta}(t)$ satisfy the modulation equations which are quadratic in terms of $h$. 
	We refer to Nakanishi-Schlag \cite{NSch3} for the radial case and \cite{NSch}\footnote{In Nakanishi-Schlag \cite{NSch}, a Lorentz transform is applied from the beginning so that the stability analysis is reduced to the case that $\beta(0)=0$.} for the non-radial setting for the asymptotic stability and the dynamics around $Q$. 
	
	
	To study the multi-soliton solutions given by \eqref{eq:multi}, we
	need to analyze the multi-potential version of \eqref{eq:maineq-2}
	above. To formulate our problem, we start with the assumptions on potentials and trajectories.
	
	\smallskip
	
	\noindent\textbf{Assumption on potentials.} 
	Let $V_{j}$ be a smooth exponentially decaying
	potential for $j\in\{1,2,\ldots,J\}$, such that
	\begin{equation}
	L_{j}:=-\Delta+V_{j}+1\label{eq:schop}
	\end{equation}
	has $K_{j}$ strictly negative eigenvalues $-\nu_{j,k}^{2}$ with
	$\nu_{j,k}>0$ (for $k=1,\ldots,K_{j}$) and $M_{j}=\dim\ker L_{j}$.
	Let $\left(\phi_{j,k}\right)_{k=1,\ldots,K_{j}}$ and $\left(\phi_{j,m}^{0}\right)_{m=1,\ldots,M_{j}}$
	be orthonormal (in $L^{2}$) families such that
	\begin{equation}
	L_{j}\phi_{j,k}=-\nu_{j,k}^{2}\phi_{k},\,\,L_{j}\phi_{j,m}^{0}=0.\label{eq:modescalar}
	\end{equation}
	We also
	assume that $L_{j}$ has no resonances at $1$.\footnote{ Recall that $\psi$
		is a resonance of $L_{j}$ at $1$ if it is a distributional solution
		of the equation $L_{j}\psi=\psi$ which belongs to the space $L^{2}\left(\left\langle x\right\rangle ^{-\sigma}dx\right):=\left\{ f:\,\left\langle x\right\rangle ^{-\sigma}f\in L^{2}\right\} $
		for any $\sigma>\frac{1}{2}$, but not for $\sigma=\frac{1}{2}.$} In particular, the potential coming from the linearization around the ground state $Q$ given by \eqref{eq:nQ} satisfies the assumption above with $K_j=1$ and $M_j=d$. 
	
	
	\noindent\textbf{Assumption on trajectories.} 
	Let $y_{j}(t)$ be positions of the potentials. Denote $\beta_{j}(t)$
	the Lorentz factor for the potential $V_j$. We write $\boldsymbol{\beta}(t)=(\beta_{1}(t),\ldots,\beta_{J}(t))\in\mathbb{R}^{3J}$,
	$\boldsymbol{y}(t)=(y_{1}(t),\ldots,y_{J}(t))\in\mathbb{R}^{3J}$.
	Motivated by the modulation equations from the stability analysis, we impose that $y_{j}\left(0\right)=y_j,$
	$\beta_{j}\left(0\right)=\beta_{j}$ with $|\beta_j|<1$, $\beta_j\neq \beta_k$ for $k\neq j$ and
	\begin{equation}
	\left\Vert \beta_{j}'\left(t\right)\right\Vert _{L_{t}^{1}\bigcap L_{t}^{\infty}}+\left\Vert y'_{j}\left(t\right)-\beta_{j}\left(t\right)\right\Vert _{L_{t}^{1}\bigcap L_{t}^{\infty}}\lesssim\delta\ll1.\label{eq:maincond}
	\end{equation}In the application of the  nonlinear analysis of multi-soliton solutions, these conditions are consistent with the modulation equations which are quadratic in terms of radiation terms and the view of the endpoint Strichartz estimate.
	
	The main focus of
	this paper is the following linear Klein-Gordon equation
	\begin{equation}
	\partial_{t}^{2}u=\Delta u-u-\sum_{j=1}^{J}(V_{j})_{\beta_{j}(t)}(\cdot-y_{j}(t))u+F\label{eq:maineq}
	\end{equation}
	under assumptions \eqref{eq:maincond}. 	Note that the Lorentz transformation is applied to the potentials
	$V_{j}$, according to their the Lorentz factors $\beta_{j}\left(t\right)$. 	\footnote{It is possible to allow slow rotations applied to potentials. Since we will not apply this in the nonlinear problem, we do not pursue this point in this paper.}

	The time-dependent potential in the homogeneous part of equation \eqref{eq:maineq}
	is a generalized version of the charge transfer Hamiltonian where
	the trajectories of potentials are straight lines, see for example Rodnianski-Schlag-Soffer
	\cite{RSS} and Chen \cite{C1}.
	
	\subsection{Basic setting and main results}
	In this subsection, we introduce the basic settings, functional spaces and  present main results of our paper. 
	\subsubsection{Exponential dichotomy}
	
	Now we introduce notations and recall
	exponential dichotomy for the Klein-Gordon equation from Chen-Jendrej
	\cite{CJ}. For an alternative approach, we refer to Metcalfe-Sterbenz-Tataru \cite{MST}.
	
	Consider the homogeneous version of our main equation \eqref{eq:maineq}
	\begin{equation}\label{eq:homomaineq}
	\partial_{t}^{2}u=\Delta u-u-\sum_{j=1}^{J}(V_{j})_{\beta_{j}(t)}(\cdot-y_{j}(t))u. 
	\end{equation}
	First of all, we can write the equation above as a dynamical system.
	Denoting\[\boldsymbol{u}\left(t\right)=\left(\begin{array}{c}
	u\\
	u_{t}
	\end{array}\right)\] we can write
	\begin{equation}
	\partial_{t}\boldsymbol{u}\left(t\right)=\mathcal{J}H\left(t\right)\boldsymbol{u}\left(t\right),\label{eq:dy1}
	\end{equation}
	where
	\[
	\mathcal{J}:=\left(\begin{array}{cc}
	0 & 1\\
	-1 & 0
	\end{array}\right),\,\ \ H\left(t\right):=\left(\begin{array}{cc}
	-\Delta+1+\sum_{j=1}^{J}(V_{j})_{\beta_{j}(t)}(\cdot-y_{j}(t)) & 0\\
	0 & 1
	\end{array}\right).
	\]
	Denote the evolution operator of the above system as $\mathcal{T}\left(t,s\right)$.
	It introduces a continuous forward dynamical system in $\mathcal{H}=H^{1}\times L^{2}$.
	For a general element in $\boldsymbol{v}\in \mathcal{H}$, $\boldsymbol{v}_{1}$
	denotes the first row and $\boldsymbol{v}_{2}$ denotes the second
	row of $\boldsymbol{v}$ respectively.
	
	From the spectral assumptions on each potential, \eqref{eq:modescalar}, following C\^ote-Martel \cite{CMart}, C\^ote-Martel \cite[Lemma 1]{CMu} and Chen-Jendrej \cite{CJ},  we now give explicit formulas for the stable, unstable
	and (iterated) null components of the flow.
	First of all, we define
	\begin{align}
		\cY_{j,k, \beta}^-(x) &:= \eee^{\gamma\nu_{j,k}\beta\cdot x}(\phi_{j,k}, -\gamma\beta\cdot\grad \phi_{j,k}-\gamma\nu_{j,k}\phi_{j,k})_\beta(x), \label{eq:kg-Ym-def} \\
		\cY_{j,k, \beta}^+(x) &:= \eee^{-\gamma\nu_{j,k}\beta\cdot x}
		(\phi_{j,k}, -\gamma\beta\cdot\grad \phi_{j,k}+\gamma\nu_{j,k}\phi_{j,k})_\beta(x),\label{eq:kg-Yp-def} 
		\\ \cY_{j,m, \beta}^0(x) &:= (\phi_{j,m}^0, -\gamma\beta\cdot\grad \phi_{j,m}^0)_\beta(x),\label{eq:kg-Y0-def} 
		\\ \cY_{j,m, \beta}^1(x) &:= ({-}(\beta \cdot x)\phi_{j,m}^0, \gamma \phi_{j,m}^0 + \gamma(\beta\cdot x)(\beta\cdot \grad \phi_{j,m}^0))_\beta(x),\label{eq:kg-Y1-def} 
		\\ \alpha_{j,k, \beta}^-(x) &:= \mathcal{J}\cY_{k, \beta}^+(x) = \eee^{-\gamma\nu_{j,k}\beta\cdot x}({-}\gamma\beta\cdot\grad \phi_{j,k} + \gamma\nu_{j,k}\phi_{j,k}, -\phi_{j,k})_\beta(x), \label{eq:kg-am-def} \\
		\alpha_{j,k, \beta}^+(x) &:= \mathcal{J}\cY_{k, \beta}^-(x) = \eee^{\gamma\nu_{j,k}\beta\cdot x}({-}\gamma\beta\cdot\grad \phi_{j,k} - \gamma\nu_{j,k}\phi_{j,k}, -\phi_{j,k})_\beta(x), \label{eq:kg-ap-def} \\
		\alpha_{j,m, \beta}^0(x) &:= \mathcal{J}\cY_{m, \beta}^0(x) = ({-}\gamma \beta\cdot\grad \phi_{j,m}^0, -\phi_{j,m}^0)_\beta(x),\label{eq:kg-a0-def} 
		\\ \alpha_{j,m, \beta}^1(x) &:= \mathcal{J}\cY_{m, \beta}^1(x) = (\gamma \phi_{j,m}^0 + \gamma(\beta\cdot x)(\beta\cdot \grad \phi_{j,m}^0), (\beta \cdot x)\phi_{j,m}^0)_\beta(x).\label{eq:kg-a1-def} 
	\end{align}
	Then we have the following important modes for the moving-potential problem \eqref{eq:dy1} above:
	\begin{align*}
		\cY_{j, k}^-(t) &:= \cY_{j,k, \beta_j(t)}^-(\cdot - y_j(t)), \\
		\cY_{j, k}^+(t) &:= \cY_{j,k, \beta_j(t)}^+(\cdot - y_j(t)), \\
		\cY_{j, m}^0(t) &:= \cY_{j,m, \beta_j(t)}^0(\cdot - y_j(t)), \\
		\cY_{j, m}^1(t) &:= \cY_{j,m, \beta_j(t)}^1(\cdot - y_j(t)), \\
		\alpha_{j, k}^-(t) &:= \alpha_{j,k, \beta_j(t)}^-(\cdot - y_j(t)), \\
		\alpha_{j, k}^+(t) &:= \alpha_{j,k, \beta_j(t)}^+(\cdot - y_j(t)), \\
		\alpha_{j, m}^0(t) &:= \alpha_{j,m, \beta_j(t)}^0(\cdot - y_j(t)), \\
		\alpha_{j, m}^1(t) &:= \alpha_{j,m, \beta_j(t)}^1(\cdot - y_j(t)),
	\end{align*}
	where $k \in \{1, \ldots, K_j\}$ and $m \in \{1, \ldots, M_j\}$.
	
	In Chen-Jendrej \cite{CJ}, we gave a sufficient condition for the
	existence of an exponential dichotomy for a general linear dynamical
	system (not necessarily invertible) in a Banach space. In particular,
	our analysis applied to the Klein--Gordon equation with a finite
	number of potentials whose centers move at sublight speeds with small
	accelerations. We obtain the stable, unstable  and centre spaces of the linear
	flow \eqref{eq:dy1}.  Morally, this dichotomy plays the role of the spectral decomposition in the standard case with a stationary potential via the spectral theory.  We now recall the main result from Chen-Jendrej \cite{CJ} applied to our current setting.
	\begin{prop}
		\label{prop:expdicho}Let $\max_j |\beta_j(t)|\leq v<1$, $\nu_{e}:=\min\{\nu_{j,k}\}$ and
		$K:=\sum_{j=1}^{J}K_{j}$. Under the assumptions on the trajectories \eqref{eq:maincond}, suppose there exists $\eta>0$ such that
		for $t$ large enough,
		\[
		|y_{j}(t)-y_{l}(t)|\geq\frac{1}{\eta}\qquad\text{for all }j\neq l.
		\]
		Consider the semigroup generated $T\left(\tau,t\right)$ associated
		with the linear flow in $\mathcal{H}:=H^{1}\times L^{2}$. Then there exist a subspace $X_{\text{u}}\left(t\right)$
		of dimension $K$, a subspace $X_{\text{cs}}\left(t\right)$
		of codimension $K$, a subspace $X_{\text{s}}\left(t\right)$
		of dimension $K$ and a subspace $X_{\text{c}}(t)$ such that $X_{\text{s}}\left(t\right)\oplus X_{\text{c}}(t)=X_{\text{cs}}\left(t\right)$,
		$X_{\text{cs}}\left(t\right)\varoplus X_{\text{u}}\left(t\right)=\mathcal{H}$
		and the followings are true:
		\begin{enumerate}
			\item If $\hm{h}_{t}\in X_{\text{s}}\left(t\right)$ then $\limsup_{\tau\rightarrow\infty}\frac{1}{\tau-t}\log\left\Vert \mathcal{T}\left(\tau,t\right)\hm{h}_{t}\right\Vert _{\mathcal{H}}\leq-\nu_{\epsilon}\sqrt{1-v^{2}}$.
			\item If $\hm{h}_{t}\in X_{\text{c}}(t)$ and  $\hm{h}_{t}\neq\hm{0}$, then $\lim_{\tau\rightarrow\infty}\frac{1}{\tau-t}\log\left\Vert \mathcal{T}\left(\tau,t\right)\hm{h}_{t}\right\Vert _{\mathcal{H}}=0$.
			\item If $\hm{h}_{t}\in X_{u}\left(t\right)$  and  $\hm{h}_{t}\neq\hm{0}$, then $\lim\inf_{\tau\rightarrow\infty}\frac{1}{\tau-t}\log\left\Vert \mathcal{T}\left(\tau,t\right)\hm{h}_{t}\right\Vert _{\mathcal{H}}\geq\nu_{\epsilon}\sqrt{1-v^{2}}.$
		\end{enumerate}
		Moreover, there exist projections $\pi_{\text{cs}}(t):\mathcal{H}\to X_{\text{cs}}(t)$
		and $\pi_{\text{u}}(t):\mathcal{H}\to X_{\text{u}}(t)$ such that
		\begin{enumerate}
			\item these projections commute with the flow
			\[
			\mathcal{T}(\tau,t)\circ\pi_{\text{cs }}(t)=\pi_{c\text{s}}\left(\tau\right)\circ\mathcal{T}(\tau,t)
			\]
			and
			\[
			\mathcal{T}(\tau,t)\circ\pi_{\text{u}}(t)=\pi_{\text{u}}(\tau)\circ\mathcal{T}(\tau,t),
			\]
			\item there exists a constant  $C$ such that
			\[
			\|\pi_{\text{cs}}(t)\|_{\mathscr{L}(\mathcal{H})}+\|\pi_{\text{u}}(t)\|_{\mathscr{L}(\mathcal{H})}\leq C.
			\]
		\end{enumerate}
		One can also find projections $\pi_{\text{s}}\left(t\right):\mathcal{H}\rightarrow X_{\text{s}}\left(t\right)$,
		$\pi_{\text{c}}\left(t\right):\mathcal{H}\rightarrow X_{\text{c}}\left(t\right)$
		such that
		\[
		\mathcal{T}(\tau,t)\circ\pi_{\text{s}}(t)=\pi_{\text{s}}(\tau)\circ\mathcal{T}(\tau,t),\ \|\pi_{\text{s}}(t)\|_{\mathscr{L}(\mathcal{H})}\leq C
		\]
		and
		\[
		\mathcal{T}(\tau,t)\circ\pi_{\text{c}}(t)=\pi_{\text{c}}(\tau)\circ\mathcal{T}(\tau,t),\ \|\pi_{\text{c}}(t)\|_{\mathscr{L}(\mathcal{H})}\leq C.
		\]
		
	\end{prop}
	
	\begin{defn}
		We will call $X_{\text{s}}\left(t\right)$ the unstable space, $X_{\text{cs}}\left(t\right)$
		the centre-stable space, $X_{\text{s}}\left(t\right)$ the stable
		space and $X_{\text{c}}\left(t\right)$ the centre space.
	\end{defn}
	To obtain dispersive estimates, naturally, one has to restrict flows onto the centre-stable space.  In this subspace, one needs to further remove  null components to avoid the polynomial growth.
	\begin{defn}\label{def:zeroproj}
		We define the projection $\pi_0(t)$ as the the projection onto the subspace spanned by $\cY_{j, m}^0(t)$ and $\cY_{j, m}^1(t)$.  We also define  $\pi_\pm(t)$ as projections onto the spans of $\cY_{j, m}^{\pm}(t)$ respectively.
	\end{defn}
	In this paper, we consider solutions such that \begin{equation}\label{eq:maincondzero}
	\pi_0(t)\bm{u}(t)=0,\,\,\forall t\in\mathbb{R}.
	\end{equation}In other words,$$\left<\alpha_{j, m}^0(t),\bm{u}(t)\right>=\left<
	\alpha_{j, m}^1(t),\bm{u}(t)\right>=0,\qquad\forall j,m.$$ In the practical nonlinear application, this condition is always achieved by the modulation techniques.

	\subsubsection{Functional spaces}
	
	Next, we define  important functional spaces in our analysis. Setting
	\begin{equation}
	\mathcal{D}:=\sqrt{1-\Delta},\label{eq:D}
	\end{equation}
	firstly we define the space
	\begin{equation}
	\mathcal{D}^{\frac{\nu}{2}}L_{x}^{2}:=\left\{ g:\,g=\mathcal{D}^{\frac{\nu}{2}}f,\,\text{for some}\,f\in L_{x}^{2}\right\} .\label{eq:Dweightspace}
	\end{equation}
	Given $g\in\mathcal{D}^{\frac{\nu}{2}}L_{x}^{2}$ with $g=\mathcal{D}^{\frac{\nu}{2}}f$, 
	then norm of it given by
	\begin{equation}
	\text{\ensuremath{\left\Vert g\right\Vert }}_{\mathcal{D}^{\frac{\nu}{2}}L_{x}^{2}}:=\text{\ensuremath{\left\Vert \mathcal{D}^{-\frac{\nu}{2}}g\right\Vert }}_{L_{x}^{2}}=\text{\ensuremath{\left\Vert f\right\Vert }}_{L_{x}^{2}}.\label{eq:Dweightnorm}
	\end{equation}
	Secondly, we define the weighed $L^{2}$ space with the center $y$
	as
	\begin{equation}
	L_{x}^{2}\left\langle \cdot-y\right\rangle ^{\sigma}:=\left\{ g:g=\left\langle \cdot-y\right\rangle ^{-\alpha}f,\,\text{for some}\,f\in L_{x}^{2}\right\} .\label{eq:Xweightspace}
	\end{equation}
	Given $g\in L_{x}^{2}\left\langle \cdot-y\right\rangle ^{\sigma}$ with
	$g=\left\langle \cdot-y\right\rangle ^{-\alpha}f$ then norm of it given
	by
	\begin{equation}
	\text{\ensuremath{\left\Vert g\right\Vert }}_{L_{x}^{2}\left\langle \cdot-y\right\rangle ^{\sigma}}:=\text{\ensuremath{\left\Vert \left\langle \cdot-y\right\rangle ^{\alpha}g\right\Vert }}_{L_{x}^{2}}=\text{\ensuremath{\left\Vert f\right\Vert }}_{L_{x}^{2}}.\label{eq:Xweightnorm}
	\end{equation}
	Let $B_{p,q}^{s}(\mathbb{R}^{d})$ denote the inhomogeneous Besov
	space based on $L^{p}(\mathbb{R}^{d})$ for any $d\ge1$, $s\in\mathbb{R}$
	and $p,q\in[1,\infty]$. For brevity, we use the standard notation
	\[
	H^{s}:=B_{2,2}^{s},\ C^{s}:=B_{\infty,\infty}^{s}.
	\]
	The homogeneous versions are denoted by $\dot{B}_{p,q}^{s}$, $\dot{H}^{s}$
	and $\dot{C}^{s}$, respectively. For $s\in(0,1)$, we have the equivalent
	semi-norms by the difference, see Bergh-L\"ofstr\"om \cite[Exercice 7,  page 162]{BL} and  Bahouri-Chemin-Danchin \cite[Theorem 2.36]{BCD},
	\[
	\|\varphi\|_{\dot{B}_{p,q}^{s}}\simeq\|\sup_{|y|\le\sigma}\|\varphi(x)-\varphi(x-y)\|_{L^{p}}\|_{L^{q}(d\sigma/\sigma)}.
	\]
	Finally, the weighted $L^{2}$ space,  $L^{2,s}(\mathbb{R}^{d})$, is defined
	by the norm
	\[
	\|\varphi\|_{L^{2,s}\left(\mathbb{R}^{d}\right)}=\|\left\langle x\right\rangle ^{s}\varphi\|_{L^{2}(\mathbb{R}^{d})}
	\]
	for any $s\in\mathbb{R}$. Hence $L^{2,s}$ is the Fourier image
	of $H^{s}$. 
	
	\subsubsection{Main results}
	
	Hereafter, throughout this paper,  we restrict our attention to dimension $\mathbb{R}^{3+1}$.
	But transparently, one can observe that the same analysis can be applied
	to problems with dimensions $\mathbb{R}^{d+1}$ with $d\geq3$.

	With the notations above, we can state the main results in this paper.
	\begin{thm}
		\label{thm:mainthmlinear}For $0<\nu\ll1$ and $\sigma>14$, denote
		\[
		\mathfrak{W}:=L_{t}^{2}\left(\bigcap_{j=1}^{J}\mathcal{D}^{\frac{\nu}{2}}L_{x}^{2}\left\langle \cdot-y_{j}\left(t\right)\right\rangle ^{-\sigma}\right)
		\]
		and the Strichartz norm
		\[
		S=L_{t}^{\infty}L_{x}^{2}\bigcap L_{t}^{2}B_{6,2}^{-5/6}.
		\]
		Consider the system
		\begin{equation}\label{eq:equ}
		\partial_{t}\boldsymbol{u}\left(t\right)=\mathcal{J}H\left(t\right)\boldsymbol{u}\left(t\right)+\boldsymbol{F}
		\end{equation}such that $$\pi_0(t)\bm{u}(t)=0,\,\,\forall t\in\mathbb{R}.$$
		Using the notations above, one has the local energy decay estimate and  Strichartz estimates
		\begin{equation}
		\left\Vert \mathcal{D}\left(\pi_{cs}\left(t\right)\boldsymbol{u}\left(t\right)\right)_{1}\right\Vert _{\mathfrak{W}\bigcap S}+\left\Vert \left(\pi_{cs}\left(t\right)\boldsymbol{u}\left(t\right)\right)_{2}\right\Vert _{\mathfrak{W}\bigcap S}\lesssim\left\Vert \hm{u}\left(0\right)\right\Vert _{\mathcal{H}}+\|\cD F_1\|_{(\mathfrak{W}\cap S)^{*}}+\| F_2\|_{(\mathfrak{W}\cap S)^{*}}.\label{eq:mainstrichartz}
		\end{equation}
		Moreover, indeed $\pi_{s}\left(t\right)\boldsymbol{u}\left(t\right)$
		scatters to a free wave. There exists $\boldsymbol{\psi}_{+}\in \mathcal{H}$
		such that
		\begin{equation}
		\left\Vert \pi_{cs}\left(t\right)\boldsymbol{u}\left(t\right)-e^{\mathcal{J}H_{0}t}\boldsymbol{\psi}_{+}\right\Vert _{\mathcal{H}}\rightarrow0,\,\,t\rightarrow\infty.\label{eq:mainscatter}
		\end{equation}
	\end{thm}
	This theorem  in particular shows that  the solution in the centre-stable space
	constructed in Chen-Jendrej \cite{CJ} indeed scatters to the free
	linear flow if it is orthogonal to zero modes.

	\begin{rem}
		One can also restrict the time interval in the estimate \eqref{eq:mainstrichartz}
		to $\left[0,T\right]$. We can also take the initial data at $t_{0}=T$
		and solve the equation backwards. Then one has
		\[
		\left\Vert \mathcal{D}\left(\pi_{c}\left(t\right)\boldsymbol{u}\left(t\right)\right)_{1}\right\Vert _{\mathfrak{W}\bigcap S}+\left\Vert \left(\pi_{c}\left(t\right)\boldsymbol{u}\left(t\right)\right)_{2}\right\Vert _{\mathfrak{W}\bigcap S}\lesssim\left\Vert u\left(T\right)\right\Vert _{\mathcal{H}}+\|\cD F_1\|_{(\mathfrak{W}\cap S)^{*}}+\| F_2\|_{(\mathfrak{W}\cap S)^{*}}.
		\]Note that here we restrict onto the centre space.
	\end{rem}
	
	As a direct inspection of our proof, see Remark \ref{rem:proplinear}, we also have the following proposition which is useful in the nonlinear setting, see Chen-Jendrej  \cite{CJ2}.  We record it here for the sake of completeness.
	\begin{prop}
		\label{pro:linear}
		Let $\bs h$ be a solution of
		\begin{equation}\label{eq:linearpro}
		\partial_{t}\bs h(t)=\mathcal{J}H(t)\bs h(t)+\bs F(t)
		\end{equation}
		such that
		\begin{equation}\label{eq:linearcond}
		\begin{aligned}
		&\la \alpha_{j, \beta_j(t)}^0(\cdot - y_j(t)), \bs h(t)\ra = \la \alpha_{j, \beta_j(t)}^1(\cdot - y_j(t)), \bs h(t)\ra = 0\\
		&\la \alpha_{j, \beta_j(t)}^-(\cdot - y_j(t)), \bs h(t)\ra = \la \alpha_{j, \beta_j(t)}^+(\cdot - y_j(t)), \bs h(t)\ra = 0,\qquad\forall t\in\bR, 1 \leq j \leq J.
		\end{aligned}
		\end{equation}
		Using the notations from Theorem \ref{thm:mainthmlinear}, one has the local energy decay estimate and  Strichartz estimates
		\begin{equation}
		\label{eq:mainstrichartzh}
		\|\cD h_1\|_{\mathfrak{W}\cap S}+\| h_2\|_{\mathfrak{W}\cap S}\lesssim\|\bs h_0\|_{\mathcal{H}}+\|\cD F_1\|_{(\mathfrak{W}\cap S)^{*}}+\| F_2\|_{(\mathfrak{W}\cap S)^{*}}.\end{equation}
		Moreover, $\bs h$ scatters to a free wave: there exists $\bs{\psi}_{+}\in \mathcal{H}$
		such that
		\begin{equation}
		\|\bs h(t)-\eee^{\mathcal{J}H_{0}t}\bs{\psi}_{+}\|_{\mathcal{H}}\to 0,\qquad\text{as }t\to\infty.\label{eq:mainscatterh}
		\end{equation}
	\end{prop}
	
	\begin{rem}
		One can also work on the linear system  	\eqref{eq:linearpro} and \eqref{eq:linearcond} in the proposition above directly. Then there is no need to invoke the exponential dichotomy since by assumptions, we remove the unstable directions directly. Given the result from the proposition above, one can alternatively prove the existence of exponential dichotomy, see Chen-Jendrej \cite{CJ1}.
	\end{rem}
\begin{rem}\label{rem:weakenlinear}
The condition \eqref{eq:linearcond} can be weaken to 
	\begin{equation}\label{eq:linearrem}
		\begin{aligned}
		&\big |\la \alpha_{j, \beta_j(t)}^0(\cdot - y_j(t)), \bs h(t)\ra\big | + \big |\la \alpha_{j, \beta_j(t)}^1(\cdot - y_j(t)), \bs h(t)\ra \big|\\
		&+\big|\la \alpha_{j, \beta_j(t)}^-(\cdot - y_j(t)), \bs h(t)\ra \big|+\big| \la \alpha_{j, \beta_j(t)}^+(\cdot - y_j(t)), \bs h(t)\ra\big |\in L^2_t, 1 \leq j \leq J.
		\end{aligned}
		\end{equation}
Denote the sum of the $L^2_t$ norms of all expressions above as $\mathfrak{U}$, then one has
\begin{equation}
		\|\cD h_1\|_{\mathfrak{W}\cap S}+\| h_2\|_{\mathfrak{W}\cap S}\lesssim\|\bs h_0\|_{\mathcal{H}}+\|\cD F_1\|_{(\mathfrak{W}\cap S)^{*}}+\| F_2\|_{(\mathfrak{W}\cap S)^{*}}+\mathfrak{U}.\end{equation}
In nonlinear applications,
\begin{equation*}
    \big |\la \alpha_{j, \beta_j(t)}^0(\cdot - y_j(t)), \bs h(t)\ra\big | + \big |\la \alpha_{j, \beta_j(t)}^1(\cdot - y_j(t)), \bs h(t)\ra \big|=0
\end{equation*} is achieved by the modulation analysis and
\begin{equation*}
 \big|\la \alpha_{j, \beta_j(t)}^-(\cdot - y_j(t)), \bs h(t)\ra \big|+\big| \la \alpha_{j, \beta_j(t)}^+(\cdot - y_j(t)), \bs h(t)\ra\big |\in L^2_t   
\end{equation*}
is analyzed separately via projecting onto the unstable/stable modes.
\end{rem}	
	In Theorem \ref{thm:mainthmlinear}, we conclude that the centre-stable part of the solution to the perturbed system scatters to the free wave.  From the view of the scattering theory, it is also desirable to conclude the other direction, i.e., the existence of wave operators. To simplify the problem, for this problem, we further assume that all $\phi_{j,m}^{0}$ from \eqref{eq:modescalar} are $0$.

	After establishing the backward estimates, we obtain the existence of the wave operator for our linear model.
	\begin{thm}
		Given any free data $\boldsymbol{\phi}_{0}\in \mathcal{H}$, one can find a unique  perturbed data $\boldsymbol{u}(0)\in X_\text{c}(0)$ such that the solution $\boldsymbol{u}(t)$ to \eqref{eq:equ} with the initial data $\boldsymbol{u}(0)$ satisfies
		\begin{equation}
		\left\Vert \pi_{c}\left(t\right)\boldsymbol{u}\left(t\right)-e^{\mathcal{J}H_{0}t}\boldsymbol{\phi}_{0}\right\Vert _{\mathcal{H}}\rightarrow0,\,\,t\rightarrow\infty.\label{eq:exiwave}
		\end{equation}
	\end{thm}

	Given the result above and the stable space, we have the following result for the case when $\boldsymbol{F}(t)=\pi_{cs}(t)\boldsymbol{F}(t)$. 
	\begin{cor}
		Assuming $\boldsymbol{F}(t)=\pi_{cs}(t)\boldsymbol{F}(t)$, for any free data $\boldsymbol{\phi}_{0}\in H^1\times L^2$, there is a dimension $K$  manifold $M\in \mathcal{H}$ such that for any $\boldsymbol{u}(0)\in M$, 
		the solution $\boldsymbol{u}(t)$  to \eqref{eq:equ} with the initial data $\boldsymbol{u}(0)$ satisfies
		\begin{equation}
		\left\Vert\boldsymbol{u}\left(t\right)-e^{\mathcal{J}H_{0}t}\boldsymbol{\phi}_{0}\right\Vert _{\mathcal{H}}\rightarrow0,\,\,t\rightarrow\infty.\label{eq:exiwave1}
		\end{equation}
	\end{cor}
	\subsection{Main ideas}
	Here in this subsection, we illustrate some major difficulties. We  also sketch and highlight some key ideas in our poof to overcome those obstacles. 
	
	Consider the linear equation given by 
	\begin{equation*}
		\partial_{t}^{2}u=\Delta u-u-\sum_{j=1}^{J}(V_{j})_{\beta_{j}(t)}(\cdot-y_{j}(t))u+F
	\end{equation*}which can be regarded as a generalized version of the linear Klein-Gordon equation with charge transfer Hamiltonian. 
	Using the ideas from the study of the charge transfer models in various
	setting, see Rodnianski-Schlag-Soffer \cite{RSS} and Chen \cite{C1}
	for the Schr\"odinger and wave settings respectively, we first use
	the decomposition in channels. Namely, we will decompose the space
	into several different regions and each of them will be dominated
	by the influences of different potentials. Then we will use the Duhamel formulas to expand the solution with respect to different potentials in their dominated channels.
	
	
	We first note that  in the current setting,
	the trajectories of potentials are motivated by the interaction of solitons, whence they are much more complicated than those
	models studied before.  One can not apply the invariance of the principle part, the Lorentz invariance, to reduce potentials to the static problems.  Moreover, the negative eigenvalues associated
	with each moving potentials will produce exponential unstable modes.
	Naturally, to obtain Strichartz estimates and local energy decay estimates,
	we have to project the flow away from unstable behavior. Unlike the static
	potential problem and one single moving potential problem, due to
	the interaction of potentials and unstable modes, the projections
	onto the stable, centre and unstable spaces are far from being standard. To
	obtain these projections, we use the exponential dichotomy applied
	to the Klein-Gordon equation with multiple potentials  in
	Chen-Jendrej \cite{CJ}.
	

	The second main difficulty in our setting is that as introduced above,
	motivated by the interaction solitons and the modulation method, the
	trajectories are perturbations of the linear motions. Due to the perturbative
	feature of the trajectories of potentials, we further develop the ideas stemming
	from Beceanu \cite{Bec1,Bec2} for the Schr\"odinger problem and
	Nakanishi-Schlag \cite{NSch} in the Klein-Gordon setting. In Nakanishi-Schlag
	\cite{NSch}, after performing a Lorentz transform first, the authors analyzed the one-potential problem with
	small velocity. Here we need to handle multiple potentials and large
	velocities.  We first establish the analysis in Nakanishi-Schlag \cite{NSch}
	with a moving potential and estimates in non-shifted coordinates. In
	this setting, we need a very general version of the local energy estimate
	for both homogeneous and inhomogeneous Klein-Gordon equation with a
	moving potential. Traditionally, as a wave type equation, one always
	applies the Lorentz transform to reduce the moving potential problem
	to the standard case with a static potential. But by doing this, we
	will miss the general form the inhomogeneous local energy decay. To
	achieve general local decay estimates, instead of using Lorentz transforms,
	here we apply the Galilean change of variable and obtain a system 
	\begin{equation}
	\frac{d}{dt}\hm{u}  =\left(\begin{array}{cc}
	\beta\cdot\nabla & 1\\
	\Delta-V_{\beta}\left(x\right)-1 & \beta\cdot\nabla
	\end{array}\right)\hm{u}+\hm{F}
	\end{equation}
	where $\hm{u}=\left(\begin{array}{c}
	u\\
	u_{t}
	\end{array}\right)$. In this setting,  we are allowed to  put a general inhomogeneous term $\hm{F}$ (the $\hm{F}$ coming from the scalar equation has a vanishing first component). These type of general inhomogeneous terms appear in the analysis of time-dependent potentials. The matrix operator above turns out to give
	us a clean formalism to obtain decay estimates via matrix resolvents. In Section \ref{sec:localenergy}, we provide a self-contained presentation of the spectral analysis of this matrix operator and prove local energy decay. We believe that this formalism has its independent interest and application
	in other problems. 
	
	

   Thirdly, complicated trajectories will also interact with point spectra of each potentials. If all trajectories are linear, the unstable modes associated with each potential will decay exponentially if we consider solutions in the centre-stable space mentioned above. With general trajectories considered in this paper, one can only measure decay of  unstable modes in the integral sense.  As usual, to stabilize unstable modes, we solve the ODE associated to each mode from $t=\infty$. But this gives additional difficulties to perform bootstrap arguments since the bootstrap assumption are always made for a finite-time interval meanwhile these ODEs need all the information on all the way to $\infty$. Due the appearance of additional interaction terms caused by trajectories and unstable modes, one could not conclude the exponential control as the problem with linear trajectories. Unlike nonlinear problems, here we could not use an iteration scheme to resolve this problem. Since each potential has zero modes, one could not easily construct a sequence of solutions over larger and larger intervals to approximate the original solution like the strategy used in Chen \cite{C1}. To resolve this difficulty, we split the integral appearing in the solution formulas for the unstable modes into two regions and we apply different bootstrap assumptions in each region. For the finite time region, one can use the standard bootstrap estimates. For infinite time region, after controlling the solution all the way to the time $T$ at which we imposed bootstrap assumptions, then we carefully control the energy of the solution starting from  $T$ via the exponential dichotomy.

	 In the core part of the channel analysis,  expanding the solution with respect to the dominated Hamiltonian
	in the corresponding channel, we need to deal with interaction of
	different potentials. To understand the interaction of potentials,
	a very general and refined truncated inhomogeneous is established
	here to handle terms like
	\begin{equation}
	\left\langle x-\beta_{1}t-c_{1}\left(t\right)\right\rangle ^{-\sigma}\int_{0}^{t-M}e^{i\left(t-s\right)\mathcal{D}}\mathrm{U}_{A}\left(t,s\right)\left\langle \cdot-\beta_{2}s-c_{2}\left(s\right)\right\rangle ^{-\alpha}\mathcal{D}^{-1}f\left(s\right)\,ds\label{eq:interaction1}
	\end{equation}
	where $\mathrm{U}_{A}(t,s)$, $c_j(t)$ are caused by the perturbations of trajectories. We refer to Chen \cite{C1}, for the wave equation setting with
	linear trajectories for potentials. Here the analysis to  show these inhomogeneous
	interaction estimates in this general setting is much more involved due
	the complicated trajectories of potentials. With  careful analysis of this interaction term, one can show that the interaction among potentials are small with some coefficient $M^{-\eta}$ for some $\eta>0$. Then with a large $M$, one can absorb interactions to the LHS.
	
	
	After performing the procedures above, for some positive $\sigma, \nu$,  we estimate the local  energy 
	\begin{equation}
	\sum_{j=1}^{J}\left(\int_{0}^{\infty}\left\langle x-y_{j}\left(t\right)\right\rangle ^{-2\sigma}\left|\mathcal{D}^{1-\frac{\nu}{2}}u\left(t,x\right)\right|^{2}+\left\langle x-y_{j}\left(t\right)\right\rangle ^{-2\sigma}\left|\mathcal{D}^{-\frac{\nu}{2}}u_t\left(t,x\right)\right|^{2}\,dxdt\right)^{\frac{1}{2}}
	\end{equation}
	by the initial energy. This local energy decay estimate then implies Strichartz estimates and scattering.
	
	\subsection{Outline of the paper}
This paper is organized as follows: In Section \ref{sec:Prelim}, we recall standard linear estimates including Strichartz estimates and local energy decay etc.  We will sytematically study the spectral theory, resolvent estimates, local energy decay in the matrix formalism in Section \ref{sec:localenergy}. In Section \ref{sec:onepotential}, we revisit the linear theory in Nakanishi-Schlag \cite{NSch} in general frame work of  the matrix formalism without reducing to a static potential. In Section \ref{sec:multi}, we show Stricharz estimates, scattering for the general multiple potential problems. Finally, in Section \ref{sec:waveop}, after care analysis of the scattering map and Strichartz estimates in the central direction, we show the existence of wave operators. 
	
	\subsection{Notations}
	
	\textquotedblleft $A:=B\lyxmathsym{\textquotedblright}$ or $\lyxmathsym{\textquotedblleft}B=:A\lyxmathsym{\textquotedblright}$
	is the definition of $A$ by means of the expression $B$. We use
	the notation $\langle x\rangle=\left(1+|x|^{2}\right)^{\frac{1}{2}}$.
	The bracket $\left\langle \cdot,\cdot\right\rangle $ denotes the
	distributional pairing and the scalar product in the spaces $L^{2}$,
	$L^{2}\times L^{2}$ . For positive quantities $a$ and $b$, we write
	$a\lesssim b$ for $a\leq Cb$ where $C$ is some prescribed constant.
	Also $a\simeq b$ for $a\lesssim b$ and $b\lesssim a$. We denote
	$B_{R}(x)$ the open ball of centered at $x$ with radius $R$ in
	$\mathbb{R}^{3}$. We also denote by $\chi$ a standard $C^{\infty}$
	cut-off function, that is $\chi(x)=1$ for $\left|x\right|\leq1$,
	$\chi(x)=0$ for $\left|x\right|>2$ and $0\leq\chi(x)\leq1$ for
	$1\leq\left|x\right|\leq2$ .

	We always use the bold font to denote a pair of functions as a vector function.  For example, $\hm{u}=\left(\begin{array}{c}
	u_1\\
	u_2
	\end{array}\right)$  and $\hm{f}=\left(\begin{array}{c}
	f_1\\
	f_2
	\end{array}\right)$.     For any space $X$ which measures a scalar function $f$ in the norm $\Vert f\Vert _X$, we use the notation $X_\mathcal{H}$ to measure the corresponding vector function $\hm{f}=\left(\begin{array}{c}
	f_1\\
	f_2
	\end{array}\right)$ by the norm
	\begin{equation}\label{eq:X_H}
	\Vert\hm{f}\Vert_{X_\mathcal{H}}=\Vert (\sqrt{-\Delta+1})f_1\Vert _X+\Vert f_2\Vert _X.
	\end{equation}

	\section{Preliminaries\label{sec:Prelim}}
	In this section, we collect some basic estimates and the Hamiltonian formalism for the Klein-Gordon equation.

	\subsection{Linear estimates in the standard setting}
	
	We first recall some basic linear estimates for Klein-Gordon equations.  Here it is more convenient to formulate estimates in terms of the half-Klein-Gordon evolution.  
	
	First all, we have Strichartz estimates. For the free case, i.e, $V=0$,
	see Keel-Tao \cite{KT}, Ibrahim-Masmoudi-Nakanishi \cite{IMN}. For
	the Klein-Gordon equation with a real-valued static potential, see Chen-Beceanu \cite{BC} and D'Ancona-Fanelli \cite{DaF}.
	\begin{thm}
		\label{thm:StriKGfree}Consider the half Klein-Gordon flow
		\[
		z_{t}=i(\sqrt{H+1})z+F
		\]
		then one has
		\[
		\left\Vert P_{c}z\right\Vert _{S}\lesssim\left\Vert z\left(0\right)\right\Vert _{L^{2}}+\left\Vert F\right\Vert _{S^{*}}
		\]
		where $P_c$ is the projection on the continuous spectrum of $H$ and
		\[
		S=L_{t}^{\infty}L_{x}^{2}\bigcap L_{t}^{2}B_{6,2}^{-5/6}.
		\]
	\end{thm}
	
	Next, we recall the local energy decay for the Klein-Gordon flow.
	For the detailed proof, see Nakanishi-Schlag \cite[Lemma 9.3]{NSch}.
	\begin{lem}
		\label{lem:standlocal}The linear
		evolution $e^{it\sqrt{H+1}}$ satisfies the following local decay
		estimates
		\begin{equation}
		\left\Vert \left\langle x\right\rangle ^{-1}e^{it\sqrt{H+1}}P_{c}\varphi\right\Vert _{L_{t}^{2}L_{x}^{2}}\lesssim\left\Vert \varphi\right\Vert _{L_{x}^{2}},\label{eq:stalocho}
		\end{equation}
		\begin{equation}
		\left\Vert \left\langle x\right\rangle ^{-1}\int_{0}^{t}e^{i(t-s)\sqrt{H+1}}P_{c}f\,ds\right\Vert _{L_{t}^{2}L_{x}^{2}}\lesssim\left\Vert \left\langle x\right\rangle f\right\Vert _{L_{t}^{2}L_{x}^{2}},\label{eq:stalocinh}
		\end{equation}
		where again $P_c$ is the projection on the continuous spectrum of $H$ 
	\end{lem}

	\subsection{Hamiltonian formalism}
	Consider  main equation \eqref{eq:maineq}
	\begin{equation}
	\partial_{t}^{2}u=\Delta u-u-\sum_{j=1}^{J}(V_{j})_{\beta_{j}(t)}(\cdot-y_{j}(t))u+F. 
	\end{equation}
	As introduced before, we can write the equation above as a Hamiltonian dynamical system.
	Denoting\[\boldsymbol{u}\left(t\right)=\left(\begin{array}{c}
	u\\
	u_{t}
	\end{array}\right),\,\boldsymbol{F}=\left(\begin{array}{c}
	0\\
	F
	\end{array}\right)\] we can write
	\begin{equation}
	\partial_{t}\boldsymbol{u}\left(t\right)=\mathcal{J}H\left(t\right)\boldsymbol{u}\left(t\right)+\boldsymbol{F}
	\end{equation}
	where
	\[
	J:=\left(\begin{array}{cc}
	0 & 1\\
	-1 & 0
	\end{array}\right),\,\ \ H\left(t\right):=\left(\begin{array}{cc}
	-\Delta+1+\sum_{j=1}^{J}(V_{j})_{\beta_{j}(t)}(\cdot-y_{j}(t)) & 0\\
	0 & 1
	\end{array}\right).
	\]
	It introduces a continuous forward dynamical system in $\mathcal{H}:=H^{1}\times L^{2}$. The natural symplectic form associated with this Hamiltonian formalism is\begin{equation}\label{eq:sympOmega}
	\omega(\hm{u},\hm{v})=\left\langle J\hm{u},\hm{v}\right\rangle=\int u_2 v_1 -u_1 v_2\,dx.
	\end{equation}
	Throughout, we will use $\omega(\hm{u},\hm{v})$  and $\left\langle J\hm{u},\hm{v}\right\rangle$ interchangeably.

	\section{Local energy decay}\label{sec:localenergy}
	In this section, we analyze the local energy decay for the Klein-Gordon equation.   Our motivation is to under the model problem:
	\begin{equation}
	\partial_{tt}u-\Delta u+u+V_{\beta}\left(x-\beta t\right)u=F\label{eq:scalarinhom}
	\end{equation}
	where $\beta$ denotes a vector such that $\left|\beta\right|<1$.  
	Again, we focus on $\mathbb{R}^{3+1}$
	but the argument here holds for general dimensions.

	
	Using the Hamiltonian formalism and denoting $\hm{u}=\left(\begin{array}{c}
	u\\
	u_{t}
	\end{array}\right)$, one can rewrite the equation \eqref{eq:scalarinhom} in the matrix
	form
	\begin{equation}
	\frac{d}{dt}\hm{u}=\left(\begin{array}{cc}
	0 & 1\\
	\Delta-V_{\beta}\left(x-\beta t\right)-1 & 0
	\end{array}\right)\hm{u}+\hm{F}.\label{eq:-26-1}
	\end{equation}
	Here we can put a general inhomogeneous term $\hm{F}=\left(\begin{array}{c}
	F_{1}\\
	F_{2}
	\end{array}\right)$ which is important for our later applications. (This term coming from the scalar equation should have a special
	form $\hm{F}=\left(\begin{array}{c}
	0\\
	F
	\end{array}\right).$)
	
	Using the moving frame $x-\beta t\rightarrow x$, the equation above
	can be written as
	\begin{align}
		\frac{d}{dt}\hm{u} & =\left(\begin{array}{cc}
			\beta\cdot\nabla & 1\\
			\Delta-V_{\beta}\left(x\right)-1 & \beta\cdot\nabla
		\end{array}\right)\hm{u}+\hm{F}\label{eq:movingframe}\\
		& :=\mathcal{L}_{\beta}\hm{u}+\hm{F}.
	\end{align}
	We believe that this formalism has its independent interest and application
	in other problems. For example, in $1$d, some pointwise decay estimates are obtained in Kopylova \cite{Kop} and applied to the asymptotic stability of moving kinks in the relativistic Ginzburg-Landau equation in Komech-Kopylova \cite{KoKo1}.
	
	\subsection{Spectral theory}
	
	We start with the spectral theory of $\mathcal{L}_{\beta}$. Here
	we record the spectral information for
	\begin{equation}
	\mathcal{L}_{\beta}=\left(\begin{array}{cc}
	\beta\cdot\nabla & 1\\
	\Delta-V_{\beta}\left(x\right)-1 & \beta\cdot\nabla
	\end{array}\right).\label{eq:generalbeta}
	\end{equation}
	Suppose that the scalar Schr\"odinger operator
	\[
	L=-\Delta+1+V
	\]
	has no resonances at $1$. Let $K$ be the number of strictly negative
	eigenvalues of $L$ (counted with multiplicities) and let $M=\dim\ker L$.
	Let $-\nu_{1}^{2}\ldots,-\nu_{K}^{2}$ (with $\nu_{k}>0$) be the
	strictly negative eigenvalues and $\left(\phi_{k}\right)_{k=1,\ldots,K}$
	and $\left(\phi_{m}^{0}\right)_{m=1,\ldots,M}$ be orthonormal (in
	$L^{2}$) families such that
	\[
	L\phi_{k}=-\nu_{k}^{2}\phi_{k},\,\,L\phi_{m}^{0}=0.
	\]
	By the explicit computations in Chen-Jendrej \cite{CJ}, C\^ote-Mu\~noz \cite{CMu} and C\^ote-Martel \cite{CMart},
	we have the following spectral information for $\mathcal{L}_{\beta}$. 
	\begin{lem}\label{lem:lbetaspectral}
		The matrix operator $\mathcal{L}_{\beta}$ is a linear operator
		defined on $L^{2}\times L^{2}$ with domain $H^{2}\times H^{1}$.
		The spectrum of $\mathcal{L}_{\beta}$, $\sigma\left(\mathcal{L}_{\beta}\right)$,
		consists of the continuous spectrum
		\[
		\sigma_{c}\left(\mathcal{L}_{\beta}\right)=i(-\infty,-\frac{1}{\gamma}]\bigcup i[\frac{1}{\gamma},\infty),\quad\gamma=\frac{1}{\sqrt{1-\left|\beta\right|^{2}}}
		\]
		and the point spectrum
		\[
		\sigma_{p}\left(\mathcal{L}_{\beta}\right)=\left\{ 0\right\} \bigcup\left(\bigcup_{k=1}^{K}\left\{ \pm\frac{\nu_{k}}{\gamma}\right\} \right)
		\]
		where $-\nu_{k}^{2}$ are negative eigenvalues of $L$. 
		
		Moreover, one has explicit formulas for (iterated) null components
		\begin{align}\mathcal{Y}_{m,\beta}^{0}(x) & :=(\phi_{m}^{0},-\gamma\beta\cdot\nabla\phi_{m}^{0})_{\beta}(x),\\
			\mathcal{Y}_{m,\beta}^{1}(x) & :=(-(\beta\cdot x)\phi_{m}^{0},\gamma\phi_{m}^{0}+\gamma(\beta\cdot x)(\beta\cdot\nabla\phi_{m}^{0}))_{\beta}(x),
		\end{align}
		and eigenfunctions
		\begin{align*}
			\mathcal{Y}_{k,\beta}^{-}(x) & :=e^{\gamma\nu_{k}\beta\cdot x}(\phi_{k},-\gamma\beta\cdot\nabla\phi_{k}-\gamma\nu_{k}\phi_{k})_{\beta}(x),\\
			\mathcal{Y}_{k,\beta}^{+}(x) & :=e^{-\gamma\nu_{k}\beta\cdot x}(\phi_{k},-\gamma\beta\cdot\nabla\phi_{k}+\gamma\nu_{k}\phi_{k})_{\beta}(x)
		\end{align*}
		associated with $\pm\frac{\nu_{k}}{\gamma}$ respectively.
	\end{lem}
	\begin{proof}
		
		The computations for discrete eigenvalues are straightforward. See
		Chen-Jendrej \cite{CJ}, C\^ote-Mu\~noz \cite{CMu} and C\^ote-Martel \cite{CMart}.  		Since the perturbed operator $\mathcal{L}_{\beta}$ is a compact perturbation
		of the free one. From Lemma \ref{lem:freematrix}, by Weyl's theorem,
		the essential spectrum
		\[
		\sigma_{c}\left(\mathcal{L}_{\beta}\right)=i(-\infty,-\frac{1}{\gamma}]\bigcup i[\frac{1}{\gamma},\infty)
		\]
		is the same as the free operator.
		
	\end{proof}
	With the explicit expressions above, we define the projection with respect to the symplectic form $\omega$, see \eqref{eq:sympOmega}, onto
	the discrete spectrum as
	\begin{align}\label{eq:Pdbeta}
		\mathcal{P}_{d,\beta} & =\mathcal{P}_{+,\beta}+\mathcal{P}_{-,\beta}+\mathcal{P}_{0,\beta}
	\end{align}
	where $\mathcal{P}_{\pm,\beta}$ is the
	projection onto the subspace spanned by $\left\{ \mathcal{Y}_{k,\beta}^{\pm}\right\} _{k=1,\ldots K}$
	and $\mathcal{P}_{0,\beta}$ is the projection
	onto the subspace spanned by $\left\{ \mathcal{Y}_{m,\beta}^{0},\mathcal{Y}_{m,\beta}^{1}\right\} _{m=1,\ldots M}$. 
	
	The projection onto the continuous spectrum is given by
	\begin{equation}
	\mathcal{P}_{c,\beta}:=1-\mathcal{P}_{d,\beta}.\label{eq:Pcbeta}
	\end{equation}

	\subsection{Local energy decay in the Hamiltonian formalism}
	
	We define the weighed Sobolev space for $\in\mathbb{N}$ and $\alpha\in\mathbb{R}$
	\[
	H^{k,\alpha}=\left\{ f\in H^{k}|\,\left\langle x\right\rangle ^{\alpha}\partial^{j}f\in L^{2},\,j=0,\ldots k\right\} 
	\]
	and $H^{0,\alpha}=L^{2,\alpha}.$
	
	Denote the local energy space as
	\begin{equation}
	\mathfrak{H}^{\tau}=H^{1,\tau}\times L^{2,\tau}\label{eq:-26}
	\end{equation}
	and the standard energy space is denoted as $\mathfrak{H}=\mathfrak{H}^{0}.$
	
	The main estimate of this section is the following:
	\begin{thm}
		\label{thm:localenergy}Let $\left|\beta\right|<1$ and $\tau>1$.
		The solution $\hm{u}$ to the equation \eqref{eq:movingframe} above
		with initial data $\hm{u}\left(0\right)=\text{\ensuremath{\hm{u}_{0}} }$satisfies
		\begin{equation}
		\left\Vert \mathcal{P}_{c,\beta}\hm{u}\right\Vert _{L_{t}^{2}\mathfrak{H}^{-\tau}}\lesssim\left\Vert \hm{u}_{0}\right\Vert _{\mathfrak{H}}+\left\Vert \hm{F}\right\Vert _{L^{2}\mathfrak{H}^{\tau}}.\label{eq:localenergythm}
		\end{equation}
	\end{thm}
	
	It is known that to obtain the local energy decay which is equivalent to the Kato smoothing
	estimate for $\mathcal{L}_{\beta}$. To achieve this, we first establish the limiting
	absorption principle for resolvents of $\mathcal{L}_{\beta}$.
	
	\subsection{Resolvents}
	
	In this subsection, we  explicitly compute  the  matrix resolvents.

	For the sake of simplicity, throughout our computations, without loss of generality, we suppose
	that the velocity is along $\vec{e}_{1}$. Our equation is reduced
	to
	\begin{equation}
	\frac{d}{dt}\hm{u}=\left(\begin{array}{cc}
	\beta\partial_{x_{1}} & 1\\
	\Delta-V_{\beta}\left(x\right)-1 & \beta\partial_{x_{1}}
	\end{array}\right)\hm{u}+\hm{F}.\label{eq:-1-2}
	\end{equation}
	and the linear operator is given by
	\[
	\mathcal{L}_{\beta}=\left(\begin{array}{cc}
	\beta\partial_{x_{1}} & 1\\
	\Delta-V_{\beta}\left(x\right)-1 & \beta\partial_{x_{1}}
	\end{array}\right).
	\]

	\subsubsection{Free resolvents}
	
	Starting with the free case
	\begin{equation}
	\mathcal{L}_{0,\beta}:=\left(\begin{array}{cc}
	\beta\partial_{x_{1}} & 1\\
	\Delta-1 & \beta\partial_{x_{1}}
	\end{array}\right)\label{eq:L0beta}
	\end{equation}
	we are interested in the resolvent
	\begin{equation}
	\mathcal{R}_{0,\beta}\left(\lambda\right)=\left(\mathcal{L}_{0,\beta}-i\lambda\right)^{-1}.\label{eq:RLbeta0}
	\end{equation}
	To obtain the explicit formula for the resolvent, taking the Fourier
	transform of the resolvent \eqref{eq:RLbeta0}, one has{\small
		\begin{equation}
		\left(\begin{array}{cc}
		-\left(i\beta k_{1}+i\lambda\right) & 1\\
		-\left(k^{2}+1\right) & -\left(i\beta k_{1}+i\lambda\right)
		\end{array}\right)^{-1}=\left(\left(i\beta k_{1}+i\lambda\right)^{2}+\left(k^{2}+1\right)\right)^{-1}\left(\begin{array}{cc}
		-\left(i\beta k_{1}+i\lambda\right) & -1\\
		\left(k^{2}+1\right) & -\left(i\beta k_{1}+i\lambda\right)
		\end{array}\right).\label{eq:-5-2}
		\end{equation}}
	Therefore, taking the inverse of the matrices above and then performing
	the inverse Fourier transform, we get
	{\small\begin{equation}
		\mathcal{R}_{0,\beta}\left(\lambda\right)=\left(\begin{array}{cc}
		\beta\partial_{x_{1}}-i\lambda & -1\\
		-\Delta+1 & \beta\partial_{x_{1}}-i\lambda
		\end{array}\right)\mathscr{R}_{0,\beta}\left(\lambda\right)=\left(\begin{array}{cc}
		\left(\beta\partial_{x_{1}}-i\lambda\right)\mathscr{R}_{0,\beta}\left(\lambda\right) & -\mathscr{R}_{0,\beta}\left(\lambda\right)\\
		1-\left(\beta\partial_{x_{1}}-i\lambda\right)^{2}\mathscr{R}_{0,\beta}\left(\lambda\right) & \left(\beta\partial_{x_{1}}-i\lambda\right)\mathscr{R}_{0,\beta}\left(\lambda\right)
		\end{array}\right)\label{eq:-6-3}
		\end{equation}}where the resolvent $\mathscr{R}_{0,\beta}\left(\lambda\right)$ is
	given by the inverse Fourier transform of $$\left(\left(i\beta k_{1}+i\lambda\right)^{2}+\left(k^{2}+1\right)\right)^{-1}.$$
	We note that
	\begin{equation}
	\left(\left(i\beta k_{1}+i\lambda\right)^{2}+\left(k^{2}+1\right)\right)=\left(\left(1-\beta^{2}\right)k_{1}^{2}+k_{2}^{2}+k_{3}^{2}-\lambda^{2}+1-2\beta k_{1}\lambda\right).\label{eq:-7-1}
	\end{equation}
	Hence in the physical space,
	\begin{equation}
	\left(\left(1-\beta^{2}\right)k_{1}^{2}+k_{2}^{2}+k_{3}^{2}-\lambda^{2}+1-2\beta k_{1}\lambda\right)\check{}=-\left(1-\beta^{2}\right)\partial_{x_{1}}^{2}-\tilde{\Delta}+1-\lambda^{2}-i2\beta\lambda\partial_{x_{1}}\label{eq:-8-2}
	\end{equation}
	where
	$
	\tilde{\Delta}=\partial_{2}^{2}+\partial_{3}^{2}.	$
	
	Setting
	\begin{equation}
	H_{0,\beta}=-\left(1-\beta^{2}\right)\partial_{x_{1}}^{2}-\tilde{\Delta},\label{eq:H0beta}
	\end{equation}
	for given any smooth function $\phi\left(x\right)$, one observes
	that
	\begin{equation}
	\left(H_{0,\beta}-1+\lambda^{2}-i2\beta\lambda\partial_{x_{1}}\right)\phi\left(x\right)=e^{-i\gamma^{2}\beta\lambda x_{1}}\left(H_{0,\beta}-\left(\gamma^{2}\lambda^{2}-1\right)\right)e^{i\gamma^{2}\beta\lambda x_{1}}\phi\left(x\right).\label{eq:-10-2}
	\end{equation}
	Therefore, the resolvent is given by
	\begin{equation}
	\mathscr{R}_{0,\beta}\left(\lambda\right)=e^{-i\gamma^{2}\beta\lambda x_{1}}\left(H_{0,\beta}-\left(\gamma^{2}\lambda^{2}-1\right)\right)^{-1}e^{i\gamma^{2}\beta\lambda x_{1}}.\label{eq:reolventbeta0}
	\end{equation}
	It remains to compute the resolvent
	\begin{equation}
	R_{0,\beta}\left(\mu^{2}\right)=\left(H_{0,\beta}-\mu^{2}\right)^{-1}.\label{eq:-12-2}
	\end{equation}
	Denoting $\left|x\right|_{\beta}=\sqrt{\gamma^{2}x_{1}^{2}+x_{2}^{2}+x_{3}^{2}},$
	by a direct and explicit computation, in terms of the integral kernel,
	one has
	\begin{equation}
	R_{0,\beta}\left(\mu^{2}\right)=\left(\mathcal{H}_{0,\beta}-\mu^{2}\right)^{-1}=\gamma\frac{e^{i\mu\left|x-y\right|_{\beta}}}{4\pi\left|x-y\right|_{\beta}}.\label{eq:-14-2}
	\end{equation}
	From the explicit computations above, we conclude the spectral information
	for $\mathcal{L}_{0,\beta}$.
	\begin{lem}
		\label{lem:freematrix}The matrix operator $\mathcal{L}_{0,\beta}$
		is a linear operator defined on $L^{2}\times L^{2}$ with domain
		$H^{2}\times H^{1}$. The spectrum of $\mathcal{L}_{0,\beta}$, $\sigma\left(\mathcal{L}_{0,\beta}\right)$,
		only consists of the continuous spectrum
		\[
		\sigma_{c}\left(\mathcal{L}_{0,\beta}\right)=i(-\infty,-\frac{1}{\gamma}]\bigcup i[\frac{1}{\gamma},\infty).
		\]
		For $\lambda\in(-\infty,-\frac{1}{\gamma}]\bigcup[\frac{1}{\gamma},\infty)$,
		the resolvent is given by
		\begin{equation}
		\mathcal{R}_{0,\beta}\left(\lambda\right)=\left(\mathcal{L}_{0,\beta}-i\lambda\right)^{-1}=\left(\begin{array}{cc}
		\beta\partial_{x_{1}}-i\lambda & -1\\
		-\Delta+1 & \beta\partial_{x_{1}}-i\lambda
		\end{array}\right)\mathscr{R}_{0,\beta}\left(\lambda\right)\label{eq:freematrixresolvent}
		\end{equation}
		where
		\begin{equation}
		\mathscr{R}_{0,\beta}\left(\lambda\right)=e^{-i\gamma^{2}\beta\lambda x_{1}}\left(H_{0,\beta}-\left(\gamma^{2}\lambda^{2}-1\right)\right)^{-1}e^{i\gamma^{2}\beta\lambda x_{1}}.\label{eq:R0beteH0beta}
		\end{equation}
		and in terms of the integral kernel,
		\begin{equation}
		\left(H_{0,\beta}-\left(\gamma^{2}\lambda^{2}-1\right)\right)^{-1}=\gamma\frac{e^{i\sqrt{\gamma^{2}\lambda^{2}-1}\left|x-y\right|_{\beta}}}{4\pi\left|x-y\right|_{\beta}}\label{eq:freeresolbeta0}
		\end{equation}
		with $\left|x\right|_{\beta}=\sqrt{\gamma^{2}x_{1}^{2}+x_{2}^{2}+x_{3}^{2}}$.
	\end{lem}

	\subsubsection{Perturbed resolvent}
	Consider the perturbed matrix operator
	\[
	\mathcal{L}_{\beta}=\left(\begin{array}{cc}
	\beta\partial_{x_{1}} & 1\\
	\Delta-V_{\beta}\left(x\right)-1 & \beta\partial_{x_{1}}
	\end{array}\right).
	\]
	One can write
	\[
	\mathcal{L}_{\beta}=\mathcal{L}_{0,\beta}+\mathcal{V}_{\beta}
	\]
	where
	\[
	\mathcal{V}_{\beta}=\left(\begin{array}{cc}
	0 & 0\\
	-V_{\beta} & 0
	\end{array}\right).
	\]
	To compute the resolvent,  we use the general
	resolvent formula 
	\begin{equation}
	\mathcal{R}_{\beta}\left(\lambda\right)=\left(\mathcal{L}_{\beta}-i\lambda\right)^{-1}=\left(1+\mathcal{R}_{0,\beta}\left(\lambda\right)\mathcal{V}_{\beta}\right)^{-1}\mathcal{R}_{0,\beta}\left(\lambda\right).\label{eq:-28-1}
	\end{equation}
	Recalling the resolvent for the free operator, \eqref{eq:freematrixresolvent},
	direct computations give
	\begin{align}\mathcal{R}_{0,\beta}\left(\lambda\right)\mathcal{V}_{\beta}=
		\left(\begin{array}{cc}
			\left(\beta\partial_{x_{1}}-i\lambda\right)\mathscr{R}_{0,\beta}\left(\lambda\right) & -\mathscr{R}_{0,\beta}\left(\lambda\right)\\
			1-\left(\beta\partial_{x_{1}}-i\lambda\right)^{2}\mathscr{R}_{0,\beta}\left(\lambda\right) & \left(\beta\partial_{x_{1}}-i\lambda\right)\mathscr{R}_{0,\beta}\left(\lambda\right)
		\end{array}\right)\left(\begin{array}{cc}
			0 & 0\\
			-V_{\beta} & 0
		\end{array}\right)\nonumber \\
		=\left(\begin{array}{cc}
			\mathscr{R}_{0,\beta}\left(\lambda\right)V_{\beta} & 0\\
			-\left(\beta\partial_{x_{1}}-i\lambda\right)\mathscr{R}_{0,\beta}\left(\lambda\right)V_{\beta} & 0
		\end{array}\right)
	\end{align}
	and then
	\begin{equation}
	1+\mathcal{R}_{0,\beta}\left(\lambda\right)\mathcal{V}_\beta=\left(\begin{array}{cc}
	1+\mathscr{R}_{0,\beta}\left(\lambda\right)V_{\beta} & 0\\
	-\left(\beta\partial_{x_{1}}-i\lambda\right)\mathscr{R}_{0,\beta}\left(\lambda\right)V_{\beta} & 1
	\end{array}\right).\label{eq:-32-1}
	\end{equation}
	Therefore, taking the inverse, one has{\footnotesize
		\begin{equation}
		\left(1+\mathcal{R}_{0,\beta}\left(\lambda\right)\mathcal{V}_\beta\right)^{-1}=\left(\begin{array}{cc}
		1+\mathscr{R}_{0,\beta}\left(\lambda\right)V_{\beta} & 0\\
		-\left(\beta\partial_{x_{1}}-i\lambda\right)\mathscr{R}_{0,\beta}\left(\lambda\right)V_{\beta} & 1
		\end{array}\right)^{-1}=\left(\begin{array}{cc}
		\left(1+\mathscr{R}_{0,\beta}\left(\lambda\right)V_{\beta}\right)^{-1} & 0\\
		\left(\beta\partial_{x_{1}}-i\lambda\right)\left(1-\left(1+\mathscr{R}_{0,\beta}\left(\lambda\right)V_{\beta}\right)^{-1}\right) & 1
		\end{array}\right).\label{eq:-33-1}
		\end{equation}}Recalling the resolvent \eqref{eq:reolventbeta0} in the free case, by the general resolvent formula, we write the perturbed resolvent as 
	\begin{equation}
	\mathscr{R}_{\beta}\left(\lambda\right)=\left(H_{\beta}+1-\lambda^{2}-i2\beta\lambda\partial_{x_{1}}\right)^{-1}=\left(1+\mathscr{R}_{0,\beta}\left(\lambda\right)V_{\beta}\right)^{-1}\mathscr{R}_{0,\beta}\left(\lambda\right).\label{eq:-35-1}
	\end{equation}
	Denote
	\begin{equation}
	H_{\beta}=-\left(1-\beta^{2}\right)\partial_{x_{1}}^{2}-\tilde{\Delta}+V_{\beta,}\label{eq:Hbeta}
	\end{equation}
	and set resolvent $R_{\beta}$ is given as the resolvent for \eqref{eq:Hbeta}.
	
	With notations above,	then we have{\footnotesize
		\begin{align}
			\mathcal{R}_{\beta}\left(\lambda\right) & =\left(\begin{array}{cc}
				\left(1+\mathscr{R}_{0,\beta}\left(\lambda\right)V_{\beta}\right)^{-1} & 0\\
				\left(\beta\partial_{x_{1}}-i\lambda\right)\left(1-\left(1+\mathscr{R}_{0,\beta}\left(\lambda\right)V_{\beta}\right)^{-1}\right) & 1
			\end{array}\right)\left(\begin{array}{cc}
				\left(\beta\partial_{x_{1}}-i\lambda\right)\mathscr{R}_{0,\beta}\left(\lambda\right) & -\mathscr{R}_{0,\beta}\left(\lambda\right)\\
				1-\left(\beta\partial_{x_{1}}-i\lambda\right)^{2}\mathscr{R}_{0,\beta}\left(\lambda\right) & \left(\beta\partial_{x_{1}}-i\lambda\right)\mathscr{R}_{0,\beta}\left(\lambda\right)
			\end{array}\right)\nonumber \\
			& =\left(\begin{array}{cc}
				\mathscr{R}_{\beta}\left(\lambda\right)\left(\beta\partial_{x_{1}}-i\lambda\right) & -\mathscr{R}_{\beta}\left(\lambda\right)\\
				1-\left(\beta\partial_{x_{1}}-i\lambda\right)\mathscr{R}_{\beta}\left(\lambda\right)\left(\beta\partial_{x_{1}}-i\lambda\right) & \left(\beta\partial_{x_{1}}-i\lambda\right)\mathscr{R}_{\beta}\left(\lambda\right)
			\end{array}\right)\label{eq:perturbedmatrixresol}
	\end{align}}
	where as in the free case
	\begin{equation}
	\mathscr{R}_{\beta}\left(\lambda\right)=e^{-i\gamma^{2}\beta\lambda x_{1}}R_{\beta}\left(\lambda^{2}\gamma^{2}-1\right)e^{i\gamma^{2}\beta\lambda x_{1}}.\label{eq:relationmatrixscalar}
	\end{equation}
	The same as in the free problem, the resolvent $R_{\beta}\left(\mu^{2}\right)$
	can be related to the resolvent for $\beta=0$, $R_{0}\left(\mu^{2}\right)=\left(-\Delta+V-\mu^{2}\right)^{-1}$,
	by
	\begin{equation}
	R_{\beta}\left(\mu^{2};x,y\right)=\left(H_{\beta}-\mu^{2}\right)^{-1}=\gamma R_{0}\left(\mu^{2};\gamma x_{1},\tilde{x},\gamma y_{1},\tilde{y}\right).\label{eq:relationbeta0}
	\end{equation}
	in terms of the integral kernel where $\tilde{x}$ denotes $x_{2},x_{3}$
	variables.
	
	\subsection{Limiting absorption principle}
	With explicit computations above, now we can translate the limiting absorption principle from the scalar resolvents to matrix resolvents.
	\subsubsection{Free  resolvent with $\beta=0$}

	
	In this case, the matrix resolvent is given as
	\begin{equation}
	\mathcal{R}_{0,0}(\lambda)=\left(\mathcal{L}_{0,0}-i\lambda\right)^{-1}=\left(\begin{array}{cc}
	-i\lambda R_{0,0}\left(\lambda^{2}-1\right) & -R_{0,0}\left(\lambda^{2}-1\right)\\
	\left(1+\lambda^{2}R_{0,0}\left(\lambda^{2}-1\right)\right)-i & \lambda R_{0,0}\left(\lambda^{2}-1\right)
	\end{array}\right).\label{eq:-17-1}
	\end{equation}
	Recall that the integral kernel for the standard scalar Schr\"odinger
	resolvent
	\begin{equation}
	R_{0,0}\left(\mu^{2},x,y\right)=\left(-\Delta-\mu^{2}\right)^{-1}=\frac{e^{i\mu\left|x-y\right|}}{4\pi\left|x-y\right|}.\label{eq:-18-2}
	\end{equation}
	We also recall the standard resolvents estimates, see Agmon \cite{Agm} and Komech-Kopylova \cite{KoKo}.
	\begin{lem}
		\label{lem:limitingscalarfree}Given notations above, one has with
		$\tau>\frac{1}{2}$ and any $r>0$
		\begin{equation}
		\left\Vert R_{0,0}\left(\mu^{2}\right)\right\Vert _{\mathcal{L}\left(H^{m,\tau},H^{m+\ell,-\tau}\right)}\leq C\left(r\right)\left|\mu\right|^{-\left(1-\ell\right)},\mu\in\mathbb{C\backslash\mathbb{R}}\,\,\left|\mu\right|\geq r\label{eq:freeresollarge}
		\end{equation}
		for $m=0,1$, $\ell=-1,0,1$ such that $m+\ell\in\left\{ 0,1\right\} $. 
		
		We also have with $\tau>1$:
		\begin{equation}
		\left\Vert R_{0,0}\left(\mu^{2}\right)\right\Vert _{\mathcal{L}\left(H^{m,\tau},H^{m+\ell,-\tau}\right)}\leq C\left(r\right)\left\langle 1+\left|\mu\right|\right\rangle ^{-\frac{1-\ell}{2}},\,\mu\in\mathbb{C}\backslash\mathbb{R}\label{eq:freeresolsmall}
		\end{equation}
		for $m=0,1$, $\ell=-1,0,1$ such that $m+\ell\in\left\{ 0,1\right\} $. 
	\end{lem}
	
	With the mapping properties of the standard resolvents, we can conclude
	the limiting absorption for the matrix resolvent.
	\begin{lem}
		\label{lem:standardfreelocal}For $\tau>1$, we have
		\begin{equation}
		\sup_{i\lambda\notin\sigma\left(\mathcal{L}_{0,0}\right)}\left\Vert \mathcal{R}_{0,0}\left(\lambda\right)\hm{\Psi}\right\Vert _{\mathfrak{H}^{-\tau}}\lesssim\left\Vert \hm{\Psi}\right\Vert _{\mathfrak{H}^{\tau}}.\label{eq:-27}
		\end{equation}
	\end{lem}
	
	\begin{proof}
		Explicitly, we have
		\begin{equation}
		\mathcal{R}_{0,0}\left(\lambda\right)\hm{\Psi}=\left(\begin{array}{cc}
		-i\lambda R_{0,0}\left(\lambda^{2}-1\right) & -R_{0,0}\left(\lambda^{2}-1\right)\\
		\left(1+\lambda^{2}R_{0,0}\left(\lambda^{2}-1\right)\right) & \lambda R_{0,0}\left(\lambda^{2}-1\right)
		\end{array}\right)\left(\begin{array}{c}
		\psi_{1}\\
		\psi_{2}
		\end{array}\right).\label{eq:-21-2}
		\end{equation}
		We apply the estimates from Lemma \ref{lem:limitingscalarfree} above to each entry of the matrix resolvent. 
		From estimate \eqref{eq:freeresolsmall}, 
		for $\tau>1$ and $i\lambda\notin \sigma(\mathcal{L}_{0,0})$, 
		we know
		\begin{equation}
		\left\Vert \lambda R_{0,0}\left(\lambda^{2}-1\right)\psi_{1}\right\Vert _{H^{1,-\tau}}\lesssim\left\Vert \psi_{1}\right\Vert _{H^{1,\tau}}\label{eq:-22-2}
		\end{equation}
		\begin{equation}
		\left\Vert R_{0,0}\left(\lambda^{2}-1\right)\psi_{2}\right\Vert _{H^{1,-\tau}}\lesssim\left\Vert \psi_{2}\right\Vert _{L^{2,\tau}}\label{eq:-23-2}
		\end{equation}
		\begin{equation}
		\left\Vert	\left ( 1+\lambda^{2}R_{0,0}\left(\lambda^{2}-1\right)	\right	)\psi_{1}\right\Vert _{L^{2,-\tau}}\lesssim\left\Vert \psi_{1}\right\Vert _{H^{1,\tau}}\label{eq:-24-2}
		\end{equation}
		and
		\begin{equation}
		\left\Vert \lambda R_{0,0}\left(\lambda^{2}-1\right)\psi_{2}\right\Vert _{H^{0,-\tau}}\lesssim\left\Vert \psi_{2}\right\Vert _{L^{2,\tau}}.\label{eq:-25-1}
		\end{equation}
		The desired results follow.
	\end{proof}
	
	\subsubsection{Free resolvent with $\beta\protect\neq0$}
	
	Recall that  the formula for the free resolvent is given by \eqref{eq:freematrixresolvent}.
	From the explicit formula \eqref{eq:freeresolbeta0}, we simply observe
	that the estimates from Lemma \ref{lem:limitingscalarfree} trivially
	hold for $R_{0,\beta}(\mu^2)$ since one just needs to
	rescale $x_{1}$ variable on both sides. With the explicit formula
	\eqref{eq:R0beteH0beta}, one can directly extend these estimates to
	$\mathscr{R}_{0,\beta}$.
	\begin{lem}
		\label{lem:limitingscalarfreebeta}Given notations above, one has
		\begin{equation}
		\left\Vert \mathscr{R}_{0,\beta}\left(\lambda\right)\right\Vert _{\mathcal{L}\left(H^{m,\tau},H^{m+\ell.-\tau}\right)}\leq C\left(r\right)\left|\lambda\right|^{-\left(1-\ell\right)},\lambda\in\mathbb{C\backslash\mathbb{R}}\,\,\left|\gamma\lambda\right|\geq r\label{eq:freeresollarge-1}
		\end{equation}
		for $m=0,1$, $\ell=-1,0,1$ with $\tau>\frac{1}{2}$ and any $r>0$. 
		
		We also have
		\begin{equation}
		\left\Vert \mathscr{R}_{0,\beta}\left(\lambda\right)\right\Vert _{\mathcal{L}\left(H^{m,\tau},H^{m+\ell,-\tau}\right)}\leq C\left(r\right)\left\langle 1+\left|\lambda\right|\right\rangle ^{-\frac{1-\ell}{2}},\,\lambda\in\mathbb{C}\backslash\mathbb{R}\label{eq:freeresolsmall-1}
		\end{equation}
		for $m=0,1$, $\ell=-1,0,1$ with $\tau>1$.
	\end{lem}
	
	\begin{proof}
		The results follow from direct inspections.
	\end{proof}
	Finally, we can conclude the mapping properties of the  matrix
	resolvent.
	\begin{lem}\label{lem:freematrixbeta}
		For $\tau>1$ and $\left|\beta\right|<1$, we have
		\begin{equation}
		\sup_{i\lambda\notin\sigma\left(\mathcal{L}_{0,\beta}\right)}\left\Vert \mathcal{R}_{0,\beta}\left(\lambda\right)\hm{\Psi}\right\Vert _{\mathfrak{H}^{-\tau}}\lesssim\left\Vert \hm{\Psi}\right\Vert _{\mathfrak{H}^{\tau}}.\label{eq:-27-1}
		\end{equation}
	\end{lem}
	
	\begin{proof}
		Again for the matrix problem, we write it down explicitly
		\[
		\mathcal{R}_{0,\beta}\left(\lambda\right)\hm{\Psi}=\left(\begin{array}{cc}
		\left(\beta\partial_{x_{1}}-i\lambda\right)\mathscr{R}_{0,\beta}\left(\lambda\right) & -\mathscr{R}_{0,\beta}\left(\lambda\right)\\
		1-\left(\beta\partial_{x_{1}}-i\lambda\right)\mathscr{R}_{0,\beta}\left(\lambda\right)\left(\beta\partial_{x_{1}}-i\lambda\right) & \left(\beta\partial_{x_{1}}-i\lambda\right)\mathscr{R}_{0,\beta}\left(\lambda\right)
		\end{array}\right)\left(\begin{array}{c}
		\psi_{1}\\
		\psi_{2}
		\end{array}\right)
		\]
		and then check the mapping properties for each entry. Explicitly,
		we need to estimate
		\begin{equation}
		\mathscr{R}_{0,\beta}\left(\lambda\right)\left(\beta\partial_{x_{1}}-i\lambda\right)\psi_{1},\quad
		\mathscr{R}_{0,\beta}\left(\lambda\right)\psi_{2}\label{eq:freebeta2}
		\end{equation}
		and
		\begin{equation}
		\left(\beta\partial_{x_{1}}-i\lambda\right)\mathscr{R}_{0,\beta}\left(\lambda\right)\left(\beta\partial_{x_{1}}-i\lambda\right)\psi_{1},\qquad
		\left(\beta\partial_{x_{1}}-i\lambda\right)\mathscr{R}_{0,\beta}\left(\lambda\right)\psi_{2}.\label{eq:freebeta4}
		\end{equation}
		To obtain the limiting absorption, we need to bound the first two
		expressions in $H^{1,-\tau}$ and the last two in $L^{2,-\tau}$.
		For $|\lambda|$ small, the desired results holds trivially. We check
		the estimates for $|\lambda|$ large. 
		
		For the first one, we separate into two pieces. For $\mathscr{R}_{0,\beta}\left(\lambda\right)\lambda\psi_{1}$,
		we use the bound of $\mathscr{R}_{0,\beta}\left(\lambda\right)$ from
		$H^{1,\tau}$ to $H^{1,-\tau}$, \eqref{eq:freeresollarge-1}, which
		has the decay rate $\frac{1}{\lambda}$ to cancel out $\lambda$ here.
		So we get
		\begin{equation}
		\left\Vert \mathscr{R}_{0,\beta}\left(\lambda\right)\lambda\psi_{1}\right\Vert _{H^{1,-\tau}}\lesssim\left\Vert \psi_{1}\right\Vert _{H^{1,\tau}}.\label{eq:freeebeta11}
		\end{equation}
		For the other piece, $\mathscr{R}_{0,\beta}\left(\lambda\right)\left(\beta\partial_{x_{1}}\right)\psi_{1}$,
		notice that $\left(\beta\partial_{x_{1}}\right)\psi_{1}\in L^{2,\tau}$,
		we need to use the bound of $\mathscr{R}_{0,\beta}\left(\lambda\right)$
		from $L^{2,\tau}$ to $H^{1,-\tau}$, \eqref{eq:freeresollarge-1} with
		$\ell=1$, which is uniform in $\lambda$. Therefore combining \eqref{eq:freeebeta11},
		one has
		\[
		\left\Vert \mathscr{R}_{0,\beta}\left(\lambda\right)\left(\beta\partial_{x_{1}}-i\lambda\right)\psi_{1}\right\Vert _{H^{1,-\tau}}\lesssim\left\Vert \psi_{1}\right\Vert _{H^{1,\tau}}.
		\]
		Similar computations hold for for other expressions. To bound $\mathscr{R}_{0,\beta}\left(\lambda\right)\psi_{2}$
		in $H^{1,-\tau}$, we need to use the bound of $\mathscr{R}_{0,\beta}\left(\lambda\right)$
		from $L^{2,\tau}$ to $H^{1,-\tau}$ which is uniform in $\lambda$.
		So it follows
		\begin{equation}
		\left\Vert \mathscr{R}_{0,\beta}\left(\lambda\right)\psi_{2}\right\Vert _{H^{1,-\tau}}\lesssim\left\Vert \psi_{2}\right\Vert _{L^{2,\tau}}.\label{eq:freeebeta11-1}
		\end{equation}
		To bound $\left(\beta\partial_{x_{1}}\right)\mathscr{R}_{0,\beta}\left(\lambda\right)\psi_{2}$
		in $L^{2,-\tau}$, again we need to use the bound of $\mathscr{R}_{0,\beta}\left(\lambda\right)$
		from $L^{2,\tau}$ to $H^{1,-\tau}$ which is uniform in $\lambda$.
		Then to bound $\lambda R_{0,\beta}\left(\lambda\right)\psi_{2}$ in
		$L^{2,-\tau}$, we use bound of $\mathscr{R}_{0,\beta}\left(\lambda\right)$
		from $L^{2,\tau}$ to $L^{2,-\tau}$ which has $\frac{1}{\lambda}$
		decay. Therefore, we obtain
		\[
		\left\Vert \left(\beta\partial_{x_{1}}-\lambda\right)\mathscr{R}_{0,\beta}\left(\lambda\right)\psi_{2}\right\Vert _{L^{2,\tau}}\lesssim\left\Vert \psi_{2}\right\Vert _{L_{\tau}^{2,\tau}}.
		\]
		Finally, we analyze $\left(\beta\partial_{x_{1}}-i\lambda\right)\mathscr{R}_{0,\beta}\left(\lambda\right)\left(\beta\partial_{x_{1}}-i\lambda\right)\psi_{1}$.
		Expanding everything, many pieces can be estimated similarly to arguments
		above. There are two terms different from cases above. The first one
		is to estimate
		\[
		\lambda^{2}\mathscr{R}_{0,\beta}\left(\lambda\right)\psi_{1}
		\]
		in $L^{2,-\tau}$. From the resolvent bound for $\mathscr{R}_{0,\beta}\left(\lambda\right)$
		from $H^{1,\tau}$ to $L^{2,-\tau}$, \eqref{eq:freeresollarge-1} with $\ell=-1$, it has $\frac{1}{\lambda^{2}}$
		decay which cancels out with the $\lambda^{2}$ factor. Secondly, to
		bound
		\[
		\left(\beta\partial_{x_{1}}\right)\mathscr{R}_{0,\beta}\left(\lambda\right)\left(\partial_{x_{1}}\right)\psi_{1}
		\]
		in $L^{2,-\tau}$, we use the resolvent bound from $L^{2,\tau}$ to
		$H^{1,-\tau}$ which is uniform. Hence, one has
		\[
		\left\Vert \left(\beta\partial_{x_{1}}-i\lambda\right)\mathscr{R}_{0,\beta}\left(\lambda\right)\left(\beta\partial_{x_{1}}-i\lambda\right)\psi_{1}\right\Vert _{L^{2,-\tau}}\lesssim\left\Vert \psi_{1}\right\Vert _{H^{1,\tau}}.
		\]
		Putting computations above together, the desired results follow for
		the matrix problem.
	\end{proof}
	
	\subsubsection{Perturbed resolvent}
	
	With the understanding of the free problem, for the perturbed problem,
	we work on the general case directly. Recall that the matrix resolvent
	for the perturbed operator is given by \eqref{eq:perturbedmatrixresol}.
	Note that by our assumption, the potential is generic. Using notations
	above, we have the following resolvent estimates. Again for the proof, see Agmon \cite{Agm} and Komech-Kopylova \cite{KoKo}.
	\begin{lem}\label{lem:perresoldecay}
		Given notations above, by \eqref{eq:relationbeta0} and the standard
		resolvent estimates, one has
		\begin{equation}
		\left\Vert R_{\beta}\left(\mu^{2}\right)\right\Vert _{\mathcal{L}\left(H^{m,\tau},H^{m+\ell,-\tau}\right)}\leq C\left(r\right)\left|\mu\right|^{-\left(1-\ell\right)},\mu\in\mathbb{C\backslash\mathbb{R}}\,\,\left|\mu\right|\geq r\label{eq:-19-1-1}
		\end{equation}
		for $m=0,1$, $\ell=-1,0,1$ with $\tau>\frac{1}{2}$ with $m+\ell\in\left\{ 0,1\right\} $.
		For small $|\mu|$, these are uniformly bounded for $\tau>1$.
		
		Moreover, by the relation \eqref{eq:relationmatrixscalar}, we also
		have
		\begin{equation}
		\left\Vert \mathscr{R}_{\beta}\left(\lambda\right)\right\Vert _{\mathcal{L}\left(H^{m,\tau},H^{m+\ell,-\tau}\right)}\leq C\left(r\right)\left|\lambda\right|^{-\left(1-\ell\right)},\,\lambda\in\mathbb{C\backslash\mathbb{R}}\,\,\left|\lambda\right|\geq r\label{eq:-19-1}
		\end{equation}
		for $m=0,1$, $\ell=-1,0,1$ with $\tau>\frac{1}{2}$ with $m+\ell\in\left\{ 0,1\right\} $.
		For small $|\lambda|$, these are uniformly bounded for $\tau>1$.
	\end{lem}
	
	So as before we should conclude that
	\begin{lem}
		\label{lem:matrixresolbeta}For $\left|\beta\right|<1$ and $\tau>1$,
		for the perturbed resolvent, one has
		\begin{equation}
		\sup_{i\lambda\notin\sigma\left(\mathcal{L}_{\beta}\right)}\left\Vert \mathcal{R}_{\beta}\left(\lambda\right)\hm{\Psi}\right\Vert _{\mathfrak{H}^{-\tau}}\lesssim\left\Vert \hm{\Psi}\right\Vert _{\mathfrak{H}^{\tau}}.\label{eq:-40}
		\end{equation}
	\end{lem}
	
	\begin{proof}
		Note that when
		\begin{equation}
		\mathcal{R}_{\beta}\left(\lambda\right)=\left(\begin{array}{cc}
		\mathscr{R}_{\beta}\left(\lambda\right)\left(\beta\partial_{x_{1}}-i\lambda\right) & -\mathscr{R}_{\beta}\left(\lambda\right)\\
		1-\left(\beta\partial_{x_{1}}-i\lambda\right)\mathscr{R}_{\beta}\left(\lambda\right)\left(\beta\partial_{x_{1}}-i\lambda\right) & \left(\beta\partial_{x_{1}}-i\lambda\right)\mathscr{R}_{\beta}\left(\lambda\right)
		\end{array}\right)\label{eq:-41}
		\end{equation}
		acts on a vector $\hm{\Psi}\in H^{1}\times L^{2}=\mathcal{H}$, we have the following
		\begin{equation}
		\mathscr{R}_{\beta}\left(\lambda\right)\left(\beta\partial_{x_{1}}-i\lambda\right)\psi_{1},\,\mathscr{R}_{\beta}\left(\lambda\right)\psi_{2}\label{eq:-42}
		\end{equation}
		\begin{equation}
		\left(\beta\partial_{x_{1}}-\lambda\right)\mathscr{R}_{\beta}\left(\lambda\right)\left(\beta\partial_{x_{1}}-\lambda\right)\psi_{1},\,\left(\beta\partial_{x_{1}}-i\lambda\right)\mathscr{R}_{\beta}\left(\lambda\right)\psi_{2}.\label{eq:-44}
		\end{equation}
		To obtain the limiting absorption, we need to bound the first two
		expressions in $H^{1,\tau}$ and the last two in $L^{2,\tau}$. The
		analysis is the same as the free problem in Lemma \ref{lem:freematrixbeta}.
	\end{proof}
	
	\subsection{Proof of local energy decay estimates}
	
	Finally, with resolvent estimates above, we show the local decay estimate
	for the system
	\begin{equation}
	\frac{d}{dt}\hm{u}=\left(\begin{array}{cc}
	\beta\partial_{x_{1}} & 1\\
	\Delta-V_{\beta}\left(x\right)-1 & \beta\partial_{x_{1}}
	\end{array}\right)\hm{u}+\hm{F}.\label{eq:-1-1-4}
	\end{equation}
	Formally, by Duhamel formula, the solution can be written as
	\begin{equation}
	\hm{u}(t)=e^{\mathcal{L}_{t\beta}}\hm{u}_{0}+\int_{0}^{t}e^{\left(t-s\right)\mathcal{L}_{\beta}}\hm{F}\left(s\right)\,ds.\label{eq:-2-1}
	\end{equation}
	More precisely, we have the following lemma on the evolution operator $e^{t\mathcal{L}_{\beta}}$.
	\begin{lem}Under our setting, the evolution operator can be written as
		\begin{align*}
			e^{t\mathcal{L}_{\beta}} & =\frac{1}{2\pi}\int_{|\lambda|\geq\frac{1}{\gamma}}e^{i\lambda t}\left(\mathcal{R}_{\beta}\left(\lambda+i0\right)-\mathcal{R}_{\beta}\left(\lambda-i0\right)\right)\,d\lambda+e^{\mathcal{L}_{t\beta}}\mathcal{P}_{d,\beta}
		\end{align*}
		where $\mathcal{P}_{\beta,d}$ is defined as and the formula \eqref{eq:Pdbeta} above and the
		convergence of the integral are understood in the following weak sense:
		if $\hm{\phi}$ and $\hm{\psi}$ are in $\mathfrak{H}^{\tau}\bigcap\left(H^{2}(\mathbb{R}^{3})\times H^{1}(\mathbb{R}^{3})\right)$
		with $\tau>\frac{1}{2}$, then
		\begin{align*}
			\left\langle e^{t\mathcal{L}_{\beta}}\hm{\phi},\hm{\psi}\right\rangle  & =\lim_{R\rightarrow\infty}\frac{1}{2\pi}\int_{R\geq|\lambda|\geq\frac{1}{\gamma}}e^{i\lambda t}\left\langle \left(\mathcal{R}_{\beta}\left(\lambda+i0\right)-\left(\lambda-i0\right)\right)\hm{\phi},\hm{\psi}\right\rangle d\lambda +\left\langle e^{\mathcal{L}_{t\beta}}\mathcal{P}_{d}\hm{\phi},\hm{\psi}\right\rangle 
		\end{align*}
		for all $t$ and the integrand is well-defined by the limiting absorption
		principle.
	\end{lem}
	
	\begin{proof}
		Given the decay of resolvents $\left\Vert \mathcal{R}_{\beta}\left(\lambda\right)\right\Vert_{\mathcal{L}(\mathfrak{H},\mathfrak{H})}\lesssim \frac{1}{|\lambda|}$, see Lemma  \ref{lem:perresoldecay}, the evolution operator $e^{t\mathcal{L}_{\beta}}$ can be defined via the Hille-Yosida theorem.  The explicit expression of it can be obtained by a a standard contour deformation argument, see Erdogan-Schlag \cite[Lemma 12]{ESch} for full details. Although in  Erdogan-Schlag \cite{ESch}, the authors dealed with non-selfadjoint matrix Schr\"odinger operators, the argument there is general enough which can be applied to our problem identically. 

	\end{proof}
	\begin{rem}

		As a byproduct of the proof above, for the homogeneous
		problem $\left(\hm{F}=0\right)$, from the Laplace transfrom of the
		resolvent, one has
		\begin{equation}
		\theta\left(t\right)\mathcal{P}_{c,\beta}\hm{u}\left(t\right)=\frac{1}{2\pi}\int_{\mathbb{R}}e^{i\left(\lambda+\epsilon i\right)t}\mathcal{R}_{\beta}\left(\lambda+i\epsilon\right)\mathcal{P}_{c,\beta}\hm{u}_{0}\,d\lambda\label{eq:-3-1}
		\end{equation}
		and
		\begin{equation}
		\theta\left(-t\right)\mathcal{P}_{c,\beta}\hm{u}\left(t\right)=\frac{1}{2\pi }\int_{\mathbb{R}}e^{i\left(\lambda-\epsilon i\right)t}\mathcal{R}_{\beta}\left(\lambda-i\epsilon\right)\mathcal{P}_{c,\beta}\hm{u}_{0}\,d\lambda\label{eq:-9-2}
		\end{equation}
		with any $1\gg\epsilon>0$ where $\theta\left(s\right)$ is the Heviside
		function.
	\end{rem}
	
	\begin{proof}[Proof of Theorem \ref{thm:localenergy}]
		We first prove the inhomogeneous estimate. For the inhomogeneous estimate, for $t\geq0$, by the Plancherel theorem in time and Lemma \ref{lem:matrixresolbeta},  for $\epsilon>0$, one has
		\begin{align*}
			\left\Vert \int\int_{0\leq s\leq t}e^{\left(t-s\right)\mathcal{L}_{\beta}-\epsilon\left(t-s\right)+it\rho}\mathcal{P}_{c,\beta}\hm{F}\left(s\right)\,dsdt\right\Vert _{L_{t}^{2}\mathfrak{H}^{-\tau}}\\
			=\left\Vert \int_{0}^{\infty}e^{is\rho}\mathcal{R}_{\beta}\left(\rho-i\epsilon\right)\mathcal{P}_{c,\beta}\hm{F}\left(s\right)\,ds\right\Vert _{L_{\rho}^{2}\mathfrak{H}^{-\tau}}
			&\lesssim\left\Vert \int_{0}^{\infty}e^{is\rho}\hm{F}\left(s\right)\,ds\right\Vert _{L_{\rho}^{2}\mathfrak{H}^{\tau}}\sim\left\Vert \hm{F}\right\Vert _{L_{t}^{2}\mathfrak{H}^{\tau}}.
		\end{align*}
		The desired result follows after passing $\epsilon\rightarrow0$.
		For $t\leq0$, we the same argument with $\epsilon<0$ give the estimate.
		Hence we conclude that
		\begin{equation}\label{eq:localenergyinhomo}
		\left\Vert \mathcal{P}_{c,\beta}\int_{0}^{t}e^{\left(t-s\right)\mathcal{L}_{\beta}}\hm{F}\left(s\right)\,ds\right\Vert _{L_{t}^{2}\mathfrak{H}^{-\tau}}\lesssim\left\Vert \hm{F}\right\Vert _{L_{t}^{2}\mathfrak{H}^{\tau}}.
		\end{equation}

		Next, we consider the homogeneous estimate. Note that since the linear operator $\mathcal{L}_{\beta}$ is not
		self-adjoint, the homogeneous estimate does not follow from the $TT^{*}$
		directly.  		We follow the approach in the spirit of the argument for the matrix
		Schr\"odinger problem by Schlag \cite{Sch1}. Alternatively, for the homogeneous problem, one can use the energy
		comparison argument
		in 	Chen-Jendrej \cite[Proposition 4.4]{CJ} 
		to reduce to the standard case
		with $\beta=0$.
		
		Using the Duhamel formual, we write
		\[
		e^{t\mathcal{L}_{\beta}}\mathcal{P}_{c,\beta}=e^{t\mathcal{L}_{0,\beta}}\mathcal{P}_{c,\beta}+\int_{0}^{t}e^{\left(t-s\right)\mathcal{L}_{0.\beta}}\mathcal{V}_{\beta}e^{s\mathcal{L}_{\beta}}\mathcal{P}_{c,\beta}\,ds.
		\]
		The first term above is easy to estimate since it is the homogeneous free problem. Applying the inhomogeneous estimate obtained above for the free case, it suffices to show
		\[
		\left\Vert \mathcal{V}_{\beta}e^{\mathcal{L}_{\beta}s}\mathcal{P}_{c,\beta}\hm{\psi}\right\Vert _{L_{t}^{2}\mathfrak{H}^{\tau}}\lesssim\left\Vert \hm{\psi}\right\Vert _{\mathfrak{H}}.
		\]
		Taking the Fourier transform in $s$, the expression above is equivalent
		to
		\[
		\left\Vert \left\langle x\right\rangle ^{-\tau}\mathcal{V}_{\beta}\left[\mathcal{P}_{c,\beta}\left(\mathcal{L}_{\beta}-i(\lambda+i0)\right)\mathcal{P}_{c}(\mathcal{L}_{\beta})\right]^{-1}\mathcal{P}_{c}(\mathcal{L}_{\beta})\hm{\psi}\right\Vert _{L_{\lambda}^{2}\mathfrak{H}}\lesssim\left\Vert \hm{\psi}\right\Vert _{\mathfrak{H}}.
		\]
		The desired result holds for the free operator $\mathcal{L}_{0,\beta}$ trivially.
		Then we use the resolvent identity
		\[
		\left(\mathcal{L}_{\beta}-i(\lambda+i0)\right)^{-1}=\left[1-\left(\mathcal{L}_{0,\beta}-i(\lambda+i0)\right)^{-1}\mathcal{V}\right]^{-1}\left(\mathcal{L}_{0,\beta}-i(\lambda+i0)\right)^{-1}.
		\]
		Since the bottoms of the spectrum are not resonances, for $\mu=\frac{1}{\gamma}$, we have
		\[
		\sup_{|\lambda|>\mu/2}\text{\ensuremath{\left\Vert \left[1-\left(\mathcal{L}_{0,\beta}-i(\lambda+i0)\right)^{-1}\mathcal{V}\right]^{-1}\right\Vert }}_{\mathfrak{H}\rightarrow\mathfrak{H}^{-\tau}}<\infty
		\]
		whence,
		\[
		\sup_{|\lambda|>\mu/2}\text{\ensuremath{\left\Vert \left\langle x\right\rangle ^{-\tau}\mathcal{V}_{\beta}\left[1-\left(\mathcal{L}_{0,\beta}-i(\lambda+i0)\right)^{-1}\mathcal{V}\right]^{-1}\right\Vert }}_{\mathfrak{H}\rightarrow\mathfrak{H}}<\infty.
		\]
		Therefore, it follows that
		\begin{align*}
			\int_{|\lambda|>\mu/2}\left\Vert \left\langle x\right\rangle ^{-\tau}\mathcal{V}_{\beta}\left[\mathcal{P}_{c,\beta}\left(\mathcal{L}_{\beta}-i(\lambda+i0)\right)\mathcal{P}_{c,\beta}\right]^{-1}\mathcal{P}_{c,\beta}\hm{\psi}\right\Vert _{\mathfrak{H}}^{2}d\lambda\\
			\lesssim\int_{|\lambda|>\mu/2}\left\Vert \left(\mathcal{L}_{0,\beta}-i(\lambda+i0)\right)^{-1}\hm{\psi}\right\Vert _{\mathfrak{H}}^{2}d\lambda
			\lesssim\left\Vert \hm{\psi}\right\Vert _{\mathfrak{H}}^{2}.
		\end{align*}
		We also have
		\[
		\sup_{|\lambda|<\mu/2}\left\Vert \left[\mathcal{P}_{c,\beta}\left(\mathcal{L}_{\beta}-i(\lambda+i0)\right)\mathcal{P}_{c,\beta}\right]^{-1}\right\Vert _{\mathfrak{H}\rightarrow\mathfrak{H}}<\infty
		\] since $\lambda$ is not in the spectrum. 
		
		Putting things together, we conclude that
		\begin{equation}\label{eq:localenergy}
		\left\Vert e^{\mathcal{L}_{\beta}t}\mathcal{P}_{c,\beta}\hm{u_0}\right\Vert _{L_{t}^{2}\mathfrak{H}^{-\tau}}\lesssim\left\Vert \hm{u}_{0}\right\Vert _{\mathfrak{H}}.
		\end{equation}

		Given \eqref{eq:localenergyinhomo}	and \eqref{eq:localenergy}, Theorem \ref{thm:localenergy} is proved.
	\end{proof}

	
	\section{Analysis of the single-potential problem\label{sec:onepotential}}
	
	In this section, we consider the following equation
	\begin{equation}
	w_{tt}-\Delta w+w+V_{\beta\left(t\right)}\left(x-y\left(t\right)\right)w=F\label{eq:onePmaineq}
	\end{equation}
	with the assumptions that
	\begin{equation}\label{eq:onePmainzero}
	\pi_0(t) \hm{w}(t)=0,\, \hm{w}=(w,w_t)^{T}
	\end{equation}
	and
	\begin{equation}
	\left\Vert \beta'\left(t\right)\right\Vert _{L^{1}\bigcap L^{\infty}}+\left\Vert y'\left(t\right)-\beta\left(t\right)\right\Vert _{L^{1}\bigcap L^{\infty}}\lesssim\delta\ll1.\label{eq:onePmaincond}
	\end{equation}
	We will establish Strichartz
	estimates for the equation  above after projecting onto the centre-stable space from Proposition \ref{prop:expdicho}.  For
	the multiple-potential problem, after performing the channel decomposition,
	in each channel, the analysis will be similar to the analysis here
	with extra interaction terms. One can also regard the analysis in this section  as the dispersive analysis
	in Nakanishi-Schlag \cite{NSch} with a large velcocity. Notice that
	since the Lorentz transform mixes up  the time and space variables, the
	equation \eqref{eq:onePmaineq} can not be reduced to the problem studied
	in Nakanishi-Schlag \cite{NSch}.
	
	\subsection{\label{subsec:LinearTraj}Linear trajectories}
	
	We start with establish some elementary estimates for the equation
	with a potential moving along a linear trajectory
	\begin{equation}
	w_{tt}-\Delta w+w+V_{\beta}\left(x-\beta t\right)w=F.\label{eq:movingL}
	\end{equation}
	Denoting $\hm{w}=\left(\begin{array}{c}
	w\\
	w_{t}
	\end{array}\right)$, one can rewrite the equation \eqref{eq:movingL} in the matrix form
	\begin{align}
		\frac{d}{dt}\hm{w} & =\left(\begin{array}{cc}
			0 & 1\\
			\Delta-V_{\beta}\left(x-\beta t\right)-1 & 0
		\end{array}\right)\hm{w}+\hm{F}\\
		& =:\mathcal{L}_{\beta}(t)\hm{w}+\hm{F}
		\label{eq:fullmovingmatrix}\end{align}
	Here we can put a general inhomogeneous term $\hm{F}=\left(\begin{array}{c}
	F_{1}\\
	F_{2}
	\end{array}\right)$. (This term coming from the scalar equation should have a special
	form $\hm{F}=\left(\begin{array}{c}
	0\\
	F
	\end{array}\right).$)

	\begin{lem}
		\label{lem:localdecaytim}Denote the homogeneous evolution operator of 	\eqref{eq:fullmovingmatrix},  the equation $	\frac{d}{dt}\hm{w} =\mathcal{L}_{\beta}(t)\hm{w}$  from
		$s$ to $t$ as $U(t,s)$. 
		%
		Then for some $\alpha>1$, we have
		\[
		\left\Vert \left\langle \cdot-\beta t\right\rangle ^{-\alpha}U\left(t,0\right)P_{c,\beta}\left(0\right)\hm{w}_{0}\right\Vert _{L_{t}^{2}\mathcal{H}}\lesssim\left\Vert \hm{w}_{0}\right\Vert _{\mathcal{H}}.
		\]
		\begin{equation}
		\left\Vert \left\langle \cdot-\beta t\right\rangle ^{-\alpha}\int_{0}^{t}U\left(t,s\right)P_{c,\beta}\left(s\right)\hm{F}\,ds\right\Vert _{L_{t}^{2}\mathcal{H}}\lesssim\left\Vert \left\langle \cdot-\beta t\right\rangle ^{\alpha}\hm{F}\right\Vert _{L_{t}^{2}\mathcal{H}}.\label{eq:localdecaytimeinh}
		\end{equation}
	\end{lem}
	
	\begin{proof}
		This is a direct consequence of  Theorem  \ref{thm:localenergy}.
	\end{proof}
	From the local energy decay, one can obtain Strichartz estimates following
	the general idea introduced in Rodnianski-Schlag \cite{RodSch}. Also
	see Chen \cite{C1,C2} for the wave setting. Since we will discuss
	Strichartz estimates and scattering in more general settings below,
	we omit the derivation of Strichartz estimates here.
	
	\subsection{Reduction}\label{subsec:Reduction}
	
	To obtain the Strichartz estimates for \eqref{eq:onePmaineq}, we first perform some computations to reduce the problem to the case  that the trajectory of the potential is a perturbation of linear trajectory.

	Define an
	approximated trajectory $	\mathrm{y}(t)$ as
	\begin{equation}
	\mathrm{y}'\left(t\right)=\beta\left(t\right),\,\mathrm{y}\left(0\right)=y\left(0\right)\label{eq:appTr}
	\end{equation}
	and set $\beta_{0}=\beta\left(0\right).$ 

	\subsubsection{Reduction of trajectories}
	
	Denoting $\hm{w}=\left(\begin{array}{c}
	w\\
	w_{t}
	\end{array}\right)$, one can rewrite the equation \eqref{eq:onePmaineq} in the matrix form
	\begin{align}
		\frac{d}{dt}\hm{w} & =\left(\begin{array}{cc}
			0 & 1\\
			\Delta-V_{\beta\left(t\right)}\left(x-y\left(t\right)\right)-1 & 0
		\end{array}\right)\hm{w}+\hm{F}.\label{eq:movingTSH}
	\end{align}
	Note that
	\begin{align}
		V_{\beta\left(t\right)}\left(x-y\left(t\right)\right)-V_{\beta_{0}}\left(x-\mathrm{y}\left(t\right)\right)= & V_{\beta\left(t\right)}\left(x-y\left(t\right)\right)-V_{\beta_{0}}\left(x-y\left(t\right)\right)\nonumber \\
		& +\left(V_{\beta_{0}}\left(x-y\left(t\right)\right)-V_{\beta_{0}}\left(x-\mathrm{y}\left(t\right)\right)\right).\label{eq:DiffVscalar}
	\end{align}
	Clearly, by the mean-value theorem, one has
	\begin{equation}
	\left|V_{\beta\left(t\right)}\left(x-y\left(t\right)\right)-V_{\beta_{0}}\left(x-y\left(t\right)\right)\right|\lesssim\left\Vert \beta'\left(t\right)\right\Vert _{L^{1}}\left|\partial_{\beta}V_{\tilde{\beta}}\left(x-y\left(t\right)\right)\right|\label{eq:DiffVscalar1}
	\end{equation}
	where $\partial_{\beta}V_{\beta}$ denotes the differentiation with
	respect to the parameter $\beta$ in the Lorentz boost. On the other
	hand, applying the mean-value theorem again, it follows
	\begin{equation}
	\left|V_{\beta_{0}}\left(x-y\left(t\right)\right)-V_{\beta_{0}}\left(x-\mathrm{y}\left(t\right)\right)\right|\lesssim\left\Vert y'\left(t\right)-\beta\left(t\right)\right\Vert _{L^{1}}\left|\nabla V_{\beta_{0}}\left(x-\tilde{y}\right)\right|.\label{eq:DiffVscalar2}
	\end{equation}
	Note that by construction \eqref{eq:appTr}, we have
	\begin{equation}
	\left|\left(\mathrm{y}\left(t\right)-\beta_{0}t\right)'\right|_{L^{\infty}}\lesssim\left|\beta\left(t\right)-\beta_{0}\right|_{L^{\infty}}\lesssim\left\Vert \beta'\left(t\right)\right\Vert _{L^{1}}.\label{eq:DiffVscalar3}
	\end{equation}
	Therefore we can rewrite the approximated trajectory as
	\begin{equation}
	\mathrm{y}\left(t\right)=\beta_{0}t+c\left(t\right)\label{eq:appTr2}
	\end{equation}
	with $\left|c'\left(t\right)\right|\lesssim\left\Vert \beta'\left(t\right)\right\Vert _{L^{1}}\lesssim\delta.$
	
	Hence putting \eqref{eq:DiffVscalar}, \eqref{eq:DiffVscalar1}, \eqref{eq:DiffVscalar2},
	\eqref{eq:DiffVscalar3} and \eqref{eq:appTr2} together, one can rewrite
	\eqref{eq:movingTSH} as
	\begin{align}
		\frac{d}{dt}\hm{w} & =\left(\begin{array}{cc}
			0 & 1\\
			\Delta-V_{\beta_{0}}\left(x-\beta_{0}t+c\left(t\right)\right)-1 & 0
		\end{array}\right)\hm{w}\label{eq:movingred}\\
		& +\hm{F}_{1}\hm{w}+\hm{F}:=\mathcal{L}_{\beta_0}(c(t))+ \hm{F}_{1}\hm{w}+\hm{F}
	\end{align}
	where
	\[
	\hm{F}_{1}\hm{w}=\left(\begin{array}{c}
	0\\
	\left(V_{\beta\left(t\right)}\left(x-y\left(t\right)\right)-V_{\beta_{0}}\left(x-\beta_{0}t+c\left(t\right)\right)\right)w
	\end{array}\right).
	\]
	By our computations above, $\hm{F}_{1}\hm{w}$ is a localized term
	which is  small in the $L^{\infty}$ norm in terms of $\left\Vert \beta'\left(t\right)\right\Vert _{L^{1}}$
	and $\left\Vert y'\left(t\right)-\beta\left(t\right)\right\Vert _{L^{1}}$.
	Under the endpoint Strichartz norm, 
	due to
	the smallness of $\hm{F}_{1}\hm{w}$ one can always absorb it
	to the left-hand side.

	\subsubsection{Reduction of zero modes}\label{subsubsec:reductionzero}
	In the computations above, we reduce the general trajectory to a linear trajectory with some perturbations.  But recall that we imposed that $\pi_0(t)\hm{w}=0$ explicitly as$$\left<\alpha_{ m}^0(t),\bm{w}(t)\right>=\left<
	\alpha_{m}^1(t),\bm{w}(t)\right>=0,\qquad\forall m$$
	with respect to the original trajectory.
	
	After the reduction \eqref{eq:movingred}, we decompose $\bm{w}$ with respect to the linear operator from \eqref{eq:movingred} as
	\begin{equation}
	\bm{w}(t)=\mathcal{P}_{c,\beta_0}(c(t)) \bm {w}(t) +\mathcal{P}_{0,\beta_0}(c(t)) \bm {w}(t) +\mathcal{P}_{\pm,\beta_0}(c(t)) \bm {w}(t) 
	\end{equation}
	where $\mathcal{P}_{c,\beta_0}(c(t)),\,\mathcal{P}_{0,\beta_0}(c(t)),\,\mathcal{P}_{\pm,\beta_0}(c(t))$ are projections onto the continuous spectrum, zero modes and stable/unstable modes with respect to $\mathcal{L}_{\beta_0}(c(t))$ from \eqref{eq:movingred}.
	
	Now $\mathcal{P}_{0,\beta_0}(c(t)) \bm {w}(t)\neq 0$. But we note that we can write
	\begin{equation}
	\bm{w}(t)= \mathcal{P}_{c,\beta_0}(c(t)) \bm {w}(t) +[\mathcal{P}_{0,\beta_0}(c(t))-\pi_0(t)] \bm {w}(t) +\mathcal{P}_{\pm,\beta_0}(c(t)) \bm {w}(t). 
	\end{equation}
	From the computations above on the difference of trajectories, for the local energy norm, we have\begin{equation}
	\left\Vert\left\langle \cdot-\beta_0 t+c(t)\right\rangle ^{-\alpha}[\mathcal{P}_{0,\beta_0}(c(t))-\pi_0(t)] \bm {w}(t)\right\Vert_{H^1\times L^2} \lesssim\delta \left\Vert\left\langle \cdot-\beta_0 t+c(t)\right\rangle^{-\alpha}\bm {w}(t)\right\Vert_{H^1\times L^2}.
	\end{equation}
	Actually, this computations holds for other norms we are interested for example Strichartz norms.   Therefore, in all spaces we are interested in, one has
	\begin{equation}\label{eq:wwzero}
	\bm{w}(t) \sim [1-\mathcal{P}_{0,\beta_0}(c(t))] \bm {w}(t). 
	\end{equation}In particular, the zero modes with respect to the linear operator coming from \eqref{eq:movingred} will not result in any polynomial growth. Due to the relation \eqref{eq:wwzero}, in our practical computations, we can simply disregard the influence of zero modes.
	
	\begin{rem}
		The same reduction analysis can be applied to stable/unstable modes in the setting of Proposition \ref{pro:linear}.
	\end{rem}

	
	\subsection{\label{subsec:Plineartraj}Perturbation of linear trajectories}
	
	From computations above, we know that to establish Strichartz estimates
	and the local energy decay for \eqref{eq:onePmaineq}, it suffices to
	consider a reduced problem in which the potential moves along a perturbation of a linear trajectory
	with a fixed Lorentz factor:
	\begin{equation}
	w_{tt}-\Delta w+w+V_{\beta}\left(x-\beta t+c\left(t\right)\right)w=F\label{eq:pertlinearTrajw}
	\end{equation}
	with $\left|c'\left(t\right)\right|\ll1$. 
	
	Again writing the equation in its Hamiltonian form, we have
	\[
	\frac{d}{dt}\hm{w}=\left(\begin{array}{cc}
	0 & 1\\
	\Delta-V_{\beta}\left(x-\beta t+c\left(t\right)\right)-1 & 0
	\end{array}\right)\hm{w}+\hm{F}.
	\]
	We will prove local energy decay and Strichartz estimates for the system above. Here we
	will use the short-hand notation
	\[
	L_{t}^{p}L_{x}^{q}:=L^{p}([0,\infty);L_{x}^{q}).
	\]
	Making a change of variable $y=x+c\left(t\right)$ and projecting
	onto the stable space, it suffices to consider the following equation
	\begin{equation}
	\hm{\gamma}_{t}=\mathcal{L}_{\beta}(t)\hm{\gamma}+P_{c}\left(t\right)\left[a\left(t\right)\nabla\hm{\gamma}+\hm{F}\right]\label{eq:pertlinearTrajgamma}
	\end{equation}
	where
	\begin{equation}\label{eq:cLbeta}
	\mathcal{L}_{\beta}(t):=\left(\begin{array}{cc}
	0 & 1\\
	\Delta-V_{\beta}\left(x-\beta t\right)-1 & 0
	\end{array}\right)
	\end{equation}
	and $a\left(t\right)=c'\left(t\right)$, $\gamma\left(0\right)=P_{c}\left(0\right)\gamma\left(0\right)$
	and $\left\Vert a\right\Vert _{L^{\infty}}\ll1$. Here $P_{c}\left(t\right)$
	is the projection on the continuous spectrum associated with $\mathcal{L}_{\beta}(t)$ which is via  \eqref{eq:Pcbeta} from  \eqref{eq:Pdbeta} after a shift  by $-\beta t$. 
	

	Denote $A\left(t\right)=a\left(t\right)\cdot\nabla.$ It is sufficient
	to consider Strichartz estimates for the following equation
	\begin{equation}
	z_{t}=\mathcal{L}_{\beta}(t)P_{c}\left(t\right)z-\kappa P_{d}\left(t\right)z+\left[A\left(t\right)\left(P_{c}\left(t\right)z\right)+F\right],\ \ z\left(0\right)=\hm{\gamma}\left(0\right)\label{eq:pertlinearMmo}
	\end{equation}
	where $ P_{d}\left(t\right)$ is the projection on to the disrete spectrum with respect to $\mathcal{L}_{\beta}(t)$ which is  a shift of  \eqref{eq:Pdbeta}   by $-\beta t$  and 
	$\kappa$ is a positive number determined by  the eigenvalue of the
	stationary operator and the velocity). Note that since $P_c(t)$ commutes with the principal part of the linear flow, taking $P_c(t)z(t)=\hm{\gamma}(t)$, the equation \eqref{eq:pertlinearMmo} is reduced back to the original equation  \eqref{eq:pertlinearTrajgamma}.
	
	Here we analyze the principal part of the equation \eqref{eq:pertlinearMmo} more explicitly:
	\begin{equation}
	z_{t}=\mathcal{L}_{\beta}(t)P_{c}\left(t\right)z-\kappa P_{d}\left(t\right)z.\label{eq:zeqsim}
	\end{equation}
	Projecting onto the subspace
	associated with $P_{d}\left(t\right)$, one has
	\begin{equation}
	\left(P_{d}\left(t\right)z\right)_{t}=-\kappa\left(P_{d}\left(t\right)z\right).\label{eq:zeqd}
	\end{equation}
	On the other hand, along the stable space, we have
	\begin{equation}
	\left(P_{c}\left(t\right)z\right)_{t}=\mathcal{L}_{\beta}(t)P_{c}\left(t\right)z.\label{eq:zeqc}
	\end{equation}
	If we denote the evolution operator from $s$ to $t$ of the
	linear equation \eqref{eq:zeqsim}  above as $\mathcal{E}\left(t,s\right)$. 
	Then from \eqref{eq:zeqd} and \eqref{eq:zeqc}, we can write
	\[
	z\left(t\right)=\mathcal{E}\left(t,0\right)P_{c}\left(0\right)z\left(0\right)+e^{-\kappa t}P_{d}\left(0\right)z\left(0\right).
	\]
	Moreover, one also has
	\[
	z\left(t\right)=\mathcal{E}\left(t,s\right)P_{c}\left(s\right)z\left(s\right)+e^{-\left(t-s\right)\kappa}P_{d}\left(s\right)z\left(s\right).
	\]
	Transparently, $\mathcal{E}\left(t,s\right)P_{c}\left(s\right)$ is
	the same as $U\left(t,s\right)P_{c}\left(s\right)$ from Lemma \ref{lem:localdecaytim}.
	Note that \eqref{eq:zeqsim} can be rearranged as
	\[
	z_{t}=\mathcal{L}_{\beta}(t)z-\kappa P_{d}\left(t\right)z-\mathcal{L}_{\beta}(t)P_{d}\left(t\right)z.
	\]
	We can write the original equation \eqref{eq:pertlinearMmo} as
	\begin{equation}
	z_{t}=\left[\mathfrak{L}_{0}+\mathrm{V}_{\beta}\left(x-\beta t\right)+A\left(t\right)\right]z+\tilde{F}\label{eq:-8}
	\end{equation}
	where
	\begin{equation}
	\mathrm{V}_{\beta}\left(x-\beta t\right)=-\mathcal{L}_{\beta}(t)P_{d}\left(t\right)-\kappa P_{d}\left(t\right)+\mathcal{K}_{\beta}\left(x-\beta t\right)\label{eq:pertlinearPotential}
	\end{equation}
	and
	\begin{equation}
	\tilde{F}\left(t\right)=-A\left(t\right)P_{d}\left(t\right)z\left(t\right)+\mathrm{F}\left(t\right).\label{eq:pertlinearF}
	\end{equation}
	Here $\mathcal{K_{\beta}}\left(x-\beta t\right)=\left(\begin{array}{c}
	0\\
	V_{\beta}\left(x-\beta t\right)
	\end{array}\right)$ denotes the difference between $\mathcal{L}_{\beta}(t)$ and the
	free operator
	\begin{equation}\label{eq:freematrix}
	\mathfrak{L}_{0}:=\left(\begin{array}{cc}
	0 & 1\\
	\Delta-1 & 0
	\end{array}\right).
	\end{equation}
	Denote $\mathrm{U}_{A}\left(t,s\right)$ the evolution operator defined
	by the equation $u_{t}=A\left(t\right)u$. More explicitly,
	\begin{equation}
	\mathrm{U}_{A}\left(t,s\right)\varphi=\varphi\left(x+b\left(t,s\right)\right)\label{eq:UA}
	\end{equation}
	where
	$
	b\left(t,s\right)=\int_{s}^{t}a\left(\tau\right)\,d\tau.$
	
	Trivially, $\mathfrak{L}_{0}$ and $A\left(t\right)$ commute with
	each other and their evolution operators also commute. We can write
	down the solution $z$ using the Duhamel formula
	\begin{equation}
	z(t)=z_{0}\left(t\right)+\int_{0}^{t}e^{\left(t-s\right)\mathfrak{L}_{0}}\mathrm{U}_{A}\left(t,s\right)\left\{ \mathrm{V}_{\beta}\left(s\right)z\left(s\right)+\tilde{F}\left(s\right)\right\} \,ds\label{eq:zduhamel}
	\end{equation}
	where
	\begin{equation}
	z_{0}=e^{t\mathfrak{L}_{0}}\mathrm{U}_{A}\left(t,0\right)z\left(0\right).\label{eq:z0}
	\end{equation}
	To study estimates for \eqref{eq:zduhamel}, we introduce three auxiliary
	operators. Define
	\begin{equation}
	Tg:=\int_{0}^{t}e^{\left(t-s\right)\mathfrak{L}_{0}}\mathrm{U}_{A}\left(t,s\right)\mathrm{V}_{\beta}\left(s\right)g\left(s\right)\,ds\label{eq:T}
	\end{equation}
	\begin{equation}
	T_{0}g:=\int_{0}^{t}e^{\left(t-s\right)\mathfrak{L}_{0}}\mathrm{V}_{\beta}\left(s\right)g\left(s\right)\,ds\label{eq:T0}
	\end{equation}
	\begin{equation}
	T_{1}g:=\int_{0}^{t}\mathrm{U}_{\beta}\left(t,s\right)\mathrm{V}_{\beta}\left(s\right)g\left(s\right)\,ds\label{eq:T1}
	\end{equation}
	where $\mathrm{U}_{\beta}\left(t,s\right)$ denotes the evolution
	of the linear equation
	\begin{equation}
	z_{t}=\left[\mathfrak{L}_{0}+\mathrm{V}_{\beta}\left(x-\beta t\right)\right]z.\label{eq:linearU}
	\end{equation}
	Then we observe that by construction \eqref{eq:T} and the expansion
	\eqref{eq:zduhamel}, one has
	\begin{equation}
	\left(1-T\right)z=z_{0}+\int_{0}^{t}e^{\left(t-s\right)\mathfrak{L}_{0}}\mathrm{U}_{A}\left(t,s\right)\tilde{F}\left(s\right)\,ds.\label{eq:1-T}
	\end{equation}
	To understand the operator $T$, by \eqref{eq:T0} and \eqref{eq:T} for
	a general inhomogeneous term $\hm{g}$, we have
	\begin{equation}
	\left(T_{0}-T\right)\hm{g}=\int_{0}^{t}e^{\left(t-s\right)\mathfrak{L}_{0}}\left(1-\mathrm{U}_{A}\left(t,s\right)\right)\hm{g}\left(s\right)\,ds.\label{eq:T0-T}
	\end{equation}
	To obtain Strichartz estimates for $z$ from \eqref{eq:zduhamel}, we
	first need to show the local energy decay. The local energy decay
	of $z$ is achieved by two key points:
	\begin{itemize}
		\item The RHS of \eqref{eq:1-T} can be estimated in the local energy norm.
		\item The invertbility of $\left(1-T\right)$ in \eqref{eq:1-T} in the local
		energy norm.
	\end{itemize}
	To establish the desired results above, we split the proof into the following steps:
	\begin{enumerate}
		\item We will first show that $\left(1-T_{0}\right)$ is invertible.
		\item Then
		using dispersive estimates, we show that the difference  between $\left(T_{0}-T\right)$ is small.
		\item The invertiblity of $\left(1-T\right)$ will follow the two items above by the Neumann's series.
	\end{enumerate}
	The dispersive estimate will also result in the estimate of the RHS of \eqref{eq:1-T}.
	We proceed analysis of steps above with some lemmata as in Nakanishi-Schlag \cite{NSch}
	adapted to our current setting with a large velocity.
	
	First of all, we have the invertbility of $\left(1-T_{0}\right)$.
	The inverse of $\left(1-T_{0}\right)$ can be computed explicitly
	as following:
	\begin{lem}
		\label{lem:T0T1}Given notations from \eqref{eq:T0} and \eqref{eq:T1}, $T_{0}$, $T_{1}$ are bounded on $L_{t}^{2}\mathcal{H}\left\langle \cdot-\beta t\right\rangle ^{-\sigma}$
		in the following sense: for any $\alpha>1$, we have for $j=0,1$
		\begin{equation}
		\left\Vert \left\langle \cdot-\beta t\right\rangle ^{-\alpha}T_{j}\hm{g}\right\Vert _{L_{t}^{2}\mathcal{H}}\lesssim\left\Vert \left\langle \cdot-\beta t\right\rangle ^{-\alpha}\hm{g}\right\Vert _{L_{t}^{2}\mathcal{H}}\label{eq:T0T1decay}
		\end{equation}
		and
		\begin{equation}
		\left(1-T_{0}\right)\left(1+T_{1}\right)=\left(1+T_{1}\right)\left(1-T_{0}\right)=1.\label{eq:T0T1rela}
		\end{equation}
	\end{lem}
	
	\begin{proof}
		We first prove two identities in \eqref{eq:T0T1rela}. By construction,
		\begin{align}
			T_{0}T_{1}\hm{f} & =\int\int_{0<t_{1}<t_{0}<t}e^{\left(t-t_{0}\right)\mathfrak{L}_{0}}\mathrm{V}_{\beta}\left(t_{0}\right)\mathrm{U}_{\beta}\left(t_{0},t_{1}\right)\mathrm{V}_{\beta}\left(t_{1}\right)\hm{f}\left(t_{1}\right)\,dt_{1}dt_{0}\nonumber \\
			& =\int_{0}^{t}\left[\mathrm{U}_{\beta}\left(t,t_{1}\right)-e^{\left(t-t_{1}\right)\mathfrak{L}_{0}}\right]\mathrm{V}_{\beta}\left(t_{1}\right)\hm{f}\left(t_{1}\right)\,dt_{1}=T_{1}\hm{f}-T_{0}\hm{f}.\label{eq:T0T1}
		\end{align}
		where in the last step, we used
		\begin{equation}
		\mathrm{U}_{\beta}\left(t,t_{1}\right)\hm{\phi}=e^{\left(t-t_{1}\right)\mathfrak{L}_{0}}\hm{\phi}+\int_{t_{1}}^{t}e^{\left(t-t_0\right)\mathfrak{L}_{0}}\mathrm{V}_{\beta}\left(t_0\right)\mathrm{U}_{\beta}\left(t_0,t_{1}\right)\hm{\phi}\,dt_0.\label{eq:expUbeta}
		\end{equation}
		Reversing the Duhamel formula, we can also obtain
		\[
		T_{1}T_{0}=-T_{0}+T_{1}
		\]
		in the same manner. Therefore,
		\[
		\left(1-T_{0}\right)\left(1+T_{1}\right)=\left(1+T_{1}\right)\left(1-T_{0}\right)=1.
		\]
		Next we prove weighted estimates. First of all, for the free case,
		we recall that
		\begin{equation}
		T_{0}g=\int_{0}^{t}e^{\left(t-s\right)\mathfrak{L}_{0}}\mathrm{V}_{\beta}\left(s\right)\hm{g}\left(s\right)\,ds.
		\end{equation}
		The desired estimate is nothing but the Kato smoothing or the local
		energy decay for the operator $e^{\mathfrak{L}_{0}t+i\beta t\nabla}.$
		The perturbed version is discussed in Lemma \ref{lem:localdecaytim}.
		More precisely, we write
		\begin{align}
			T_{1}g & =\int_{0}^{t}\mathrm{U}_{\beta}\left(t,s\right)\mathrm{V}_{\beta}\left(s\right)\hm{g}\left(s\right)\,ds\nonumber \\
			& =\int_{0}^{t}\mathrm{U}_{\beta}\left(t,s\right)P_{c}\left(s\right)\mathrm{V}_{\beta}\left(s\right)\hm{g}\left(s\right)\,ds+\int_{0}^{t}\mathrm{U}_{\beta}\left(t,s\right)P_{d}\left(s\right)\mathrm{V}_{\beta}\left(s\right)\hm{g}\left(s\right)\,ds\nonumber \\
			& =\int_{0}^{t}U\left(t,s\right)P_{c}\left(s\right)\mathrm{V}_{\beta}\left(s\right)\hm{g}\left(s\right)\,ds +\int_{0}^{t}\mathrm{U}_{\beta}\left(t,s\right)P_{d}\left(s\right)\mathrm{V}_{\beta}\left(s\right)\hm{g}\left(s\right)\,ds\label{eq:T1expan}
		\end{align}
		where in the last equality, we used the notation from Lemma \ref{lem:localdecaytim}.
		Then the first integral on the RHS of \eqref{eq:T1expan}
		can be bounded by estimates in Lemma \ref{lem:localdecaytim}. The
		second integral can be bounded by choosing $\kappa>0$ large enough
		to defeat the exponential growth rates of unstable modes. 
	\end{proof}
	
	\subsubsection{\label{subsec:DiffT}Difference of $T$ and $T_{0}$}
	Next, we analyze the difference between $T$ and $T_{0}$.
	We know that from the computations above, $\left(1-T_{0}\right)$
	is invertible in the weighted space and its inverse is explicitly given by $(1+T_1)$. If the difference between $T$
	and $T_{0}$ is small enough, then one can also invert $\left(1-T\right)$
	in the weighted space using the Neumann's series.

	We start with some interaction estimates. See Section \ref{ssubec:TruncateInter}
	for the most general version.

	Recalling the notation  $\mathrm{U}_A$, \eqref{eq:UA}, and the operator $\mathcal{D}=\sqrt{-\Delta+1}$ from \eqref{eq:D},  we denote
	\begin{equation}
	I_{j}^{M}\left(f,g\right)=\int_{0}^{\infty}\int_{0}^{\max\left\{ t-M,0\right\} }\left\langle e^{\left(t-s\right)i\mathcal{D}}\mathrm{U}_{A}\left(t,s\right)\varDelta_{j}w\left(s\right)f\left(s\right),w\left(t\right)g\left(t\right)\right\rangle \,dsdt,\label{eq:Imjbeta}
	\end{equation}
	for $j\geq0$, $M\geq1$, $w\left(x,t\right)=\left\langle x-\beta t\right\rangle ^{-\sigma}$
	and $f,\,g\in L_{t,x}^{2}$. $1=\sum_{j=0}^{\infty}\varDelta_{j}$
	is a Littlewood-Paley decomposition defined by $\varDelta_{j}u=\mathcal{F}^{-1}\varphi\left(2^{-j}\xi\right)\hat{u}\left(\xi\right)$
	for all $j\geq1$ with some radial non-negative $\varphi\in C_{0}^{\infty}$
	supported on $\frac{1}{2}<\left|\xi\right|<2$. To sum $I_{j}$ over
	$j\geq0$ we use the almost orthogonality as following
	\begin{equation}
	I_{j}^{M}\left(f,g\right)=I_{j}^{M}\left(Q_{j}f,Q_{j}g\right),\label{eq:ImjQbeta}
	\end{equation}
	\[
	Q_{j}f=\sum_{\left|k-j\right|\leq1}\varDelta_{k}f+w^{-1}\left[\varDelta_{k},w\right]f,
	\]
	where the commutator term is small in the sense that
	\[
	\left\Vert w^{-1}\left[\varDelta_{k},w\right]f\right\Vert _{L^{2}}\lesssim2^{-k}\left\Vert f\right\Vert _{L^{2}}.
	\]
	Next, we prove Lemma 9.5 from Nakanishi-Schlag \cite{NSch} in our
	setting.
	\begin{lem}
		\label{lem:interactionbeta}Let $\sigma>\frac{3}{2}$ and $M\geq1$.
		Then for any $j\geq0$, we have
		\begin{equation}
		\left|I_{j}^{M}\left(f,g\right)\right|\lesssim C\left(\sigma\right)2^{2j}M^{\frac{3}{2}-\sigma}\left\Vert f\right\Vert _{L_{t}^{2}L_{x}^{2}}\left\Vert g\right\Vert _{L_{t}^{2}L_{x}^{2}},\label{eq:interactionbeta}
		\end{equation}
		provided $\left\Vert a\right\Vert _{L_{t}^{\infty}}\ll1$.
	\end{lem}
	
	\begin{proof}
		Define the space-time function $K_{j}$ and $K^{j}$ by
		\begin{equation}
		K_{j}\left(t,x\right):=2^{3j}K^{j}\left(2^{j}t,2^{j}x\right):=e^{it\mathcal{D}}2^{3j}\hat{\varphi}\left(2^{j}x\right).\label{eq:Kjbeta}
		\end{equation}
		Then we can rewrite
		\[
		I_{j}^{M}\left(f,g\right)=J_{j}^{M}\left(f,g\right)+\overline{J_{j}^{M}\left(g,f\right)}
		\]
		where{\small
			\[
			J_{j}^{M}\left(f,g\right)=\int_{\left|x-\beta t\right|>\left|y-s\beta\right|,s>0,t>s+M}K_{j}\left(t-s,x-y+b\right)w\left(s,y\right)f\left(s,y\right)w\left(t,x\right)\overline{g}\left(t,x\right)\,dsdtdxdy.
			\]}We set $t-s=u$ and
		\[
		z=-x+y-b+\beta t-s\beta,\,\,h=x+b-\beta t.
		\]
		Then one has
		\[
		h-b=x-\beta t,\ z+h=y-s\beta.
		\]
		Using the new coordinates, the original integral is rewritten as
		\begin{align}
			J_{j}^{M}\left(f,g\right) & =\int_{\left|h-b\right|>\left|h+z\right|,s>0,u>M}2^{3j}K^{j}\left(2^{j}u,2^{j}\left(z-\beta u\right)\right)w\left(s,h-b\right)f\left(s,h-b\right)\nonumber \\
			& \ \ \ \ \ \ \ \ \ w\left(s+u,h+z\right)\overline{g}\left(s+u,h+z\right)\,dsdudhdz.\label{eq:Jmj1}
		\end{align}
		where $b=b\left(t,s\right)=\int_{s}^{t}a\left(\tau\right)\,d\tau$. 
		
		Recall that we have following pointwise decay for $K^{j}$. For the
		proof, see
		Nakanishi-Schlag  \cite[Proof of (9.60)]{NSch}. 
		For $t\gtrsim1$, one has for $N\in\mathbb{N}$
		arbitrary
		\[
		\left|K^{j}\left(t,x\right)\right|\lesssim\begin{cases}
		t^{-1}\left\langle \left|t-\left|x\right|\right|+2^{-2j}t\right\rangle ^{-N} & \left(\left|t-\left|x\right|\right|\nsim2^{-2k}t\right),\\
		t^{-1}\left\langle t-\left|x\right|\right\rangle ^{-\frac{1}{2}} & \left(\left|t-\left|x\right|\right|\sim2^{-2k}t\right).
		\end{cases}
		\]
		For $u\nsim\left|z-\beta u\right|$, applying the first bound of the
		estimate above, by \eqref{eq:Jmj1}, we can bound{\small
			\begin{align}
				\left|J_{j}^{M}\left(f,g\right)\right| & \lesssim\int2^{3j}\left(2^{j}u\right)^{-3-N}\left|w\left(s,y-b\right)f\left(s,u-b\right)w\left(s+u,y+z\right)\overline{g}\left(s+u,y+z\right)\right|dsdudhdz\nonumber \\
				& \lesssim2^{-jN}M^{-2-N}\left\Vert f\right\Vert _{L_{t}^{2}L_{x}^{2}}\left\Vert g\right\Vert _{L_{t}^{2}L_{x}^{2}}.\label{eq:firstboundJ}
		\end{align}}If $u\sim\left|z-\beta u\right|$ which implies that $\text{\ensuremath{u}\ensuremath{\sim\left|z\right|}}$
		since $\left|\beta\right|<1$. It follows that
		\[
		\left|b\left(t,s\right)\right|\lesssim\left|t-s\right|=u\sim\left|z\right|\lesssim\left|h-b\right|+\left|h+z\right|+\left|b\right|\lesssim\left|h-b\right|
		\]
		since $\left|h+z\right|<\left|h-b\right|$ by construction. Therefore the integral
		$J_{j}^{M}\left(f,g\right)$ above in this region is bounded by
		\begin{align}
			\int_{u\ge M,\left|u\right|\lesssim\left|h-b\right|}\frac{2^{3j}}{2^{j}u}\frac{\left|f\left(s,u-b\right)\overline{g}\left(s+u,y+z\right)\right|}{\left\langle h-b\right\rangle ^{\sigma}\left\langle h+z\right\rangle ^{\sigma}}\,dydzduds\nonumber \\
			\lesssim\int_{u\geq M}2^{j}u^{-1}\left\Vert \left\langle h-b\right\rangle ^{-\sigma}\left\langle h+z\right\rangle ^{-\sigma}\right\Vert _{L_{\left\{ h:\left|h-b\right|\gtrsim u\right\} }^{2}L_{\left|\left|z\right|\sim u\right|}^{2}}\nonumber \\
			\times\left\Vert f\left(s,\cdot\right)\right\Vert _{L_{x}^{2}}\left\Vert g\left(s+u,\cdot\right)\right\Vert _{L_{x}^{2}}\,dsdu\nonumber \\
			\lesssim\int_{u>M}2^{2j}u^{\frac{1}{2}-\sigma}\left\Vert f\left(s,\cdot\right)\right\Vert _{L_{x}^{2}}\left\Vert g\left(s+u,\cdot\right)\right\Vert _{L_{x}^{2}}\,dsdu\nonumber \\
			\lesssim2^{2j}M^{\frac{3}{2}-\sigma}\left\Vert f\right\Vert _{L_{t}^{2}L_{x}^{2}}\left\Vert g\right\Vert _{L_{t}^{2}L_{x}^{2}}.\label{eq:secondboundJ}
		\end{align}
		Therefore, overall, putting \eqref{eq:firstboundJ} and \eqref{eq:secondboundJ}
		together, one has
		\[
		\left|I_{j}^{M}\left(f,g\right)\right|\lesssim C\left(\sigma\right)2^{2j}M^{\frac{3}{2}-\sigma}\left\Vert f\right\Vert _{L_{t}^{2}L_{x}^{2}}\left\Vert g\right\Vert _{L_{t}^{2}L_{x}^{2}}
		\]
		as desired.
	\end{proof}
	The key dispersive estimate is the following lemma.
	\begin{lem}
		\label{lem:shiftslant}Let $1>\nu>0$ and $\sigma\geq\frac{1}{2}+\frac{2}{\nu}$.
		If $\left\Vert a\right\Vert _{L_{t}^{\infty}}\ll1$ then
		\begin{equation}
		\left\Vert \left\langle \cdot-\beta t\right\rangle ^{-\sigma}e^{it\mathcal{D}}\mathrm{U}_{A}\left(t,0\right)\phi\right\Vert _{L_{t}^{2}L_{x}^{2}}\lesssim\left\Vert \phi\right\Vert _{H_{x}^{\frac{\nu}{2}}},\label{eq:shiftfreehom}
		\end{equation}
		\[
		\left\Vert \left\langle \cdot-\beta t\right\rangle ^{-\sigma}\int_{0}^{t}e^{i\left(t-s\right)\mathcal{D}}\mathrm{U}_{A}\left(t,s\right)f\left(s\right)\,ds\right\Vert _{L_{t}^{2}L_{x}^{2}}\lesssim\left\Vert \mathcal{D}^{\nu}f\right\Vert _{L_{t}^{1}H_{x}^{-\frac{\nu}{2}}+L_{t}^{2}L_{x}^{2}\left(\left\langle \cdot-\beta t\right\rangle ^{\sigma}\right)}.
		\]
	\end{lem}
	
	\begin{proof}
		We notice that when $s\sim t$, the interaction is strong. In this
		region, we will apply the trivial $L_{x}^{2}$ conservation:
		\begin{equation}
		\left|I_{j}^{M}\left(f,g\right)-I_{j}^{0}\left(f,g\right)\right|\lesssim M\left\Vert f\right\Vert _{L_{t,x}^{2}}\left\Vert g\right\Vert _{L_{t,x}^{2}}.\label{eq:trivialbeta}
		\end{equation}
		For any $1>\nu>0$, we choose $M=2^{\nu j}$ and take $\sigma$ as
		in the condition of the lemma. Then it follows 
		\[
		2^{2j}M^{\frac{3}{2}-\sigma}\leq M=2^{j\nu}.
		\]
		By \eqref{eq:interactionbeta} from Lemma \ref{lem:interactionbeta} and \eqref{eq:trivialbeta},
		we have
		\begin{equation}
		\left|I_{j}^{0}\left(f,g\right)\right|\lesssim2^{\nu j}\left\Vert f\right\Vert _{L_{t,x}^{2}}\left\Vert g\right\Vert _{L_{t,x}^{2}}.\label{eq:own}
		\end{equation}
		Applying the Littlewood-Paley decomposition and the almost orthogonality,
		it follows
		\begin{align*}
			\left|\int\int_{0}^{t}\left\langle e^{i\left(t-s\right)\mathcal{D}}\mathrm{U}_{A}\left(t,s\right)w\left(s\right)f\left(s\right),w\left(t\right)g\left(t\right)\right\rangle \,dsdt\right| \leq\sum_{j=0}^{\infty}\left|I_{j}^{0}\left(f,g\right)\right|\\
			\lesssim\sum_{j=0}^{\infty}\sum_{\left|j-k\right|+\left|j-\ell\right|\leq2}2^{\nu j}\left\Vert \varDelta_{k}f\right\Vert _{L_{t,x}^{2}}\left\Vert \varDelta_{\ell}g\right\Vert _{L_{t,x}^{2}}+2^{\left(\nu-1\right)j}\left\Vert f\right\Vert _{L_{t,x}^{2}}\left\Vert g\right\Vert _{L_{t,x}^{2}}
			\lesssim\left\Vert f\right\Vert _{L_{t,}^{2}H_{x}^{\nu}}\left\Vert g\right\Vert _{L_{t,x}^{2}}
		\end{align*}
		which implies via duality,
		\begin{align*}
			\left\Vert \left\langle \cdot-\beta t\right\rangle ^{-\sigma}\int_{0}^{t}e^{i\left(t-s\right)\mathcal{D}}\mathrm{U}_{A}\left(t,s\right)f\left(s\right)\,ds\right\Vert _{L_{t}^{2}L_{x}^{2}} & \lesssim\left\Vert \mathcal{D}^{\nu}\left(\left\langle \cdot-\beta t\right\rangle ^{\sigma}\right)f\right\Vert _{L_{t,x}^{2}}\\
			& \lesssim\left\Vert \mathcal{D}^{\nu}f\right\Vert _{L_{t}^{2}L_{x}^{2}\left(\left\langle \cdot-\beta t\right\rangle ^{\sigma}\right).}
		\end{align*}
		Similar estimates hold for $t<s<\infty$. 
		
		Therefore one can conclude that
		\begin{equation}
		\left\Vert \left\langle \cdot-\beta t\right\rangle ^{-\sigma}\int e^{i\left(t-s\right)\mathcal{D}}\mathrm{U}_{A}\left(t,s\right)f\left(s\right)\,ds\right\Vert _{L_{t}^{2}L_{x}^{2}}\lesssim\left\Vert \mathcal{D}^{\nu}f\right\Vert _{L_{t}^{2}L_{x}^{2}\left(\left\langle \cdot-\beta t\right\rangle ^{\sigma}\right).}\label{eq:homobeta}
		\end{equation}
		It remains to prove the homogeneous estimate. Define
		\[
		\mathcal{K}\phi=\left\langle x-\beta t\right\rangle ^{-\sigma}e^{it\mathcal{D}}\mathrm{U}_{A}\left(t,0\right)\mathcal{D}^{-\frac{\nu}{2}}\phi
		\]
		then
		\[
		\mathcal{K}^{*}f=\int\mathcal{D}^{-\frac{\nu}{2}}e^{-is\mathcal{D}}\mathrm{U}_{A}\left(0,s\right)\left\langle \cdot-\beta s\right\rangle ^{-\sigma}f\left(s\right)\,ds.
		\]
		Consider
		\[
		\mathcal{K}\mathcal{K}^{*}f=\left\langle x-\beta t\right\rangle ^{-\sigma}\int e^{i\left(t-s\right)\mathcal{D}}\mathcal{D}^{-\nu}\mathrm{U}_{A}\left(t,s\right)\left\langle \cdot-\beta s\right\rangle ^{-\sigma}f\left(s\right)\,ds.
		\]
		By our analysis above \eqref{eq:homobeta}, it follows
		\[
		\left\Vert \mathcal{K}\mathcal{K}^{*}f\right\Vert _{L_{t,x}^{2}}\lesssim\left\Vert f\right\Vert _{L_{t}^{2}L_{x}^{2}.},
		\]
		whence
		\[
		\left\Vert \mathcal{K}\psi\right\Vert _{L_{x,t}^{2}}\lesssim\left\Vert \psi\right\Vert _{L_{x}^{2}.}
		\]
		In particular,
		\[
		\left\Vert \left\langle \cdot-\beta t\right\rangle ^{-\sigma}e^{it\mathcal{D}}\mathrm{U}_{A}\left(t,0\right)\phi\right\Vert _{L_{t}^{2}L_{x}^{2}}\lesssim\left\Vert \phi\right\Vert _{H_{x}^{\frac{\nu}{2}}}.
		\]
		The desired results are proved.
	\end{proof}
	Finally, we use the dispersive estimate above to analyze the difference between
	$T$ and $T_{0}$.
	\begin{lem}
		\label{lem:diffT}For $\sigma>14$, one has
		\[
		\left\Vert \left\langle \cdot-\beta t\right\rangle ^{-\sigma}\left(T_{0}-T\right)\hm{g}\right\Vert _{L_{t}^{2}\mathcal{H}}\lesssim\left\Vert a\right\Vert _{L_{t}^{\infty}}^{\frac{1}{4}}\left\Vert \left\langle \cdot-\beta t\right\rangle ^{-\sigma}\hm{g}\right\Vert _{L_{t}^{2}\mathcal{H}}.
		\]
		Hence if $\left\Vert a\right\Vert _{L_{t}^{\infty}}$ is small enough,
		then one can use Neumann series to conclude $\left(1-T\right)$ is invertible on $L_{t}^{2}\mathcal{H}\left(\left\langle \cdot-\beta t\right\rangle ^{-\sigma}\right)$
		as well as $L_{t}^{2}\mathcal{D}^{\nu}\mathcal{H}\left(\left\langle \cdot-\beta t\right\rangle ^{-\sigma}\right)$
		for $\nu>0$ small enough.
	\end{lem}
	
	\begin{proof}
		From the definitions of $T$ and $T_{0}$, \eqref{eq:T} and $\eqref{eq:T0}$,
		we have
		\[
		\left(T-T_{0}\right)\hm{g}=\int_{0}^{t}e^{\left(t-s\right)\mathfrak{L}_{0}}\left(\mathrm{U}_{A}\left(t,s\right)-1\right)\mathrm{V}_{\beta}\left(s\right)\hm{g}\left(s\right)\,ds.
		\]
		Note that from the defintinion of $\mathrm{V}_{\beta}\left(s\right)$,
		see \eqref{eq:pertlinearPotential}, the $\mathcal{P}_{d}(s)$ is smooth
		due to the smoothness of eigenfunctions and
		\[
		\mathcal{K_{\beta}}\left(x-\beta t\right)\hm{g}=\left(\begin{array}{c}
		0\\
		V_{\beta}\left(x-\beta t\right)g_{1}
		\end{array}\right)
		\]
		so measuring in $\mathcal{H}$, the inhomogeneous term
		$
		\mathrm{V}_{\beta}\left(s\right)\hm{g}\left(s\right)
		$
		gains one order derivatives.
		
		Explicitly, we know that
		\[
		e^{\tau\mathfrak{L}_{0}}=\left(\begin{array}{cc}
		\cos(\tau\sqrt{-\Delta+1}) & \frac{\sin(\tau\sqrt{-\Delta+1})}{(\sqrt{-\Delta+1})}\\
		-\sqrt{-\Delta+1}\sin(\tau\sqrt{-\Delta+1}) & \cos(\tau\sqrt{-\Delta+1})
		\end{array}\right).
		\]
		Then again using the norm of $\mathcal{H}$ to measure the solution,
		we note that it suffices to consider a slightly more general formulation
		in terms of the half Klein-Gordon evolution $e^{i\tau\mathcal{D}}$
		and measure the solution in the $L^{2}$ norm as the follow
		\[
		\left(\tilde{T}-\tilde{T}_{0}\right)g=\int_{0}^{t}e^{\left(t-s\right)i\mathcal{D}}\left(\mathrm{U}_{A}\left(t,s\right)-1\right)\tilde{\mathrm{V}}_{\beta}\left(s\right)g\left(s\right)\,ds
		\]
		such that $\tilde{\mathrm{V}}_{\beta}\left(s\right)$ localized around
		$x\sim\beta s$ and gains one derivative with $\left\Vert \left\langle \cdot-\beta t\right\rangle ^{-\sigma}g\right\Vert _{L_{t}^{2}L_{x}^{2}}<\infty$.
		It is sufficient to  show the following estimate
		\[
		\left\Vert \left\langle \cdot-\beta t\right\rangle ^{-\sigma}\left(\tilde{T}-\tilde{T}_{0}\right)g\right\Vert _{L_{t}^{2}L_{x}^{2}}\lesssim\left\Vert a\right\Vert _{L_{t}^{\infty}}^{\frac{1}{4}}\left\Vert \left\langle \cdot-\beta t\right\rangle ^{-\sigma}g\right\Vert _{L_{t}^{2}L_{x}^{2}}.
		\]
		We decompose the difference of $\tilde{T}-\tilde{T}_{0}$ into two
		pieces
		\[
		\left(\tilde{T}-\tilde{T}_{0}\right)g=\tilde{T}_{S}g+\tilde{T}_{L}g
		\]
		where
		\[
		\tilde{T}_{S}:=\sum_{j=0}^{\infty}T_{S}^{j}=\sum_{j=0}^{\infty}\int_{\max\left(t-M_{j},0\right)}^{t}e^{\left(t-s\right)i\mathcal{D}}\left(1-\mathrm{U}_{A}\left(t,s\right)\right)\varDelta_{j}\tilde{\mathrm{V}}_{\beta}\left(s\right)g\left(s\right)\,ds
		\]
		with $M_{j}\gg1$ will be chosen later and $\varDelta_{j}$ is the
		Littewood-Paley decomposition. $T_{L}^{j}$ is defined  by a similar
		way. 
		
		First of all, notice that
		\begin{align}
			\left\Vert e^{\left(t-s\right)i\mathcal{D}}\left(1-\mathrm{U}_{A}\left(t,s\right)\right)\varDelta_{j}\tilde{\mathrm{V}}_{\beta}g\right\Vert _{L^{2}} & \lesssim\left\Vert \left(1-U\left(t,s\right)\right)\tilde{\mathrm{V}}_{\beta}g\right\Vert _{L^{2}}\nonumber \\
			& \lesssim\left|b\left(t,s\right)\right|\left\Vert \nabla\tilde{\mathrm{V}}_{\beta}g\right\Vert _{L^{2}}\nonumber \\
			& \lesssim\left|t-s\right|\left\Vert a\right\Vert _{L_{t}^{\infty}}\left\Vert \left\langle \cdot-\beta s\right\rangle ^{-\sigma}g\right\Vert _{L^{2}}\label{eq:firstdiffT}
		\end{align}
		where in the last line we used the fact that $\tilde{\mathrm{V}}_{\beta}$
		gains one derivative.
		
		Secondly, we note that
		\begin{align}
			\left\Vert e^{\left(t-s\right)i\mathcal{D}}\left(1-\mathrm{U}_{A}\left(t,s\right)\right)\varDelta_{j}\tilde{\mathrm{V}}_{\beta}g\right\Vert _{L^{2}} & \leq2\left\Vert \left|\nabla\right|^{-1}\varDelta_{j}\nabla\tilde{\mathrm{V}}_{\beta}g\right\Vert \nonumber \\
			& \lesssim2^{-j}\left\Vert \left\langle \cdot-\beta s\right\rangle ^{-\sigma}g\right\Vert _{L^{2}.}.\label{eq:seconddiffT}
		\end{align}
		Interpolation between \eqref{eq:firstdiffT} and \eqref{eq:seconddiffT}
		will give us
		\[
		\left\Vert e^{\left(t-s\right)i\mathcal{D}}\left(1-\mathrm{U}_{A}\left(t,s\right)\right)\varDelta_{j}\tilde{\mathrm{V}}_{\beta}g\right\Vert _{L^{2}}\lesssim2^{-j/2}\left|t-s\right|^{\frac{1}{2}}\left\Vert a\right\Vert _{L_{t}^{\infty}}^{\frac{1}{2}}\left\Vert \left\langle \cdot-\beta s\right\rangle ^{-\sigma}g\right\Vert _{L^{2}.}
		\]
		Integrating in $s$, we obtain
		\[
		\left\Vert \tilde{T}_{S}^{j}g\right\Vert _{L_{t}^{2}L_{x}^{2}}\lesssim M_{j}^{\frac{3}{2}}2^{-j/2}\left\Vert a\right\Vert _{L_{t}^{\infty}}^{\frac{1}{2}}\left\Vert \left\langle \cdot-\beta s\right\rangle ^{-\sigma}g\right\Vert _{L_{x,t}^{2}.}.
		\]
		For $\tilde{T}_{L}^{j}$, we use the refined disperisve estimate,  Lemma \ref{lem:shiftslant},
		the  derivative gained from $\tilde{\mathrm{V}}_{\beta}$ and its decay, we
		have
		\[
		\left\Vert \left\langle \cdot-\beta t\right\rangle ^{-\sigma}\tilde{T}_{L}^{j}g\right\Vert _{L_{t}^{2}L_{x}^{2}}\lesssim M_{j}^{\frac{3}{2}-\sigma}2^{2j}\left\Vert \left\langle \cdot-\beta s\right\rangle ^{-\sigma}g\right\Vert _{L_{x,t}^{2}.}.
		\]
		Choosing $M_{j}=2^{\frac{j}{6}}\left\Vert a\right\Vert _{L_{t}^{\infty}}^{-\frac{1}{6}}$,  with $\sigma>14$, summing in $j\geq0$, one concludes
		\[
		\left\Vert \left\langle \cdot-\beta t\right\rangle ^{-\sigma}\left(\tilde{T}_{S}g+\tilde{T}_{L}g\right)\right\Vert _{L_{t}^{2}L_{x}^{2}}\lesssim\left\Vert a\right\Vert _{L_{t}^{\infty}}^{\frac{1}{4}}\left\Vert \left\langle \cdot-\beta t\right\rangle ^{-\sigma}g\right\Vert _{L_{t}^{2}L_{x}^{2}}.
		\]
		We are done.
	\end{proof}
	\subsubsection{Strichartz estimate for the auxiliary equation}
	
	By Lemma \ref{lem:T0T1}, we know $\left(1-T_{0}\right)$ is invertible
	in the local energy norm. From Lemma \ref{lem:diffT}, $\left(T-T_{0}\right)$
	is small in the local energy norm provided $\left\Vert a\right\Vert _{L_{t}^{\infty}}$
	is small. Using the Neumann series, we can invert $\left(1-T\right)$.
	Now we show Strichartz estimates for the solution $z$ to the
	auxiliary equation in terms of \eqref{eq:zduhamel}. Recalling \eqref{eq:1-T},
	\[
	\left(1-T\right)z=z_{0}+\int_{0}^{t}e^{\left(t-s\right)\mathfrak{L}_{0}}\mathrm{U}_{A}\left(t,s\right)\tilde{F}\left(s\right)\,ds
	\]
	with $z_{0}=e^{t\mathfrak{L}_{0}}\mathrm{U}_{A}\left(t,0\right)z\left(0\right)$,
	we invert $\left(1-T\right)$ in the local energy norm and obtain
	\begin{align*}
		\left\Vert \left\langle \cdot-\beta t\right\rangle ^{-\sigma}\mathcal{D}^{-\frac{\nu}{2}}z\right\Vert _{L_{t}^{2}\mathcal{H}} & \lesssim\left\Vert \left\langle \cdot-\beta t\right\rangle ^{-\sigma}\mathcal{D}^{-\frac{\nu}{2}}z_{0}\right\Vert _{L_{t}^{2}\mathcal{H}} +\left\Vert \mathcal{D}^{\frac{\nu}{2}}\tilde{F}\right\Vert _{L_{t}^{2}\mathcal{H}\left\langle \cdot-\beta t\right\rangle ^{\sigma}+L_{t}^{1}\mathcal{H}^{1-\frac{\nu}{2}}}\\
		& \lesssim\left\Vert z_{0}\right\Vert _{\mathcal{H}}+\left\Vert a\right\Vert _{L_{t}^{\infty}}\left\Vert \left\langle \cdot-\beta t\right\rangle ^{-\sigma}\mathcal{D}^{-\frac{\nu}{2}}z\right\Vert _{L_{t}^{2}\mathcal{H}}\\
		& +\left\Vert F\right\Vert _{L_{t}^{2}\mathcal{D}^{-\frac{\nu}{2}}\mathcal{H}\left\langle \cdot-\beta t\right\rangle ^{\sigma}+L_{t}^{1}\mathcal{H}}
	\end{align*}
	where on the RHS above, we used Lemma \ref{lem:shiftslant}.
	
	Since $\left\Vert a\right\Vert _{L_{t}^{\infty}}$ is small, one has
	\begin{equation}
	\left\Vert \left\langle \cdot-\beta t\right\rangle ^{-\sigma}\mathcal{D}^{-\frac{\nu}{2}}z\right\Vert _{L_{t}^{2}\mathcal{H}}\lesssim\left\Vert z\left(0\right)\right\Vert _{\mathcal{H}}+\left\Vert F\right\Vert _{L_{t}^{2}\mathcal{D}^{-\frac{\nu}{2}}\mathcal{H}\left\langle \cdot-\beta t\right\rangle ^{\sigma}+L_{t}^{1}\mathcal{H}}.\label{eq:localenergyZbeta}
	\end{equation}
	From the local energy decay to Strichartz estimates now is straightforward.
	We use the expansion for $z$,
	\begin{equation}
	z=z_{0}+\int_{0}^{t}e^{\left(t-s\right)\mathfrak{L}_{0}}\mathrm{U}_{A}\left(t,s\right)\left\{ \mathrm{V}_{\beta}\left(s\right)z\left(s\right)+\tilde{F}\left(s\right)\right\} \,ds\label{eq:expandZ1}
	\end{equation}
	where
	\[
	z_{0}=e^{t\mathfrak{L}_{0}}\mathrm{U}_{A}\left(t,0\right)z\left(0\right).
	\]

	Applying Strichartz norms, see Theorem \ref{thm:StriKGfree}, on both sides
	of \eqref{eq:expandZ1}, one concludes that
	\begin{align*}
		\left\Vert z\right\Vert _{S_{\mathcal{H}}} & \lesssim\left\Vert z_{0}\right\Vert _{S_{\mathcal{H}}}+\left\Vert \int_{0}^{t}e^{\left(t-s\right)\mathfrak{L}_{0}}\mathrm{U}_{A}\left(t,s\right)\left\{ \mathrm{V}_{\beta}\left(s\right)z\left(z\right)+\tilde{F}\left(s\right)\right\} \,ds\right\Vert _{S_{\mathcal{H}}}\\
		& \lesssim\left\Vert z\left(0\right)\right\Vert _{\mathcal{H}}+\left\Vert \left\langle \cdot-\beta t\right\rangle ^{-\sigma}\mathcal{D}^{-\frac{\nu}{2}}z\right\Vert _{L_{t}^{2}\mathcal{H}}+\left\Vert F\right\Vert _{S_{\mathcal{H}}^{*}}
	\end{align*}
	where we used the notation from \eqref{eq:X_H} and
	\begin{equation}\label{eq:strinorm}
	S:=L_{t}^{\infty}L_{x}^{2}\bigcap L_{t}^{2}B_{6,2}^{-5/6}, \,	\mathcal{B}:=B_{6,2}^{-5/6}.
	\end{equation}
	Therefore, we obtain Strichartz estimate for the auxiliary equation
	\[
	\left\Vert z\right\Vert _{S_{\mathcal{H}}}\lesssim\left\Vert z\left(0\right)\right\Vert _{\mathcal{H}}+\left\Vert F\right\Vert _{L_{t}^{2}\left(\mathcal{H}^{1-\frac{\nu}{2}}\left\langle \cdot-\beta t\right\rangle ^{\sigma}\bigcap\mathcal{B}_{\mathcal{H}}^{*}\right)+L_{t}^{1}\mathcal{H}}.
	\]
	where we  again used the notation from \eqref{eq:X_H}.
	Translating the result above to the original problem after applying the projection to $z$, we conclude the local energy decay and Strichartz estimates.
	\begin{prop}
		\label{prop:Stri1}Let $\nu>0$ be small and $\sigma$ be large. There
		exists $0<\delta\ll1$ such that if $\left\Vert a\right\Vert _{L^{\infty}}\leq\delta$
		then solution to the following equation
		\begin{equation}
		\hm{\gamma}_{t}=\mathcal{L}_{\beta}(t)\hm{\gamma}+P_{c}\left(t\right)\left[a\left(t\right)\nabla\hm{\gamma}+\hm{F}\right]\label{eq:-1-1-3}
		\end{equation}
		with $\hm{\gamma}\left(0\right)=P_{c}\left(0\right)\hm{\gamma}\left(0\right)$
		satisfies
		\[
		\left\Vert \hm{\gamma}\right\Vert _{S_{\mathcal{H}}\bigcap L_{t}^{2}\mathcal{D}^{\frac{\nu}{2}}\mathcal{H}\left\langle \cdot-\beta t\right\rangle ^{-\sigma}}\lesssim\left\Vert \hm{\gamma}\left(0\right)\right\Vert _{\mathcal{H}}+\left\Vert \hm{F}\right\Vert _{L_{t}^{2}\left(\mathcal{H}^{1-\frac{\nu}{2}}\left\langle \cdot-\beta t\right\rangle ^{\sigma}\bigcap\mathcal{B}_{\mathcal{H}}^{*}\right)+L_{t}^{1}\mathcal{H}}.
		\]
	\end{prop}

	The boundedness of Stricharz norms and the weighted norm, in particular,
	implies scattering in the shifted coordinate. We will discuss this later
	on in more general settings.
	
	\subsection{Strichartz estimates in the general setting}
	
	In this subsection, we conclude Strichartz estimates and local energy
	decay for \eqref{eq:onePmaineq} under the conditions \eqref{eq:onePmainzero} and \eqref{eq:onePmaincond}.
	Without loss of generality, suppose $y\left(0\right)=0$. Let $\mathrm{y}'\left(t\right)=\beta\left(t\right)$
	and $\beta_{0}=\beta\left(0\right).$
	
	Writing $\hm{w}=\left(\begin{array}{c}
	w\\
	w_{t}
	\end{array}\right)$, one has
	\begin{equation}
	\hm{w}_{t}=\mathcal{L}\left(t,\beta\left(t\right),y\left(t\right)\right)\hm{w}+\hm{F}\label{eq:generallinearv}
	\end{equation}
	where
	\[
	\mathcal{L}\left(t,\beta\left(t\right),y\left(t\right)\right)=\left(\begin{array}{cc}
	0 & 1\\
	\Delta w-w-V_{\beta\left(t\right)}\left(x-y\left(t\right)\right) & 0
	\end{array}\right).
	\]
	Denoting $\hm{v}\left(t\right)=\pi_{cs}\left(t\right)\hm{w}$ where $\pi_{cs}$ is from Proposition \ref{prop:expdicho},
	we have
	\begin{equation}
	\hm{v}_{t}=\mathcal{L}\left(t,\beta\left(t\right),y\left(t\right)\right)\hm{v}+\pi_{cs}\left(t\right)\hm{F}\label{eq:generallinearcom}
	\end{equation}
	since the projection $\pi_{cs}\left(t\right)$ commutes with the linear
	flow. By construction, in this centre-stable space, $\left\Vert \hm{v}\left(t\right)\right\Vert _{\mathcal{H}}$
	is bounded or grows subexponentially for $t\geq0$.
	\begin{thm}
		\label{thm:StriGenerlLinear}Using the notations above and norms from \eqref{eq:strinorm},  one has Strichartz
		estimates
		\begin{equation}
		\left\Vert \hm{v}\right\Vert _{S_{\mathcal{H}}\bigcap L_{t}^{2}\mathcal{D}^{\frac{\nu}{2}}\mathcal{H}\left\langle \cdot-y(t)\right\rangle ^{-\sigma}}\lesssim\left\Vert \hm{v}\left(0\right)\right\Vert _{\mathcal{H}}+\left\Vert \hm{F}\right\Vert _{L_{t}^{2}\left(\mathcal{H}^{1-\frac{\nu}{2}}\left\langle \cdot-y(t)\right\rangle ^{\sigma}\bigcap\mathcal{B}_{\mathcal{H}}^{*}\right)+L_{t}^{1}\mathcal{H}}.\label{eq:Strilineargeneral}
		\end{equation}
		Moreover, indeed $\hm{v}$ scatters to a free wave. There exists $\hm{\phi}_{+}\in\mathcal{H}$
		such that
		\[
		\left\Vert \hm{v}\left(t\right)-e^{\mathfrak{L}_{0}t}\hm{\phi}_{+}\right\Vert _{\mathcal{H}}\rightarrow\infty,\,\,t\rightarrow\infty.
		\]
	\end{thm}
	
	\begin{proof}
		Performing the reduction in Subsection \ref{subsec:Reduction}, from
		the general equation \eqref{eq:generallinearcom}, it is sufficient
		to consider
		\begin{align}
			\hm{u}_{t} & =\mathcal{L}_{\beta_{0}}\left(t\right)\hm{u}+a\left(t\right)\nabla\hm{u}+b\left(x-\beta_{0}t\right)\hm{u}+\pi_{cs}\left(t\right)\hm{F}\label{eq:zgeneral}
		\end{align}
		where $\mathcal{L}_{\beta_{0}}\left(t\right)$ uses the notation \eqref{eq:cLbeta} with
		$\beta$ replaced by $\beta_{0}$, 
		\[
		a\left(t\right)=\beta\left(t\right)-\beta_{0}+y'\left(t\right)-\beta\left(t\right)
		\]
		and
		\[
		\left|b\left(x-\beta_{0}t\right)\right|\leq\left\Vert \beta'\left(t\right)\right\Vert _{L^{1}}\left|\partial_{\beta}V_{\tilde{\beta}}\left(x-\beta_{0}t\right)\right|.
		\]
		Since $\left\Vert \hm{v}\left(t\right)\right\Vert _{\mathcal{H}}$
		is bounded or grows subexponentially for $t\geq0$, whence $\left\Vert \hm{u}\left(t\right)\right\Vert _{\mathcal{H}}$
		is of subexponential growth. 
		
		As computed in \S \ref{subsubsec:reductionzero},  given the zero-mode condition \eqref{eq:onePmainzero}, we can ignore the influence of zero modes with respect to $\mathcal{L}_{\beta_0}(t)$ in the decomposition of $\hm{u}$.  So we can write
		\[
		\hm{u}\left(t\right)=\sum_{k=1}^{K}\left(\lambda_{k,+}\left(t\right)\mathcal{Y}_{k,\beta_0}^{+}\left(t\right)+\lambda_{k,-}\left(t\right)\mathcal{Y}_{k,\beta_0}^{-}\left(t\right)\right)+\hm{\gamma}
		\]
		where the decomposition above is defined with respect to the linear
		operator $\mathcal{L}_{\beta_{0}}\left(t\right)$ with the potential
		moving with a fixed velocity $\beta_{0}$.  
		
		Applying the projection with respect to $\mathcal{L}_{\beta_{0}}\left(t\right)$
		and denoting $\hm{\gamma}=P_{c}\left(t\right)\hm{u}$, the equation
		\eqref{eq:zgeneral} is reduced to
		\begin{equation}
		\hm{\gamma}_{t}=\mathcal{L}_{\beta_{0}}\left(t\right)\hm{\gamma}+P_{c}\left(t\right)\left[a\left(t\right)\nabla\hm{u}+\tilde{F}\right]\label{eq:gammageneral}
		\end{equation}
		where
		\[
		\tilde{F}=\pi_{cs}\left(t\right)\hm{F}+b\left(x-\beta_{0}t\right)\hm{u}
		\]
		with $\hm{\gamma}\left(0\right)=P_{c}\left(0\right)\hm{u}\left(0\right)$.
		
		For the stable/unstable modes, we have
		\[
		\left(\partial_{t}\mp\frac{\nu_{k}}{\gamma}\right)\lambda_{k,\pm}=\omega\left(a\left(t\right)\nabla\hm{u}+\tilde{F},\mathcal{Y}_{k,\beta_0}^{\mp}\left(t\right)\right)
		\]
		where $\gamma_{0}=\frac{1}{\sqrt{1-\left|\beta_{0}\right|^{2}}}$
		and where $\mathcal{Y}_{k,\beta_0}^{\pm}$ is found with respect to $\mathcal{L}_{\beta_{0}}\left(t\right)$.
		Using the fact that $\left\Vert \hm{u}\right\Vert _{\mathcal{H}}$
		is bounded or grows subexponentially for $t\geq0$, we obtain the
		following system for $\lambda_{k,\pm}$:
		\[
		\lambda_{k,+}\left(t\right)=-\int_{t}^{\infty}e^{\frac{\nu_{k}}{\gamma}\left(t-s\right)}\omega\left(a\left(t\right)\nabla\hm{u}+\tilde{F},\mathcal{Y}_{k,\beta_0}^{-}\left(s\right)\right)\,ds
		\]
		\[
		\lambda_{k,-}\left(t\right)=e^{-\frac{\nu_{k}}{\gamma}t}\lambda_{k,-}\left(0\right)-\int_{0}^{t}e^{\frac{\nu_{k}}{\gamma}\left(s-t\right)}\omega\left(a\left(t\right)\nabla\hm{u}+\tilde{F},\mathcal{Y}_{k,\beta_0}^{+}\left(s\right)\right)\,ds.
		\]
		To estimate the $L_{t}^{2}\bigcap L_{t}^{\infty}$ norm of $\lambda_{k,\pm}$,
		we apply Young's inequality. For the centre-direction part given
		by the equation \eqref{eq:gammageneral}, we apply Proposition \ref{prop:Stri1}
		and conclude that
		\[
		\left\Vert \hm{\gamma}\right\Vert _{S_{\mathcal{H}}\bigcap L_{t}^{2}\mathcal{D}^{\frac{\nu}{2}}\mathcal{H}\left\langle \cdot-\beta_{0}t\right\rangle ^{-\sigma}}\lesssim\left\Vert \hm{\gamma}\left(0\right)\right\Vert _{\mathcal{H}}+\left\Vert \tilde{F}\right\Vert _{L_{t}^{2}\left(\mathcal{H}^{1-\frac{\nu}{2}}\left\langle \cdot-\beta_{0}t\right\rangle ^{\sigma}\bigcap\mathcal{B}_{\mathcal{H}}^{*}\right)+L_{t}^{1}\mathcal{H}}
		\]
		\begin{align*}
			\sum_{k=1}^{K}\left(\left\Vert \lambda_{k,+}\left(t\right)\right\Vert _{L_{t}^{2}\bigcap L_{t}^{\infty}}+\left\Vert \lambda_{k,-}\left(t\right)\right\Vert _{L_{t}^{2}\bigcap L_{t}^{\infty}}\right) & \lesssim\left\Vert \hm{u}\left(0\right)\right\Vert _{\mathcal{H}}+\delta\left\Vert \hm{\gamma}\right\Vert _{S_{\mathcal{H}}\bigcap L_{t}^{2}\mathcal{D}^{\frac{\nu}{2}}\mathcal{H}\left\langle \cdot-\beta_{0}t\right\rangle ^{-\sigma}}\\
			& +\left\Vert \pi_{cs}\left(t\right)\hm{F}\right\Vert _{L_{t}^{2}\left(\mathcal{H}^{1-\frac{\nu}{2}}\left\langle \cdot-\beta_{0}t\right\rangle ^{\sigma}\bigcap\mathcal{B}_{\mathcal{H}}^{*}\right)+L_{t}^{1}\mathcal{H}.}
		\end{align*}
		Therefore, summing two estimates above, we obtain
		\[
		\left\Vert \hm{u}\right\Vert _{S_{\mathcal{H}}\bigcap L_{t}^{2}\mathcal{D}^{\frac{\nu}{2}}\mathcal{H}\left\langle \cdot-\beta_{0}t\right\rangle ^{-\sigma}}\lesssim\left\Vert \hm{u}\left(0\right)\right\Vert _{\mathcal{H}}+\left\Vert \pi_{cs}\left(t\right)\hm{F}\right\Vert _{L_{t}^{2}\left(\mathcal{H}^{1-\frac{\nu}{2}}\left\langle \cdot-\beta_{0}t\right\rangle ^{\sigma}\bigcap\mathcal{B}_{\mathcal{H}}^{*}\right)+L_{t}^{1}\mathcal{H}}.
		\]
		Going back to the original coordinate, we get
		\[
		\left\Vert \hm{v}\right\Vert _{S_{\mathcal{H}}\bigcap L_{t}^{2}\mathcal{D}^{\frac{\nu}{2}}\mathcal{H}\left\langle \cdot-y(t)\right\rangle ^{-\sigma}}\lesssim\left\Vert \hm{v}\left(0\right)\right\Vert _{\mathcal{H}}+\left\Vert \pi_{cs}\left(t\right)\hm{F}\right\Vert _{L_{t}^{2}\left(\mathcal{H}^{1-\frac{\nu}{2}}\left\langle \cdot-y(t)\right\rangle ^{\sigma}\bigcap\mathcal{B}_{\mathcal{H}}^{*}\right)+L_{t}^{1}\mathcal{H}}
		\]
		which is equivalent to \eqref{eq:Strilineargeneral}.
		
		Finally we show $\hm{v}$ scatters to the free wave. The argument is more
		or less standard, see Chen \cite{C1,C2} for the wave setting. We write the solution as
		\[
		\hm{v}(t)=e^{\mathfrak{L}_{0}t}\hm{v}\left(0\right)+\int_{0}^{t}e^{\left(t-s\right)\mathfrak{L}_{0}}\left\{ \mathrm{V}_{\beta\left(s\right)}\left(\cdot-y\left(s\right)\right)\hm{v}\left(z\right)+\pi_{cs}\left(t\right)\hm{F}\right\} \,ds.
		\]
		By the standard argument, to establish the  scattering, we need to show
		\[
		\int_{0}^{\infty}e^{-s\mathfrak{L}_{0}}\left\{ \mathrm{V}_{\beta\left(s\right)}\left(\cdot-y\left(s\right)\right)\hm{v}\left(z\right)+\pi_{cs}\left(t\right)\hm{F}\right\} \in\mathcal{H}.
		\]
		Applying the dual version of \eqref{eq:shiftfreehom} in Lemma \ref{lem:shiftslant},
		we obtain that
		\begin{align*}
			\left\Vert \int_{0}^{\infty}e^{-s\mathfrak{L}_{0}}\left\{ \mathrm{V}_{\beta\left(s\right)}\left(\cdot-y\left(s\right)\right)\hm{v}\left(z\right)+\pi_{s}\left(t\right)\hm{F}\right\} ds\right\Vert _{\mathcal{H}} & \lesssim\left\Vert \hm{F}\right\Vert _{L_{t}^{2}\left(\mathcal{H}^{1-\frac{\nu}{2}}\left\langle \cdot-y(t)\right\rangle ^{\sigma}\bigcap\mathcal{B}_{\mathcal{H}}^{*}\right)+L_{t}^{1}\mathcal{H}}\\
			& +\left\Vert \hm{v}\right\Vert _{S_{\mathcal{H}}\bigcap L_{t}^{2}\mathcal{D}^{\frac{\nu}{2}}\mathcal{H}\left\langle \cdot-y(t)\right\rangle ^{-\sigma}}\\
			& \lesssim\left\Vert \hm{v}\left(0\right)\right\Vert _{\mathcal{H}}+\left\Vert \hm{F}\right\Vert _{L_{t}^{2}\left(\mathcal{H}^{1-\frac{\nu}{2}}\left\langle \cdot-y(t)\right\rangle ^{\sigma}\bigcap\mathcal{B}_{\mathcal{H}}^{*}\right)+L_{t}^{1}\mathcal{H}}.
		\end{align*}
		Define
		\[
		\hm{\phi}_{+}=\hm{v}\left(0\right)+\int_{0}^{\infty}e^{-s\mathfrak{L}_{0}}\left\{ \mathrm{V}_{\beta\left(s\right)}\left(\cdot-y\left(s\right)\right)\hm{v}\left(z\right)+\pi_{cs}\left(t\right)\hm{F}\right\} 
		\]
		then by construction,
		\[
		\left\Vert \hm{v}\left(t\right)-e^{\mathfrak{L}_{0}t}\hm{\phi}_{+}\right\Vert _{\mathcal{H}}\rightarrow\infty,\,\,t\rightarrow\infty
		\]
		as desired.
	\end{proof}

	\section{Analysis of the multi-potential problem}\label{sec:multi}
	
	In this section, we analyze and establish Strichartz estimates for \eqref{eq:maineq} under the assumptions \eqref{eq:maincond} and \eqref{eq:maincondzero}.
	
	
	Before going to the detailed analysis, we first do some simplification
	and reduction.
	\begin{itemize}
		\item Setting $\beta_{j}\left(0\right)=\beta_{j}$ by the same reduction
		argument in Subsection \ref{subsec:Reduction}, we will consider Strichartz
		estimates for the following problem
		\begin{equation}
		w_{tt}-\Delta w+w+\sum_{j=1}^{J}(V_{j})_{\beta_{j}}\left(\cdot-\beta_{j}t+c_{j}\left(t\right)\right)w=F\label{eq:simplifiedeq}
		\end{equation}
		with $\left|c'_{j}\left(t\right)\right|\ll1$ after projecting on
		the centre-stable space. 
		\item In additionally, we assume that all potentials are compactly supported.
		It is easy to check one can pass to the limit to obtain the same results
		for fast-decaying potentials.
		\item Due the orthogonality conditions on zero modes, as computed in \S \ref{subsubsec:reductionzero}, we can further ignore the contribution from  $\mathcal{P}_{0,\beta_{j}}(c_j(t))\hm{w}$.
	\end{itemize}
	\subsection{Basic setting} Consider the simplified model \eqref{eq:simplifiedeq}. 
	Using the Hamiltonian formalism, we write $\hm{w}=\left(\begin{array}{c}
	w\\
	w_{t}
	\end{array}\right)$ and
	\begin{equation}
	\hm{w}_{t}=\mathcal{L}\left(t\right)\hm{w}+\hm{F}\label{eq:complexgamma}
	\end{equation}
	where
	\[
	\mathcal{L}\left(t\right):=\left(\begin{array}{cc}
	0 & 1\\
	\Delta-1-\sum_{j=1}^{J}(V_{j})_{\beta_{j}}\left(\cdot-\beta_{j}t+c_{j}\left(t\right)\right) & 0
	\end{array}\right).
	\]
	We set for $j=1,\ldots J$
	\begin{equation}
	\mathcal{L}_{j}=\left(\begin{array}{cc}
	0 & 1\\
	\Delta-1-V_{j} & 0
	\end{array}\right)\label{eq:Lij}
	\end{equation}
\begin{equation}\label{eq:Vmatrix}
	\hm{{\rm V}}_{j}=\left(\begin{array}{cc}
	0 & 0\\
	-\left(V_{j}\right)_{\beta_{j}} & 0
	\end{array}\right)
\end{equation}
	and
	\begin{equation}
	\mathcal{L}_{j}\left(t\right):=\left(\begin{array}{cc}
	0 & 1\\
	\Delta-1-\left(V_{j}\right)_{\beta_{j}}\left(x-\beta_{j}t\right) & 0
	\end{array}\right).\label{eq:L2t}
	\end{equation}
	Denote $\hm{v}\left(t\right)=\pi_{cs}(t)\hm{w}$ where $\pi_{cs}(t)$
	is the projection onto the  centre-stable space, see Proposition \ref{prop:expdicho}. Then we need to establish
	Strichartz estimates for
	\begin{equation}
	\hm{v}_{t}=\mathcal{L}\left(t\right)\hm{v}+\pi_{cs}\left(t\right)\hm{F}.\label{eq:complexv}
	\end{equation}
	By construction, in this stable space, $\left\Vert \hm{v}\left(t\right)\right\Vert _{\mathcal{H}}$
	is bounded or grows subexponentially for $t\geq0$. In other words, we have
\begin{cor}
    For any $\epsilon>0$, there exists a constant $C_\epsilon$ such that for $t\geq0$ the solution to \eqref{eq:complexv}, $\hm{v}(t)=\pi_{cs}\hm{v}(t)$,  satisfies
    \begin{equation}\label{eq:centergrowth}
       \left\Vert \hm{v}\left(t\right)\right\Vert _{\mathcal{H}}\leq C_\epsilon e^{\epsilon t} \left\Vert \hm{v}\left(0\right)\right\Vert _{\mathcal{H}}.
    \end{equation}
\end{cor}

	We need to understand the local decay  near $\left(V_{j}\right)_{\beta_{j}}\left(x-\beta_{j}t+c_{j}\left(t\right)\right).$
	More precisely, we need to show
	\[
	\int\int\left\langle x-\beta_{j}t+c_{j}\left(t\right)\right\rangle ^{-\sigma}\left(|\mathcal{D}^{1-\frac{\nu}{2}} v_{1}|^{2}\left(t,x\right)+|\mathcal{D}^{-\frac{\nu}{2}}v_{2}|^{2}\left(t,x\right)\right)\,dxdt<\infty,\,j=1\ldots J.
	\]
	To establish them, we need the following ingredients: 
	\begin{enumerate}
		\item We decompose the space into $J+1$ channels. There are $J$ channels
		associated with those $J$ potentials. In these channels, the evolution
		is mainly influenced by the associated potential. In the remaining
		channel away from centers of potentials, the evolution is dominated
		by the free evolution.
		\item We need localized analysis of $T$, $T_{1}$ and $T_{0}$ from \eqref{eq:T},
		\eqref{eq:T0} and \eqref{eq:T1} for different channels.
		\item We need some truncated version of the inhomogeneous estimates like
		Chen \cite{C1} in order to apply the channel decomposition argument, bootstrapping and understand the interaction among potentials. Here we need  very general interaction estimates.
	\end{enumerate}
	
 	\subsection{Decomposition into channels}
	
	Let $\chi\left(x,r\right)$ be an non-negative smooth function such
	that $\chi\equiv1$ for $\left|x\right|\leq r$ and $\chi\equiv0$
	for $\left|x\right|\geq2r$. Then we use the following smooth functions
	to decompose the space $\mathbb{R}^{3}$ into $J+1$ regions for $0<\epsilon\ll1$,
	define for $j=1,\ldots J$
	\[
	\chi_{j}\left(t\right)=\chi\left(x-\beta_{j}t+c_{j}\left(t\right),\epsilon t\right)
	\]
	and
	\[
	\chi_{0}=1-\sum_{j=1}^{J}\chi_{j}\left(t\right).
	\]
	Clearly we can write
	\[
	\hm{v}=\sum_{j=0}^{J}\chi_{j}\hm{v}=:\sum_{j=0}^{J}\hm{v}_{j}.
	\]
	Trivially, the behaviors of $\hm{v}_{j}$ and $\hm{v}_{0}$ will be
	dominated by $\left(V_{j}\right)_{\beta_{j}}\left(x-\beta_{j}t+c_{j}\left(t\right)\right)$
	and the free evolution respectively.
	
	For any given finite time, one can always use Gronwall's inequality to control
	everything. So we start our analysis from $t=B$ which is finite but
	large enough such that $\chi_{j}\equiv1$ in the support $V_{j}\left(t\right)$
	for $j=1,\ldots J$, i.e., $\chi_j(t)V_j(t)=V_j(t)$ for $t\geq B$.
	
	For a fixed but large $T$, we use the bootstrap assumptions for $j=1,\ldots J$
	\begin{equation}
	\left(\int_{0}^{T}\int\left\langle x-\beta_{j}t+c_{j}\left(t\right)\right\rangle ^{-2\sigma}\left(|\mathcal{D}^{1-\frac{\nu}{2}} v_{1}|^{2}\left(t,x\right)+|\mathcal{D}^{-\frac{\nu}{2}}v_{2}|^{2}\left(t,x\right)\right)\,dxdt\right)^{\frac{1}{2}}<C_{j,T}.\label{eq:assumj}
	\end{equation}
The assumptions above in particular implies
 \begin{equation}\label{eq:energyCT}
      \sup_{t\in[0,T]}\left\Vert \hm{v}\left(t\right)\right\Vert _{\mathcal{H}}\leq C_0 \sum_{j=1}^J C_{j,T}
 \end{equation}for some constant $C_0$ only depending on prescribed constants. This implication is an special consequence Strichartz estimates. Since later on, we will use weighted estimates to derive Strichartz estimates in general settings, we omit the proof of \eqref{eq:energyCT} above.
 
	We will also use that notation
	\[
	L_{t}^{p}L_{x}^{q}:=L^{p}\left([0,T];L_{x}^{q}\right).
	\]
\subsection{Analysis of stable/unstable modes}
Before we analyze each channels, we first analyze the behavior of stable/unstable modes in the centre-stable direction. Consider the solution $\hm{v}(t)$ to  equation \eqref{eq:complexv}.	To analyze the bounds for the discrete modes, taking the inner product
	\[
	\lambda_{j,k,\pm}\left(t\right)=\omega\left(\hm{v},\mathcal{Y}_{j,k,\beta_j}^{\mp}\left(\cdot-\beta_jt+c_j(t)\right)\right)
	\]
	we have
\begin{align}
	\left(\partial_{t}\mp\frac{\nu_{j.k}}{\gamma_{j}}\right)\lambda_{j,k,\pm}\left(t\right)&=\omega\left(a_{j}\left(t\right)\nabla\hm{v},\mathcal{Y}_{j,k,\beta_j}^{\mp}\left(\cdot-\beta_jt+c_j(t)\right)\right)\\&+\omega\left(\pi_{cs}(t)\hm{F},\mathcal{Y}_{j,k,\beta_j}^{\mp}\left(\cdot-\beta_jt+c_j(t)\right)\right)\\&+\sum_{i\neq j}\omega\left(\hm{{\rm V}}_{i}\big(\cdot-\beta_it+c_i(t)\big)\hm{v},\mathcal{Y}_{j,k,\beta_j}^{\mp}\left(\cdot-\beta_jt+c_j(t)\right)\right)
\end{align}
	where $a_j(t)=c_j(t)'$.   We first focus on the behavior of unstable modes, $\lambda_{j,k,+}(t)$.
	Using the fact that $\left\Vert \hm{v}\right\Vert _{\mathcal{H}}$
	is bounded or grows subexponentially for $t\geq0$, we obtain the
	following system
	\begin{align}
	\lambda_{j,k,+}\left(t\right)&=-\int_{t}^{\infty}e^{\frac{\nu_{j.k}}{\gamma_{j}}\left(t-s\right)}\omega\left(a_{j}\left(s\right)\nabla\hm{v}+\sum_{i\neq j}\hm{{\rm V}}_{i}\big(\cdot-\beta_is+c_i(s)\big)\hm{v},\mathcal{Y}_{j,k,\beta_j}^{-}\left(\cdot-\beta_js+c_j(s)\right)\right)\,ds\\
	&-\int_{t}^{\infty}e^{\frac{\nu_{j.k}}{\gamma_{j}}\left(t-s\right)}\omega\left(\pi_{cs}(s)\hm{F},\mathcal{Y}_{j,k,\beta_j}^{-}\left(\cdot-\beta_js+c_j(s)\right)\right)\,ds.
\end{align}
Due to the separation of trajectories, without loss of generality, we can assume that there exist $\varpi$ small and  $L$ large such that
\begin{equation}
    |\beta_jt+ c_j(t)-(\beta_it+c_i(t)|\geq \varpi t+L,\,j\neq i
\end{equation}
 and we can always achieve this by restricting onto large time independent of $\bs v$. Since $\hm{v}(t)=\pi_{cs}\hm{v}(t)$, we take $0<\epsilon\ll \varpi$ such that \eqref{eq:centergrowth} holds
 \begin{equation}
      \left\Vert \hm{v}\left(t\right)\right\Vert _{\mathcal{H}}\leq C_\epsilon e^{\epsilon t}\left( \left\Vert \hm{v}\left(0\right)\right\Vert _{\mathcal{H}}+\left\Vert F\right\Vert _{L_{t}^{2}\mathscr{W}_{\mathcal{H}}+L_{t}^{1}\mathcal{H}}\right)
 \end{equation}
 	where
	\[
	\mathscr{W}=\bigcap_{j=1}^{J}\mathcal{D}^{-\frac{\nu}{2}}L_{x}^{2}\left\langle \cdot-\beta_{j}t+c_{j}\left(t\right)\right\rangle ^{\sigma}\bigcap B_{6/5,2}^{-5/6}.
	\]
Note that by the decay of eigenfunctions
	given by Agmon's estimate and the decay of the potential,t hen one has
\begin{align}
    \left|\omega\left(\sum_{i\neq j}\hm{{\rm V}}_{i}\big(\cdot-\beta_is+c_i(s)\big)\hm{v},\mathcal{Y}_{j,k,\beta_j}^{-}\left(\cdot-\beta_js+c_j(s)\right)\right)\right|\\\lesssim C_\epsilon e^{(-\frac{\varpi}{2}+\epsilon)s} \left( \left\Vert \hm{v}\left(0\right)\right\Vert _{\mathcal{H}}+\left\Vert F\right\Vert _{L_{t}^{2}\mathscr{W}_{\mathcal{H}}+L_{t}^{1}\mathcal{H}}\right).
\end{align}
Therefore, we know
{\small \begin{align}
   \left| \int_{t}^{\infty}e^{\frac{\nu_{j.k}}{\gamma_{j}}\left(t-s\right)}\omega\left(\sum_{i\neq j}\hm{{\rm V}}_{i}\big(\cdot-\beta_is+c_i(s)\big)\hm{v},\mathcal{Y}_{j,k,\beta_j}^{-}\left(\cdot-\beta_js+c_j(s)\right)\right)\,ds \right|\\\lesssim C_\epsilon e^{(-\frac{\varpi}{2}+\epsilon)t}  \left( \left\Vert \hm{v}\left(0\right)\right\Vert _{\mathcal{H}}+\left\Vert F\right\Vert _{L_{t}^{2}\mathscr{W}_{\mathcal{H}}+L_{t}^{1}\mathcal{H}}\right).
\end{align}}

Then we split the integral due to the nonlinear behavior of trajectories into two pieces:
\begin{align}
     -\int_{t}^{\infty}e^{\frac{\nu_{j.k}}{\gamma_{j}}\left(t-s\right)}\omega\left(a_{j}\left(s\right)\nabla\hm{v},\mathcal{Y}_{j,k,\beta_j}^{-}\left(\cdot-\beta_js+c_j(s)\right)\right)\,ds \\=   -\int_{t}^{T}e^{\frac{\nu_{j.k}}{\gamma_{j}}\left(t-s\right)}\omega\left(a_{j}\left(s\right)\nabla\hm{v},\mathcal{Y}_{j,k,\beta_j}^{-}\left(\cdot-\beta_js+c_j(s)\right)\right)\,ds\label{eq:leqT}\\
       -\int_{T}^{\infty}e^{\frac{\nu_{j.k}}{\gamma_{j}}\left(t-s\right)}\omega\left(a_{j}\left(s\right)\nabla\hm{v},\mathcal{Y}_{j,k,\beta_j}^{-}\left(\cdot-\beta_js+c_j(s)\right)\right)\,ds\label{eq:geqT}
\end{align}
For the analysis of \eqref{eq:leqT}, it will be similar to the analysis of stable modes and bootstrap assumptions can be directly used. 

The analysis of \eqref{eq:geqT} could not use the bootstrap assumptions directly since it requires information beyond $T$.To estimate this piece, we use the fact $\hm{v}$ is in the centre-stable direction and can only grow mildly. From \eqref{eq:energyCT}, one has
for $s\geq T$
\begin{equation}
      \left\Vert \hm{v}\left(s\right)\right\Vert _{\mathcal{H}}\leq C_\epsilon \big( C_0 \sum_{j=1}^J C_{j,T}+\left\Vert F\right\Vert _{L_{t}^{2}\mathscr{W}_{\mathcal{H}}+L_{t}^{1}\mathcal{H}}\big)e^{\epsilon (s-T)},
 \end{equation}
which implies
 \begin{align}
     \left| 
    -\int_{T}^{\infty}e^{\frac{\nu_{j.k}}{\gamma_{j}}\left(t-s\right)}\omega\left(a_{j}\left(s\right)\nabla\hm{v},\mathcal{Y}_{j,k,\beta_j}^{-}\left(\cdot-\beta_js+c_j(s)\right)\right)\,ds
\right|\\\leq \delta C_\epsilon \big( C_0 \sum_{j=1}^J C_{j,T}+\left\Vert F\right\Vert _{L_{t}^{2}\mathscr{W}_{\mathcal{H}}+L_{t}^{1}\mathcal{H}}\big) e^{\frac{\nu_{j.k}}{\gamma_{j}}\left(t-T\right)}.
 \end{align}
Therefore, upon $L^2$ integration in $t$, we conclude
\begin{align}
  \left\Vert   
    -\int_{T}^{\infty}e^{\frac{\nu_{j.k}}{\gamma_{j}}\left(t-s\right)}\omega\left(a_{j}\left(s\right)\nabla\hm{v},\mathcal{Y}_{j,k,\beta_j}^{-}\left(\cdot-\beta_js+c_j(s)\right)\right)\,ds\right\Vert_{L^2[0,T]}\\ \lesssim \delta C_\epsilon \big( C_0 \sum_{j=1}^J C_{j,T}+\left\Vert F\right\Vert _{L_{t}^{2}\mathscr{W}_{\mathcal{H}}+L_{t}^{1}\mathcal{H}}\big)
    \end{align}
and taking $\delta$ small, the estimate for \eqref{eq:geqT} recovers the bootstrap assumption.

The analysis of \eqref{eq:leqT} is similar to the analysis of the stable modes. We first consider the stable modes which can be written down explictly as
	\begin{align}
	\lambda_{j,k,-}\left(t\right)&=e^{-\frac{\nu_{j.k}}{\gamma_{j}}t}\lambda_{j,k,-}\left(0\right)\\&+\int_{0}^{t}e^{\frac{\nu_{j.k}}{\gamma_{j}}\left(s-t\right)}\omega\left(a_{j}\left(s\right)\nabla\hm{v}+\sum_{i\neq j}\hm{{\rm V}}_{i}\big(\cdot-\beta_is+c_i(s)\big)\hm{v},\mathcal{Y}_{j,k,\beta_j}^{-}\left(\cdot-\beta_js+c_j(s)\right)\right)\,ds\\
	&+\int_{0}^{t}e^{\frac{\nu_{j.k}}{\gamma_{j}}\left(s-t\right)}\omega\left(\pi_{cs}(s)\hm{F},\mathcal{Y}_{j,k,\beta_j}^{-}\left(\cdot-\beta_js+c_j(s)\right)\right)\,ds.
\end{align}
	To estimate $L_{t}^{2}\bigcap L_{t}^{\infty}$ norm of $\lambda_{j,k,-}$ and \eqref{eq:leqT},
	we apply Young's inequality. Again using  the decay of eigenfunctions
	given by Agmon's estimate and the decay of the potential, for $s\leq T$. we always
	have that
	\[
	 \left|\omega\left(\sum_{i\neq j}\hm{{\rm V}}_{i}\big(\cdot-\beta_is+c_i(s)\big)\hm{v},\mathcal{Y}_{j,k,\beta_j}^{-}\left(\cdot-\beta_js+c_j(s)\right)\right)\right|\\\lesssim e^{-\frac{\varpi}{2}s-\alpha L} \sum_i C_{i,T}.
	\]
for some $\alpha>0$ from the decay of eigenfunctions and potentials. Then we note that by bootstrap assumptions, again  for $s\leq T$,
\begin{equation}
    \left\Vert\omega\left(a_{j}\left(s\right)\nabla\hm{v},\mathcal{Y}_{j,k,\beta_j}^{\pm}\left(\cdot-\beta_js+c_j(s)\right)\right)\right\vert_{L^\infty_t\bigcap L^2_t}\lesssim \delta C_{j,T}.
\end{equation}
Therefore, putting all computations above together, for discrete modes, one has
	{\small\begin{align}
			\sum_{j=1}^J\sum_{k=1}^{K_{j}}\left(\left\Vert \lambda_{j,k,+}\left(t\right)\right\Vert _{L_{t}^{2}\bigcap L_{t}^{\infty}}+\left\Vert \lambda_{j,k,-}\left(t\right)\right\Vert _{L_{t}^{2}\bigcap L_{t}^{\infty}}\right) & \lesssim \sum_i \left(e^{-\alpha L}+\delta\right)C_{i,T}+\left\Vert \hm{v}\left(0\right)\right\Vert _{\mathcal{H}}\label{eq:boudbound}\\
			&+\left\Vert F\right\Vert _{L_{t}^{2}\mathscr{W}_{\mathcal{H}}+L_{t}^{1}\mathcal{H}}.\nonumber 
	\end{align}}
Therefore, taking $L$ large and $\delta$ small, estimates for discrete modes recover bootstrap assumption.
	\begin{rem}\label{rem:proplinear}
		Given the conditions of  Proposition	\ref{pro:linear}, the analysis of discrete modes are trivial. One can conclude the desired estimates above directly from the analysis of continuous spectrum which we will presented below
	\end{rem}	
\begin{rem}
     Consider the homogeneous equation
	\begin{equation}
	\hm{v}_{t}=\mathcal{L}\left(t\right)\hm{v}\label{eq:vhomorem}
	\end{equation}
and assume that $\hm{v}(t)=\pi_{cs}\hm{v}(t)$. 
From the argument above if we indeed have linear trajectories, i.e., $c_j'(t)=a_j(t)\equiv0$, the unstable modes of  flows in the centre-stable direction decay exponentially.
\end{rem}	
	\subsection{Analysis of the $j$th channel}\label{subsec:analysisJchannel}
	
	The analysis will be similar to Subsection \ref{subsec:Plineartraj}.
	Denote $\tilde{\hm{v}}\left(t,y\right)=\hm{v}\left(t,y-c_{j}\left(t\right)\right)$
	then one has
	\begin{equation}\label{eq:jthequation}
	\tilde{\hm{v}}_{t}=\mathcal{L}_{j}\left(t\right)\tilde{\hm{v}}+\left[c_{j}'\left(t\right)\cdot\nabla\tilde{\hm{v}}+\hm{F}_{j}\right]
	\end{equation}
	where $\mathcal{L}_{j}\left(t\right)=\mathfrak{L}_{0}+\hm{{\rm V}}_{j}\left(x-\beta_{j}t\right)$, see \eqref{eq:freematrix} for the definition of $\mathfrak{L}_0$,
	and
	\begin{equation}
	\hm{F}_{j}=\sum_{\rho\neq j}\hm{{\rm V}}_{\rho}\left(x-\beta_{\rho}t+c_{\rho}\left(t\right)-c_{j}\left(t\right)\right)\tilde{\hm{v}}\left(t\right)+\hm{F}\left(t\right).\label{eq:Fj}
	\end{equation}

	\subsubsection{Analysis of the continuous spectrum}
	
	Denoting
	\[
	A_{j}\left(t\right)=:a_{j}\left(t\right)\nabla=c'_{j}\left(t\right)\nabla,
	\]
	and projecting onto the continuous spectrum with respect to $\mathcal{L}_{j}(t)$,
	the original equation can be written as
	\[
	\hm{\varphi}_{t}=\mathcal{L}_{j}\left(t\right)\hm{\varphi}+P_{j,c}\left(t\right)\left[A_{j}\hm{\varphi}+\mathrm{\mathfrak{F}}_{j}\right]
	\]
	with $\hm{\varphi}\left(B\right)=P_{j,c}\left(B\right)\hm{v}\left(B,x-c_{j}\left(B\right)\right)$
	and $\mathfrak{F}_{j}=A_{j}P_{j,d}\left(t\right)\tilde{\hm{v}}+\hm{F}_{j}.$
	
	As in Subsection \ref{subsec:Plineartraj}, see equation \eqref{eq:pertlinearMmo},
	we modify the equation as\begin{equation}
	z_{t}=\left[\mathfrak{L}_{0}+\mathrm{V}_{j}\left(t\right)+A_{j}\left(t\right)\right]z+\tilde{F}\label{eq:modifiedz}
	\end{equation}
	where
	\begin{equation}
	\mathrm{V}_{j}\left(t\right)=-\mathcal{L}_{j}\left(t\right)P_{j,d}\left(t\right)-\kappa P_{j,d}\left(t\right)+\mathcal{K}_{j}\left(x-\beta_{j}t\right)\label{eq:Vjt}
	\end{equation}
	($\kappa>\frac{\nu_{j,K_{j}}}{\sqrt{1-\beta_{j}^{2}}}$ is sufficient
	for our setting),   $\mathcal{K}_j\left(x-\beta_{j} t\right)=\left(\begin{array}{c}
	0\\
	(V_j){\beta_j}\left(x-\beta_j t\right)
	\end{array}\right)$ and
	\begin{equation}
	\tilde{F}\left(t\right)=-A_{j}\left(t\right)P_{j,d}\left(t\right)z\left(t\right)+\mathfrak{F}_{j}.\label{eq:frakF}
	\end{equation}
	Let $\mathrm{U}_{j}\left(t,s\right)$ be the evolution operator defined
	by the equation $u_{t}=A_{j}\left(t\right)u$. By Duhamel formula,
	the solution to \eqref{eq:modifiedz} can be written as
	\begin{equation}
	z=z_{0}+\int_{0}^{t}e^{\left(t-s\right)\mathfrak{L}_{0}}\mathrm{U}_{j}\left(t,s\right)\left\{ \mathrm{V}_{j}\left(s\right)z\left(s\right)+\tilde{F}\left(s\right)\right\} \,ds.\label{eq:-11-1-1}
	\end{equation}
	Restricting onto the support of $\chi_{j}$, we analyze
	\[
	\chi_{j}z\left(t\right)=\chi_{j}z_{0}\left(t\right)+\chi_{j}\int_{B}^{t}e^{\left(t-s\right)\mathfrak{L}_{0}}\mathrm{U}_{j}\left(t,s\right)\left\{ \mathrm{V}_{j}z\left(s\right)+\tilde{F}\left(s\right)\right\} \,ds
	\]
	where
	\begin{equation}
	z_{0}\left(t\right)=e^{\left(t-B\right)\mathfrak{L}_{0}}\mathrm{U}_{j}\left(t,B\right)z\left(B\right).\label{eq:-12-1-1}
	\end{equation}
	To analyze the local decay of $\chi_{j}z$, we proceed as Subsection
	\ref{subsec:Plineartraj}. As before, we define
	\begin{equation}
	T^{B,j}\hm{g}=\chi_{j}\int_{B}^{t}e^{\left(t-s\right)\mathfrak{L}_{0}}\mathrm{U}_{j}\left(t,s\right)\mathrm{V}_{j}\left(s\right)\hm{g}\left(s\right)\,ds\label{eq:Tbj}
	\end{equation}
	\begin{equation}
	T_{0}^{B,j}\hm{g}=\chi_{j}\int_{B}^{t}e^{i\left(t-s\right)\mathfrak{L}_{0}}\mathrm{V}_{j}\left(s\right)\hm{g}\left(s\right)\,ds\label{eq:Tbj0}
	\end{equation}
	\begin{equation}
	T_{1}^{B,j}\hm{g}=\chi_{j}\int_{B}^{t}\mathrm{U}_{\beta}\left(t,s\right)\mathrm{V}_{j}\left(s\right)\hm{g}\left(s\right)\,ds\label{eq:Tbj1}
	\end{equation}
	where $\mathrm{U}_{j}\left(t,s\right)$ denote the evolution of
	\begin{equation}
	z_{t}=\left[\mathfrak{L}_0+\mathrm{V}_{j}\left(t\right)\right]z.\label{eq:-16-1-1}
	\end{equation}
	We observe that
	\begin{equation}
	\left(1-T^{B,j}\right)\chi_{j}z=\chi_{j}z_{0}+\chi_{j}\int_{B}^{t}e^{\left(t-s\right)\mathfrak{L}_{0}}\mathrm{U}_{j}\left(t,s\right)\tilde{F}\left(s\right)\,ds\label{eq:cha2T}
	\end{equation}
	since $\chi_{j}\mathrm{V}_{j}=\mathrm{V}_{j}$ from the initial reduction. 
	\begin{rem}
		Again, technically, in  $\mathrm{V}_{j}=-\mathcal{L}_{j}\left(t\right)P_{j,d}\left(t\right)-\kappa P_{j,d}\left(t\right)+\mathcal{K}_{j}\left(x-\beta_{j}t\right)$,
		the $P_{j,d}\left(t\right)$ part is not compactly supported. But one
		can cut off eigenfunctions and the error term decays exponentially.
		Then we can put the error term into $\tilde{F}$. 
	\end{rem}

	Moreover, as \eqref{eq:T0-T}, we also have
	\begin{equation}
	\left(T_{0}^{B,j}-T^{B,j}\right)\hm{g}=\int_{B}^{t}e^{\left(t-s\right)\mathfrak{L}_{0}}\left(1-\mathrm{U}_{j}\left(t,s\right)\right)\hm{g}\left(s\right)\,ds.\label{eq:TB0TB}
	\end{equation}
	Then we can perform the same analysis for $T^{B,j}$, $T_{0}^{B,j}$
	and $T_{1}^{B,j}$ as Subsection \ref{subsec:Plineartraj}.
	
	Adapting Lemma \ref{lem:T0T1} to the current setting, we have:
	\begin{lem}
		\label{lem:localTj}For $\sigma>1$, in the space $L_{t}^{2}\mathcal{H}\left\langle \cdot-\beta_{j}t\right\rangle ^{-\sigma}$,
		the following identities hold
		\[
		\left(1-T_{0}^{B,j}\right)\left(1+T_{1}^{B,j}\right)=\left(1+T_{1}^{B,j}\right)\left(1-T_{0}^{B,j}\right)=1.
		\]
	\end{lem}
	
	\begin{proof}
		The proof is the same as the regular version Lemma \ref{lem:T0T1}. Using
		the fact $\chi_{j}\mathrm{V}_{j}=\mathrm{V}_{j}$, one has
		\begin{align}
			T_{0}^{B,j}T_{1}^{B,j}f & =\chi_{j}\int\int_{0<t_{1}<t_{0}<t}e^{\left(t-t_{0}\right)\mathfrak{L}_{0}}\mathrm{V}_{j}\left(t_{0}\right)\chi_{j}\left(t_{0}\right)\mathrm{U}_{j}\left(t_{0},t_{1}\right)\mathrm{V}_{\beta}\left(t_{1}\right)f\left(t_{1}\right)\,dt_{1}dt_{0}\nonumber \\
			& =\int_{0}^{t}\left[\mathrm{U}_{j}\left(t,t_{1}\right)-e^{\left(t-t_{1}\right)\mathfrak{L}_{0}}\right]\mathrm{V}_{j}\left(t_{1}\right)f\left(t_{1}\right)\,dt_{1}=T_{1}^{B,j}f-T_{0}^{B,j}f\label{eq:T0T1B}
		\end{align}
		where in the last step, we used
		\begin{equation}
		\mathrm{U}_{j}\left(t,t_{1}\right)\phi=e^{\left(t-t_{1}\right)\mathfrak{L}_{0}}\phi+\int_{t_{1}}^{t}e^{\left(t-t_0\right)\mathfrak{L}_{0}}\mathrm{V}_{j}\left(t_0\right)\mathrm{U}_{j}\left(t_0,t_{1}\right)\phi\,dt_0.\label{eq:expUbeta-1}
		\end{equation}
		Reversing the Duhamel formula, we can also obtain
		\begin{equation}
		T_{1}^{B,j}T_{0}^{B,j}=-T_{0}^{B,j}+T_{1}^{B,j}\label{eq:-23-1}
		\end{equation}
		in the same manner. Therefore,
		\begin{equation}
		\left(1-T_{0}^{B,j}\right)\left(1+T_{1}^{B,j}\right)=\left(1+T_{1}^{B,j}\right)\left(1-T_{0}^{B,j}\right)=1.\label{eq:-24-1}
		\end{equation}
		The boundedness of these maps follows from the same argument as Lemma
		\ref{lem:T0T1} .
	\end{proof}
	Using Lemma \ref{lem:diffT}, for $\sigma>14$,
	\[
	\left\Vert \left\langle \cdot-\beta_{j}t\right\rangle ^{-\sigma}\left(T_{0}^{B,j}-T^{B,j}\right)\hm{g}\right\Vert _{L_{t}^{2}\mathcal{H}}\lesssim\left\Vert a_{j}\right\Vert _{L_{t}^{\infty}}^{\frac{1}{4}}\left\Vert \left\langle \cdot-\beta_{j}t\right\rangle ^{-\sigma}\hm{g}\right\Vert _{L_{t}^{2}\mathcal{H}}.
	\]
	Therefore, combining estimate above with Lemma \ref{lem:localTj},
	one can invert $\left(1-T^{B,j}\right)$ in \eqref{eq:cha2T} and estimate
	\begin{align*}
		\left\Vert \left\langle \cdot-\beta_{j}t\right\rangle ^{-\sigma}\mathcal{D}^{-\frac{\nu}{2}}\chi_{j}z\right\Vert _{L_{t}^{2}\mathcal{H}} & \lesssim\left\Vert \left\langle \cdot-\beta_{j}t\right\rangle ^{-\sigma}\mathcal{D}^{-\frac{\nu}{2}}\chi_{j}z_{0}\right\Vert _{L_{t}^{2}\mathcal{H}}\\
		& +\left\Vert \left\langle \cdot-\beta_{j}t\right\rangle ^{-\sigma}\mathcal{D}^{-\nu/2}\chi_{j}\int_{B}^{t}e^{\left(t-s\right)\mathfrak{L}_{0}}\mathrm{U}_{j}\left(t,s\right)\left(\tilde{F}\left(s\right)\right)\,ds\right\Vert _{L_{t}^{2}\mathcal{H}}.
	\end{align*}
	Clearly, by the homogeneous estimate \eqref{eq:shiftfreehom}, we have
	\begin{equation}
	\left\Vert \left\langle \cdot-\beta_{j}t\right\rangle ^{-\sigma}\mathcal{D}^{-\frac{\nu}{2}}\chi_{j}z_{0}\right\Vert _{L_{t}^{2}\mathcal{H}}\lesssim\left\Vert z\left(B\right)\right\Vert _{\mathcal{H}}\lesssim_{B}\left\Vert z\left(0\right)\right\Vert _{\mathcal{H}}.\label{eq:homoz0}
	\end{equation}
	For the inhomogeneous term, by construction \eqref{eq:Fj} and \eqref{eq:frakF},
	one has
	\begin{align}
		\left\Vert \left\langle \cdot-\beta_{j}t\right\rangle ^{-\sigma}\mathcal{D}^{-\nu/2}\chi_{j}\int_{B}^{t}e^{\left(t-s\right)\mathfrak{L}_{0}}\mathrm{U}_{j}\left(t,s\right)\left(\tilde{F}\left(s\right)\right)\,ds\right\Vert _{L_{t}^{2}\mathcal{H}}\label{eq:inhomoz1}\\
		\lesssim\left\Vert \left\langle \cdot-\beta_{j}t\right\rangle ^{-\sigma}\mathcal{D}^{-\nu/2}\chi_{j}\int_{B}^{t}e^{\left(t-s\right)\mathfrak{L}_{0}}\mathrm{U}_{j}\left(t,s\right)\left(-A_{j}P_{j,d}\left(t\right)z\left(s\right)\right)\,ds\right\Vert _{L_{t}^{2}\mathcal{H}}\label{eq:inhomoz2}\\
		+\left\Vert \left\langle \cdot-\beta_{j}t\right\rangle ^{-\sigma}\mathcal{D}^{-\nu/2}\chi_{j}\int_{B}^{t}e^{\left(t-s\right)\mathfrak{L}_{0}}\mathrm{U}_{j}\left(t,s\right)\left(\hm{F}\left(s\right)\right)\,ds\right\Vert _{L_{t}^{2}\mathcal{H}}\label{eq:inhomoz4}\\
		+\left\Vert \left\langle \cdot-\beta_{j}t\right\rangle ^{-\sigma}\mathcal{D}^{-\nu/2}\chi_{j}\int_{B}^{t}e^{\left(t-s\right)\mathfrak{L}_{0}}\mathrm{U}_{j}\left(t,s\right)\left(\mathfrak{K}_{j}\left(s\right)z\left(s\right)\right)\,ds\right\Vert _{L_{t}^{2}\mathcal{H}}\label{eq:inhomoz5}
	\end{align}
	where
	\[
	\mathfrak{K}_{j}\left(s\right):=\sum_{\rho\neq j}\mathcal{K}_{\rho}\left(x-\beta_{\rho}t+c_{\rho}\left(t\right)-c_{j}\left(t\right)\right).
	\]
	Then one can apply the inhomogenous estimate from Lemma \ref{lem:shiftslant}
	to bound
	\begin{align}
		\left\Vert \left\langle \cdot-\beta_{j}t\right\rangle ^{-\sigma}\mathcal{D}^{-\nu/2}\chi_{j}\int_{B}^{t}e^{\left(t-s\right)\mathfrak{L}_{0}}\mathrm{U}_{j}\left(t,s\right)\left(-A_{j}\left(t\right)P_{j,d}\left(t\right)z\left(t\right)\right)\,ds\right\Vert _{L_{t}^{2}\mathcal{H}}\nonumber \\
		\lesssim\left\Vert a_{j}\right\Vert _{L_{t}^{\infty}}\left\Vert \left\langle \cdot-\beta_{j}t\right\rangle ^{-\sigma}\mathcal{D}^{-\frac{\nu}{2}}z\right\Vert _{L_{t}^{2}\mathcal{H}}\label{eq:inhomoz3-1}
	\end{align}
	which can be absorbed to the left-hand side. 
	
	Secondly, again by the inhomogenous estimate from  Lemma \ref{lem:shiftslant},
	\begin{align}
		\left\Vert \left\langle \cdot-\beta_{j}t\right\rangle ^{-\sigma}\mathcal{D}^{-\nu/2}\chi_{j}\int_{B}^{t}e^{\left(t-s\right)\mathfrak{L}_{0}}\mathrm{U}_{j}\left(t,s\right)\left(F\left(s\right)\right)\,ds\right\Vert _{L_{t}^{2}\mathcal{H}}\label{eq:inhomoz4-1}\\
		\ \ \ \ \ \ \ \ \ \ \ \ \ \ \ \ \ \ \ \ \lesssim\left\Vert F\right\Vert _{L_{t}^{2}\mathcal{D}^{-\frac{\nu}{2}}\mathcal{H}\left\langle \cdot-\beta_{j}t\right\rangle ^{\sigma}+L_{t}^{1}\mathcal{H}}.\nonumber 
	\end{align}
	To analyze the term given by other potentials
	\[
	\left\Vert \left\langle \cdot-\beta_{j}t\right\rangle ^{-\sigma}\mathcal{D}^{-\nu/2}\chi_{j}\int_{B}^{t}e^{\left(t-s\right)\mathfrak{L}_{0}}\mathrm{U}_{j}\left(t,s\right)\left(\mathfrak{K}_{j}\left(s\right)z\left(s\right)\right)\,ds\right\Vert _{L_{t}^{2}\mathcal{H}},
	\]
	we notice that $\mathcal{K}_{j}$ is compactly supported around centers
	of other potentials, by the finite speed of propagation,
	\begin{align*}
		\chi_{j}\int_{B}^{t}e^{\left(t-s\right)\mathfrak{L}_{0}}\mathrm{U}_{j}\left(t,s\right)\left(\mathfrak{K}_{j}\left(s\right)z\left(s\right)\right)\,ds=	\chi_{j}\int_{B}^{t-M}e^{\left(t-s\right)\mathfrak{L}_{0}}\mathrm{U}_{j}\left(t,s\right)\left(\mathfrak{K}_{j}\left(s\right)z\left(s\right)\right)\,ds
	\end{align*}
	where $M>0$ can be picked to be large. Applying the truncated inhomogeneous
	estimates, Lemma \ref{lem:truncated},
	\begin{align}
		\left\Vert \left\langle \cdot-\beta_{j}t\right\rangle ^{-\sigma}\mathcal{D}^{-\nu/2}\chi_{j}\int_{B}^{t-M}e^{\left(t-s\right)\mathfrak{L}_{0}}\mathrm{U}_{j}\left(t,s\right)\left(\mathfrak{K}_{j}\left(s\right)z\left(s\right)\right)\,ds\right\Vert _{L_{t}^{2}\mathcal{H}}\nonumber \\
		\lesssim\sum_{\rho\neq j}\frac{1}{M^{\eta}}\left(\int_{0}^{T}\left\langle x-\beta_{\rho}t+c_{\rho}\left(t\right)-c_{j}\left(t\right)\right\rangle ^{-\sigma}\left|\mathcal{D}^{-\frac{\nu}{2}}(\mathcal{D}z_{1}+iz_{2})\left(t,x\right)\right|^{2}\,dxdt\right)^{\frac{1}{2}}\label{eq:inhomoz5-1}\\
		\leq\frac{1}{M^{\eta}}\sum_{\rho\neq j}C_{\rho,T}.\nonumber 
	\end{align}
	Overall, putting \eqref{eq:inhomoz1}, \eqref{eq:inhomoz3-1}, \eqref{eq:inhomoz4-1}
	and \eqref{eq:inhomoz5-1} together, we conclude that
	\begin{align}
		\left\Vert \left\langle \cdot-\beta_{j}t\right\rangle ^{-\sigma}\mathcal{D}^{-\frac{\nu}{2}}\chi_{j}z\right\Vert _{L_{t}^{2}\mathcal{H}} & \lesssim_{B}\left\Vert z\left(0\right)\right\Vert _{\mathcal{H}}+\delta C_{j,T}+\frac{1}{M^{\eta}}\sum_{\rho\neq j}C_{\rho,T}\label{eq:concl-1}\\
		& +\left\Vert\hm{F}\right\Vert _{L_{t}^{2}\mathcal{D}^{-\frac{\nu}{2}}\mathcal{H}\left\langle \cdot-\beta t\right\rangle ^{\sigma}+L_{t}^{1}\mathcal{H}}.\nonumber 
	\end{align}
	This finished the analysis of the centre-stable space part.

	\subsubsection{Conclusion for the $j$th channel}

	By the decomposition
	\[
	\chi_{j}\tilde{\hm{v}}=\chi_{j}\sum_{k=1}^{K_{j}}\left(\lambda_{j,k,+}\left(t\right)\mathcal{Y}_{j,k,\beta_j}^{+}(t)+\lambda_{j,k,-}\left(t\right)\mathcal{Y}_{j,k,\beta_j}^{-}(t)\right)+\chi_{j}P_{j,c}\left(t\right)\tilde{\hm{v}},
	\]
	estimates \eqref{eq:concl-1} and \eqref{eq:boudbound}, we obtain that
	\begin{align*}
		\left\Vert \left\langle \cdot-\beta_{j}t\right\rangle ^{-\sigma}\mathcal{D}^{-\frac{\nu}{2}}\chi_{j}\tilde{\hm{v}}\right\Vert _{L_{t}^{2}\mathcal{H}} & \lesssim\delta C_{j,T}+\left(\frac{1}{M^{\eta}}+e^{-\alpha L}+\delta\right)\sum_{\rho\neq j}C_{\rho,T}\\
		& +\left\Vert \tilde{\hm{v}}\left(0\right)\right\Vert _{\mathcal{H}}+\left\Vert\hm{ F}\right\Vert _{L_{t}^{2}\left(\mathcal{D}^{-\frac{\nu}{2}}L_{x}^{2}\left\langle \cdot-\beta_{j}t\right\rangle ^{\sigma}\bigcap B_{6/5,2}^{-5/6}\right)_{\mathcal{H}}+L_{t}^{1}\mathcal{H}}.
	\end{align*}
	Trivially, for some $\varepsilon>0$ and $\rho\neq j$
	\begin{align*}
		\left\Vert \left\langle \cdot-\beta_{\rho}t+c_{\rho}\left(t\right)-c_{j}\left(t\right)\right\rangle ^{-\sigma}\mathcal{D}^{-\frac{\nu}{2}}\chi_{j}\tilde{\hm{v}}\right\Vert _{L_{t}^{2}\mathcal{H}} & \lesssim B^{-\varepsilon}\left(\sum_{\rho=1}^{J}C_{\rho,T}\right)+\left\Vert \tilde{\hm{v}}\left(0\right)\right\Vert _{\mathcal{H}}\\
		& +\left\Vert\hm{F}\right\Vert _{L_{t}^{2}\left(\mathcal{D}^{-\frac{\nu}{2}}L_{x}^{2}\left\langle \cdot-\beta_{j}t\right\rangle ^{\sigma}\bigcap B_{6/5,2}^{-5/6}\right)_{\mathcal{H}}+L_{t}^{1}\mathcal{H}}.
	\end{align*}
	Therefore, in the $j$th channel, we recover the bootstrap condition  \eqref{eq:assumj}.

	\subsection{Analysis of the free channel}
	
	In this channel, the evolution is dominated by the free wave. More
	precisely,
	\[
	\chi_{0}\hm{v}=\chi_{0}e^{\left(t-B\right)\mathfrak{L}_{0}}\hm{v}\left(B\right)+\chi_{0}\int_{B}^{t}e^{\left(t-s\right)\mathfrak{L}_{0}}\left\{ \sum_{j=1}^{J}\hm{{\rm V}}_{j}\left(t\right)\hm{v}\left(s\right)+\hm{F}\left(s\right)\right\} \,ds.
	\]
	Again by the finite speed of propagation, one can find $M$ large,
	such that
	\[
	\chi_{0}\hm{v}=\chi_{0}e^{\left(t-B\right)\mathfrak{L}_{0}}\hm{v}\left(B\right)+\chi_{0}\int_{B}^{t-M}e^{\left(t-s\right)\mathfrak{L}_{0}}\left\{ \sum_{j=1}^{J}\hm{{\rm V}}_{j}\left(t\right)\hm{v}\left(s\right)+\hm{F}\left(s\right)\right\} \,ds.
	\]
	Applying Lemma \ref{lem:shiftslant} and Lemma \ref{lem:truncated}, one has
	\begin{align*}
		\sum_{j=1}^{J}\left\Vert \left\langle \cdot-\beta_{j}t+c_{j}\left(t\right)\right\rangle ^{-\sigma}\mathcal{D}^{-\nu/2}\chi_{0}\hm{v}\right\Vert _{L_{t}^{2}\mathcal{H}}\lesssim\left\Vert \hm{v}\left(B\right)\right\Vert _{L^{2}}\\
		+\frac{1}{M^{\eta}}\sum_{j=1}^{J}\left(\int_{0}^{T}\left\langle x-\beta_{j}t+c_{j}\left(t\right)\right\rangle ^{-\sigma}\left|\mathcal{D}^{-\frac{\nu}{2}}\left(\mathcal{D}v_{1}+iv_{2}\right)\right|^{2}\left(t,x\right)\,dxdt\right)^{\frac{1}{2}}\\
		<\left\Vert \hm{v}\left(B\right)\right\Vert _{\mathcal{H}}+\frac{1}{M^{\eta}}\sum_{j=1}^{J}C_{j,T}+\left\Vert F\right\Vert _{L_{t}^{2}\mathscr{W}_{\mathcal{H}}+L_{t}^{1}\mathcal{H}}.
	\end{align*}
	where
	\[
	\mathscr{W}=\bigcap_{j=1}^{J}\mathcal{D}^{-\frac{\nu}{2}}L_{x}^{2}\left\langle \cdot-\beta_{j}t+c_{j}\left(t\right)\right\rangle ^{\sigma}\bigcap B_{6/5,2}^{-5/6}
	\]
	Therefore, in this free channel, we recover again the bootstrap conditions.
	
	From our analysis above, we conclude the following proposition.
	\begin{prop}
		\label{prop:bootstrap}Consider the equation \eqref{eq:complexgamma}. Denote $\hm{v}\left(t\right)=\pi_{cs}(t)\hm{w}$ where $\pi_{cs}(t)$
		is the projection onto the  centre-stable space, see Proposition \ref{prop:expdicho}.  Given the bootstrap conditions \eqref{eq:assumj},
		one can obtain
		{		\footnotesize		\begin{align}
				\sum_{j=1}^{J}\left\Vert \left\langle \cdot-\beta_{j}t+c_{j}\left(t\right)\right\rangle ^{-\sigma}\mathcal{D}^{-\nu/2}\hm{v}\right\Vert _{L_{t}^{2}\mathcal{H}}
				<\left\Vert \hm{v}\left(B\right)\right\Vert _{\mathcal{H}}+\left(\frac{1}{M^{\eta}}+e^{-\alpha B}+\delta\right)\left(\sum_{j=1}^{J}C_{j,T}\right)\nonumber 
				+\left\Vert \hm{F}\right\Vert _{L_{t}^{2}\mathscr{W}_{\mathcal{H}}+L_{t}^{1}\mathcal{H}}
		\end{align}}
		where
		\[
		\mathscr{W}:=\bigcap_{j=1}^{J}\mathcal{D}^{-\frac{\nu}{2}}L_{x}^{2}\left\langle \cdot-\beta_{j}t+c_{j}\left(t\right)\right\rangle ^{\sigma}\bigcap B_{6/5,2}^{-5/6}
		\]
		Therefore, one has
		\begin{align*}
			\sum_{j=1}^{J}C_{j,T} & \leq\left\Vert \hm{v}\left(B\right)\right\Vert _{\mathcal{H}}+\left(\frac{1}{M^{\eta}}+e^{-\alpha B}+\delta\right)\left(\sum_{j=1}^{J}C_{j,T}\right) +\left\Vert\hm{ F}\right\Vert _{L_{t}^{2}\mathscr{W}_{\mathcal{H}}+L_{t}^{1}\mathcal{H}}.
		\end{align*}
		Hence picking $\delta$ small, $B$ and $M$ large but fixed, we can
		conclude that
		\[
		\sum_{j=1}^{J}C_{j,T}\lesssim\left\Vert \hm{v}\left(B\right)\right\Vert _{\mathcal{H}}+\left\Vert \hm{F}\right\Vert _{L_{t}^{2}\mathscr{W}_{\mathcal{H}}+L_{t}^{1}\mathcal{H}}
		\]
		which implies $\sum_{j=1}^{J}C_{j,T}$ is bounded with constants independent
		of $T$. Therefore, we can pass $T$ to $\infty$ and obtain that
		\begin{align}
			\sum_{j=1}^{J}\left\Vert \left\langle \cdot-\beta_{j}t+c_{j}\left(t\right)\right\rangle ^{-\sigma}\mathcal{D}^{-\nu/2}\hm{v}\right\Vert _{L_{t}^{2}\mathcal{H}} & \lesssim\left\Vert \hm{v}\left(B\right)\right\Vert _{\mathcal{H}}+\left\Vert \hm{F}\right\Vert _{L_{t}^{2}\mathscr{W}_{\mathcal{H}}+L_{t}^{1}\mathcal{H}}.\label{eq:localdecay}
		\end{align}
	\end{prop}

	\subsection{Strichartz estimates and scattering}
	
	In this subsection, we conclude Strichartz estimates, local decay
	estimate and scattering  for \eqref{eq:maineq} under the assumptions \eqref{eq:maincond} and \eqref{eq:maincondzero}.
	
	We write $\hm{u}=\left(\begin{array}{c}
	u\\
	u_{t}
	\end{array}\right)$ and
	\[
	\hm{u}_{t}=\mathfrak{L}\left(t\right)\hm{u}+\hm{F}
	\]
	where
	\begin{equation}
	\mathfrak{L}\left(t\right)=\left(\begin{array}{cc}
	0 & 1\\
	\Delta-1+\sum_{j=1}^{J}(V_{j})_{\beta_{j}(t)}(\cdot-y_{j}(t)) & 0
	\end{array}\right).\label{eq:Lt}
	\end{equation}
	We also set
	\begin{equation}
	\mathcal{V}\left(t\right):=\mathfrak{L}(t)-\mathfrak{L}_0.\label{eq:Vt}
	\end{equation}
	Denote $\hm{v}\left(t\right)=\pi_{cs}\left(t\right)\hm{u}$ where
	$\pi_{cs}$ is the projection onto the centre-stable space as in Proposition \ref{prop:expdicho}. Then we need to establish Strichartz estimates for
	\begin{equation}
	\hm{v}_{t}=\mathfrak{L}\left(t\right)\hm{v}+\pi_{cs}\left(t\right)\hm{F}.\label{eq:Hamgeneral}
	\end{equation}
	By construction, in this centre-stable space, $\left\Vert \hm{v}\left(t\right)\right\Vert _{\mathcal{H}}$
	is bounded or grows subexponentially for $t\geq0$.
	\begin{thm}\label{thm:mainthmHam}
		Denote
		\[
		\mathfrak{W}:=L_{t}^{2}\left(\bigcap_{j=1}^{J}\mathcal{D}^{\frac{\nu}{2}}L_{x}^{2}\left\langle \cdot-y_{j}\left(t\right)\right\rangle ^{-\sigma}\right)
		\]
		and
		\[
		S=L_{t}^{\infty}L_{x}^{2}\bigcap L_{t}^{2}B_{6/5,2}^{-5/6}.
		\]
		Using the notations above, let $u$ be the solution to \eqref{eq:Hamgeneral}.
		Then one has Strichartz estimates
		\[
		\left\Vert \hm{v}\right\Vert _{\mathfrak{W}_{\mathcal{H}}\bigcap S_{\mathcal{H}}}\lesssim\left\Vert \hm{v}\left(0\right)\right\Vert _{\mathcal{H}}+\left\Vert \hm{F}\right\Vert _{\left(\mathfrak{W}_{\mathcal{H}}\bigcap S_{\mathcal{H}}\right)^{*}}.
		\]
		Moreover, indeed $\hm{v}\left(t\right)$ scatters to a free wave.
		There exists $\hm{\phi}_{+}\in\mathcal{H}$ such that
		\[
		\left\Vert \hm{v}\left(t\right)-e^{\mathfrak{L}_{0}t}\hm{\phi}_{+}\right\Vert _{\mathcal{H}}\rightarrow\infty,\,\,t\rightarrow\infty.
		\]
	\end{thm}
	
	\begin{proof}
		Performing the reduction in Subsection \ref{subsec:Reduction}, it
		is sufficient to consider
		\begin{align*}
			\hm{v}_{t} & =\mathcal{L}(t)\hm{v}+e\left(t,x\right)\hm{v}+\hm{F}
		\end{align*}
		where as \eqref{eq:complexgamma}
		\[
		\mathcal{L}(t)=\left(\begin{array}{cc}
		0 & 1\\
		\Delta-1+\sum_{j=1}^{J}(V_{j})_{\beta_{j}}\left(\cdot-t\beta_{j}+c_{j}\left(t\right)\right) & 0
		\end{array}\right)
		\]
		and $e\left(t,x\right)=\sum_{j=1}^{J}\left(\begin{array}{c}
		0\\
		e_{j}\left(t,x\right)
		\end{array}\right)$ such that $e_{j}$ is smooth and $\left|e_{j}\left(t,x\right)\right|\lesssim\delta e^{-\alpha\left|x-\beta_{j}t\right|}$.
		
		By Proposition \ref{prop:bootstrap}, we obtain that
		\begin{align}
			\sum_{j=1}^{J}\left(\int_{0}^{\infty}\left\langle \cdot-y_{j}\left(t\right)\right\rangle ^{-2\sigma}\left|\mathcal{D}^{-\frac{\nu}{2}}\hm{v}\left(t,x\right)\right|^{2}\,dxdt\right)^{\frac{1}{2}}
			\lesssim\left\Vert \hm{v}\left(0\right)\right\Vert _{\mathcal{H}}+\left\Vert e\left(t,x\right)\hm{v}+\hm{F}\right\Vert _{\left(\mathfrak{W}\bigcap S\right)^{*}_{\mathcal{H}}}
		\end{align}
		Note that
		\[
		\left\Vert e\left(t,x\right)\hm{v}\right\Vert _{\left(\mathfrak{W}\bigcap S\right)^{*}}\leq\delta\sum_{j=1}^{J}\left(\int_{0}^{\infty}\left\langle x-y_{j}\left(t\right)\right\rangle ^{-2\sigma}\left|\mathcal{D}^{-\frac{\nu}{2}}\hm{v}\left(t,x\right)\right|^{2}\,dxdt\right)^{\frac{1}{2}}.
		\]
		Therefore, we conclude that
		\begin{align}
			\sum_{j=1}^{J}\left(\int_{0}^{\infty}\left\langle x-y_{j}\left(t\right)\right\rangle ^{-2\sigma}\left|\mathcal{D}^{-\frac{\nu}{2}}\hm{v}\left(t,x\right)\right|^{2}\,dxdt\right)^{\frac{1}{2}}
			\lesssim\left\Vert \hm{v}\left(0\right)\right\Vert _{\mathcal{H}}+\left\Vert \hm{F}\right\Vert _{\left(\mathfrak{W}\bigcap S\right)^{*}_{\mathcal{H}}}
		\end{align}
		In other words,
		\[
		\left\Vert \hm{v}\right\Vert _{\mathfrak{W}_{\mathcal{H}}}\lesssim\left\Vert \hm{v}\left(0\right)\right\Vert _{\mathcal{H}}+\left\Vert \hm{F}\right\Vert _{\left(\mathfrak{W}\bigcap S\right)^{*}_{\mathcal{H}}}.
		\]
		Then we use the Duhamel formula to expand $\hm{v}$,
		\[
		\hm{v}=e^{t\mathfrak{L}_{0}}\hm{v}\left(0\right)+\int_{0}^{t}e^{\left(t-s\right)\mathfrak{L}_{0}}\left\{ \mathcal{V}\left(s\right)\hm{v}+\hm{F}\left(s\right)\right\} \,ds.
		\]
		Applying Strichartz norms on both sides of the expansion above, we get
		\begin{align*}
			\left\Vert \hm{v}\right\Vert _{S_{\mathcal{H}}} & \lesssim\left\Vert e^{t\mathfrak{L}_{0}}\hm{v}\left(0\right)\right\Vert _{S_{\mathcal{H}}}+\left\Vert \int_{0}^{t}e^{\left(t-s\right)\mathfrak{L}_{0}}\left\{ \mathcal{V}\left(s\right)\hm{v}+\hm{F}\left(s\right)\right\} \,ds\right\Vert _{S_{\mathcal{H}}}\\
			& \lesssim\left\Vert \hm{v}\left(0\right)\right\Vert _{\mathcal{H}}+\left\Vert \hm{v}\right\Vert _{\mathfrak{W}_{\mathcal{H}}}+\left\Vert \hm{F}\right\Vert _{S^{*}_{\mathcal{H}}}\lesssim\left\Vert \hm{v}\left(0\right)\right\Vert _{\mathcal{H}}+\left\Vert \hm{F}\right\Vert _{\left(\mathfrak{W}\bigcap S\right)^{*}_{\mathcal{H}}}.
		\end{align*}
		Hence
		\[
		\left\Vert \hm{v}\right\Vert _{\mathfrak{W}_{\mathcal{H}}\bigcap S_{\mathcal{H}}}\lesssim\left\Vert \hm{v}\left(0\right)\right\Vert _{\mathcal{H}}+\left\Vert \hm{F}\right\Vert _{\left(\mathfrak{W}\bigcap S\right)^{*}_{\mathcal{H}}}.
		\]
		Next we show that $\hm{v}$ scatters to the free wave. By the standard
		argument, we need to show
		\[
		\int_{0}^{\infty}e^{-s\mathfrak{L}_{0}}\left\{ \mathcal{V}\left(s\right)\hm{v}+\hm{F}\left(s\right)\right\} \,ds\in\mathcal{H}.
		\]
		Applying the dual version of \eqref{eq:shiftfreehom} in Lemma \ref{lem:shiftslant},
		we obtain that
		\begin{align*}
			\left\Vert \int_{0}^{\infty}e^{-s\mathfrak{L}_{0}}\left\{ \mathcal{V}\left(s\right)\hm{v}+\hm{F}\left(s\right)\right\} \right\Vert _{\mathcal{H}} & \lesssim\left\Vert \hm{F}\right\Vert _{\left(\mathfrak{W}\bigcap S\right)^{*}_{\mathcal{H}}}+\left\Vert u\right\Vert _{\mathfrak{W}_{\mathcal{H}}}\\
			& \lesssim\left\Vert \hm{v}\left(0\right)\right\Vert _{\mathcal{H}}+\left\Vert \hm{F}\right\Vert _{\left(\mathfrak{W}\bigcap S\right)^{*}_{\mathcal{H}}}.
		\end{align*}
		Define
		\[
		\hm{\phi}_{+}=\hm{v}\left(0\right)+\int_{0}^{\infty}e^{-s\mathfrak{L}_{0}}\left\{ \mathcal{V}\left(s\right)\hm{v}+\hm{F}\left(s\right)\right\} \,ds
		\]
		then by construction,
		\[
		\left\Vert \hm{v}\left(t\right)-e^{\mathfrak{L}_{0}t}\hm{\phi}_{+}\right\Vert _{\mathcal{H}}\rightarrow\infty,\,\,t\rightarrow\infty
		\]
		as desired.
	\end{proof}
	
	\subsection{Truncated estimates and interaction\label{ssubec:TruncateInter}}
	
	To finish the analysis of the multi-potential problem, in this subsection,
	we present the full details of the the interaction estimate and truncated
	estimate used above.  As discussed in the single-potential setting, it suffices to consider the half-Klein-Gordon evolution.

	From the Duhamel expansion in each channel, we have to deal with the
	interaction terms of the form
	\begin{equation}
	\left\langle x-\beta_{1}t-c_{1}\left(t\right)\right\rangle ^{-\sigma}\int_{0}^{t-M}e^{\left(t-s\right)i\mathcal{D}}\mathrm{U}_{A}\left(t,s\right)\left\langle \cdot-\beta_{2}s-c_{2}\left(s\right)\right\rangle ^{-\alpha}\mathcal{D}^{-1}f\left(s\right)\,ds\label{eq:interaction}
	\end{equation}
	where $\mathrm{U}_{A}\left(t,s\right)$ the evolution operator defined
	by the equation $u_{t}=A\left(t\right)u=a(t)\nabla u$.

	We need to estimate the $L^2_t \mathcal{D}^{\frac{\nu}{2}}L_{x}^{2}$ norm
	with $\nu>0$ small and the $L_{t}^{2}L_{x}^{2}$ norm of \eqref{eq:interaction}.
	Denote the weights as
	\[
	w_{j}\left(x,t\right)=\left\langle x-\beta_{j}t-c{}_{j}\left(t\right)\right\rangle ^{-\sigma},\,j=1,2.
	\]
	Performing the Littlewood-Paley decomposition, we consider the following
	term
	\[
	I_{j}^{M}\left(f,g\right)=\int_{0}^{\infty}\int_{0}^{\max\left\{ t-M,0\right\} }\left\langle e^{\left(t-s\right)i\mathcal{D}}\mathrm{U}_{A}\left(t,s\right)\varDelta_{j}w_{2}\left(s\right)f\left(s\right),w_{1}\left(t\right)g\left(t\right)\right\rangle \,dsdt,
	\]
	for $j\geq0$, $M\gg1$, and $f,\,g\in L_{t,x}^{2}$. $1=\sum_{j=0}^{\infty}\varDelta_{j}$
	is the Littlewood-Paley decomposition defined by $\varDelta_{j}u=\mathcal{F}^{-1}\varphi\left(2^{-j}\xi\right)\hat{u}\left(\xi\right)$
	for all $j\geq1$ with some radial non-negative $\varphi\in C_{0}^{\infty}$
	supported on $\frac{1}{2}<\left|\xi\right|<2$.
	
	To sum $I_{j}$ over $j\geq0$ we use the almost orthogonality as
	following
	\[
	I_{j}^{M}\left(f,g\right)=I_{j}^{M}\left(Q_{j}f,Q_{j}g\right),
	\]
	\[
	Q_{j}f:=\sum_{\left|k-j\right|\leq1}\varDelta_{k}f+w^{-1}\left[\varDelta_{k},w\right]f,
	\]
	where the commutator term is small in the sense that
	\[
	\left\Vert w^{-1}\left[\varDelta_{k},w\right]f\right\Vert _{L^{2}}\lesssim2^{-k}\left\Vert f\right\Vert _{L^{2}}.
	\]
	In order to obtain the refined interaction estimate, we further decompose
	$w_{1}$ and $w_{2}$ dyadically:
	\[
	w_{j}^{\ell}\left(x,t\right)=\varphi\left(\frac{x-\beta_{j}t-c_{j}\left(t\right)}{2^{\ell}}\right)w_{j}\left(x.t\right),\,\ell\in\mathbb{N}.
	\]
	Notice that due to the decay of $w_{j}$, all the $L^{p}$ norms of
	$w_{j}^{\ell}$ decays like $2^{-\beta\ell}$ for $\beta>0$ which
	can be as large as desired provided $\alpha$ is large enough. Due
	to this decay rate, we are free to decompose those weights into pieces,
	and then sum up them. In the remaining part of this analysis, we take
	$\ell=0$ in both $w_{1}$ and $w_{2}$. For simplicity, to abuse
	the notations, we assume that $w_{1}$ and $w_{2}$ are compactly
	supported.
	\begin{lem}
		\label{lem:interactionLP}Let $\sigma>\frac{3}{2}$ and $M>1$ large.
		Then for any $j\geq0$, we have
		\[
		\left|I_{j}^{M}\left(f,g\right)\right|\lesssim\min\left\{ 2^{2j}M^{\frac{3}{2}-\sigma},2^{\kappa j}M^{-\eta}\right\} \left\Vert f\right\Vert _{L_{t}^{2}L_{x}^{2}}\left\Vert g\right\Vert _{L_{t}^{2}L_{x}^{2}},
		\]
		for some $0\leq\kappa<1$ and $\eta>0$ provided $\left\Vert a\right\Vert _{L_{t}^{\infty}}+\left\Vert c_{1}'\right\Vert _{L_{t}^{\infty}}+\left\Vert c_{2}'\right\Vert _{L_{t}^{\infty}}\ll1$.
	\end{lem}
	
	\begin{proof}
		We can rewrite
		\begin{equation}
		I_{j}^{M}\left(f,g\right)=J_{j}^{M}\left(f,g\right)+\overline{J_{j}^{M}\left(g,f\right)}.\label{eq:Imjdecom}
		\end{equation}
		Setting the space-time function $K_{j}$ and $K^{j}$ by
		\[
		K_{j}\left(t,x\right)=2^{3j}K^{j}\left(2^{j}t,2^{j}x\right)=e^{it\mathcal{D}}2^{3j}\hat{\varphi}\left(2^{j}x\right)
		\]
		it suffices to analyze the first piece in \eqref{eq:Imjdecom}
		\begin{equation}
		J_{j}^{M}\left(f,g\right)=\int_{\mathcal{R}\left(M\right)}K_{j}\left(t-s,x-y+b\right)w_{2}\left(s,y\right)f\left(s,y\right)w_{1}\left(t,x\right)\overline{g}\left(t,x\right)\,dsdtdxdy\label{eq:Jmgeneral}
		\end{equation}
		where $b=b\left(t,s\right)=\int_{s}^{t}a\left(\tau\right)\,d\tau$
		and the integration region is given by
		\begin{equation}
		\mathcal{R}\left(M\right)=\left\{ \left(x,y,t,s\right)\left|y-\beta_{2}s+c_{2}\left(s\right)\right|>\left|x-t\beta_{1}+c_{1}\left(t\right)\right|,s>0,t>s+M\right\} .\label{eq:integrationreg}
		\end{equation}
		Writing $K_{j}\left(t,x\right)=2^{3j}K^{j}\left(2^{j}t,2^{j}x\right)$,
		we recall that for $t\gtrsim1$, one has that for $N\in\mathbb{N}$
		arbitrary
		\begin{equation}
		\left|K^{j}\left(t,x\right)\right|\lesssim\begin{cases}
		t^{-1}\left\langle \left|t-\left|x\right|\right|+2^{-2j}t\right\rangle ^{-N} & \left(\left|t-\left|x\right|\right|\nsim2^{-2j}t\right),\\
		t^{-1}\left\langle t-\left|x\right|\right\rangle ^{-\frac{1}{2}} & \left(\left|t-\left|x\right|\right|\sim2^{-2j}t\right).
		\end{cases}\label{eq:Kjdecay}
		\end{equation}
		By the decay estimates above, from \eqref{eq:Jmgeneral}, the kernel
		has a relatively weak decay rate around $t-s\sim\left|x-y+b\right|$.
		From this fact, the size of the supports of weights, combing with
		the finite speed of propagation, for given $M_{j}$ to be determined
		later on, we can define the strong interaction region as
		\[
		\mathfrak{S}\left(M_{j},t\right):=\left\{ 0\leq s\leq t:\,\left|t-s-\left|\beta_{1}t+c_{1}\left(t\right)-\beta_{2}s-c_{2}\left(s\right)+b\left(t,s\right)\right|\right|\leq M_{j}\right\} .
		\]
		We claim that for fixed $t$, the size of the strong interaction region can be estimated as following
		\begin{equation}
		\left|\mathfrak{S}\left(M_{j},t\right)\right|\sim M_{j}\label{eq:sizestronginter}
		\end{equation}
		provided that
		\begin{equation}
		\left|\beta_{1}t+c_{1}\left(t\right)-\beta_{2}s-c_{2}\left(s\right)+b\left(t,s\right)\right|\geq\frac{1}{4}\left|t-s\right|\geq\frac{1}{4}M.\label{eq:lowerboundstronginter}
		\end{equation}
		This follows from  a direct computation to show that the derivative
		of the function
		\[
		\mathfrak{d}\left(s\right)=t-s-\left|\beta_{1}t+c_{1}\left(t\right)-\beta_{2}s-c_{2}\left(s\right)+b\left(t,s\right)\right|
		\]
		has a definite sign and it is absolute value is bounded above by $\frac{3}{2}$
		and below by $\frac{1}{2}$ provided that $M$ is large and $\left|c'_{2}+a\right|$
		is small enough.
		
		When \eqref{eq:lowerboundstronginter} does not hold, noting that
		\begin{equation}
		\left|\beta_{1}t+c_{1}\left(t\right)-\beta_{2}s-c_{2}\left(s\right)+b\left(t,s\right)\right|\leq\frac{1}{4}\left|t-s\right|\label{eq:lowerboundfail}
		\end{equation}
		in the slow decay region: $\left|x-y+b\right|\sim t-s$, we have
		\begin{align*}
			\left|t-s\right| & \sim\left|x-y+b\right|\leq\left|x-\left(\beta_{1}t+c_{1}\left(t\right)\right)\right|\\
			& +\left|y-\left(\beta_{2}s+c_{2}\left(s\right)\right)\right|+\left|\beta_1 t+c_{1}\left(t\right)-\beta_{2}s-c_{2}\left(s\right)+b\left(t,s\right)\right|
		\end{align*}
		which implies
		\begin{equation}
		\left|y-\left(\beta_{2}s+c_{2}\left(s\right)\right)\right|\gtrsim\left|t-s\right|\label{eq:lowerboundy}
		\end{equation}
		in the integration region $\mathcal{R}\left(M\right)$ given by \eqref{eq:integrationreg}.
		
		Denote
		\[
		z=-x+y-b,\,u=t-s,\,h=x+b.
		\]
		Then
		\[
		h-b=x,\,z+h=y.
		\]
		We can rewrite
		\begin{align*}
			J_{j}^{M}\left(f,g\right) & =\int_{\tilde{\mathcal{R}}\left(M\right)}2^{3j}K^{j}\left(2^{j}u,2^{j}z\right)w_{2}\left(s,h+z\right)f\left(s,h+z\right)\\
			& \ \ \ \ \ \ \ \ \ w_{1}\left(s+u,h-b\right)\overline{g}\left(s+u,h-b\right)\,dsdudhdz
		\end{align*}
		where the integration region is given by
		\[
		\tilde{\mathcal{R}}\left(M\right)=\left\{ \left|h+z-\beta_{1}\left(u+s\right)-c_{1}\left(u+s\right)\right|>\left|h-b-\beta_{2}s-c_{2}\left(s\right)\right|,s>0,u>M\right\} .
		\]
		In the region where the kernel decays fast, $u\nsim\left|z\right|$,
		we can bound
		\begin{align}
			J_{j}^{M}\left(f,g\right) & \lesssim\int2^{3j}\left(2^{j}u\right)^{-3-N}\left|w_{2}\left(s,h+z\right)f\left(s,h+z\right)w_{1}\left(s+u,h-b\right)\overline{g}\left(s+u,h-b\right)\right|dsdudhdz\nonumber \\
			& \lesssim2^{-jN}M^{-2-N}\left\Vert f\right\Vert _{L_{t}^{2}L_{x}^{2}}\left\Vert g\right\Vert _{L_{t}^{2}L_{x}^{2}}.\label{eq:strongdecayregion}
		\end{align}
		In the region where the kernel decays slowly, $u\sim\left|z\right|$,
		in the region given by \eqref{eq:lowerboundfail} and \eqref{eq:lowerboundy},
		we have
		\[
		u\lesssim\left|h+z-\beta_{2}s-c_{2}\left(s\right)\right|.
		\]
		Therefore the integral above is bounded by
		\begin{align}
			\int_{u\ge M,u\lesssim\left|h+z-\beta_{2}s-c_{2}\left(s\right)\right|}\frac{2^{3j}}{2^{j}u}\frac{\left|f\left(s,h+z\right)\overline{g}\left(s+u,h-b\right)\right|}{\left\langle h+z-\beta_{2}s-c_{2}\left(s\right)\right\rangle ^{\sigma}}w_{1}\left(s+u,h-b\right)\,dhdzduds\nonumber \\
			\lesssim\int_{u\geq M}2^{j}u^{-1}\left\Vert \left\langle h+z-\beta_{2}s-c_{2}\left(s\right)\right\rangle ^{-\sigma}w_{1}\left(s+u,h-b\right)\right\Vert _{L_{\left\{ h:\left|h+z-\beta_{2}s-c_{2}\left(s\right)\right|\gtrsim u\right\} }^{2}L_{\left|\left|z\right|\sim u\right|}^{2}}\nonumber \\
			\times\left\Vert f\left(s,\cdot\right)\right\Vert _{L_{x}^{2}}\left\Vert g\left(s+u,\cdot\right)\right\Vert _{L_{x}^{2}}\,dsdu\nonumber \\
			\lesssim\int_{u>M}2^{2j}u^{\frac{1}{2}-\sigma}\left\Vert f\left(s,\cdot\right)\right\Vert _{L_{x}^{2}}\left\Vert g\left(s+u,\cdot\right)\right\Vert _{L_{x}^{2}}\,dsdu\nonumber \\
			\lesssim2^{2j}M^{\frac{3}{2}-\sigma}\left\Vert f\right\Vert _{L_{t}^{2}L_{x}^{2}}\left\Vert g\right\Vert _{L_{t}^{2}L_{x}^{2}}.\label{eq:slowbound1}
		\end{align}
		Outside of the strong interaction regime but in the slow decay region
		$u\sim\left|z\right|$, it remains to analyze the region given by
		\[
		\left|t-s-\left|\beta_{1}t+c_{1}\left(t\right)-\beta_{2}s-c_{2}\left(s\right)+b\left(t,s\right)\right|\right|>M_{j}
		\]
		and
		\[
		\left|\beta_{1}t+c_{1}\left(t\right)-\beta_{2}s-c_{2}\left(s\right)+b\left(t,s\right)\right|>\frac{1}{4}\left|t-s\right|.
		\]
		Picking $M_{j}$ larger than the size of the supports of $w_{1}$
		and $w_{2}$, we notice that
		\begin{align*}
			\left|t-s-\left|x+b-y\right|\right| & >M_{j}-\left|x-\beta_{1}t-c_{1}\left(t\right)\right|\\
			& -\left|y-\beta_{2}s-c_{2}\left(s\right)\right|\geq\frac{1}{2}M_{j}.
		\end{align*}
		For $t-s<\left|x+b-y\right|$, the result can be obtained by symmetry. So we consider $t-s-\left|x+b-y\right|\geq\frac{1}{2}M_{j}$
		and $u\sim\left|z\right|$. By the decay estimate of the kernel, if
		$M_{j}\nsim2^{-2j}\left(t-s\right)=2^{-2j}u$, we apply the first
		decay estimate from \eqref{eq:Kjdecay} and obtain the same estimate
		as \eqref{eq:strongdecayregion}
		\begin{equation}
		\left|J_{j}^{M}\left(f,g\right)\right|\lesssim2^{-jN}M^{-2-N}\left\Vert f\right\Vert _{L_{t}^{2}L_{x}^{2}}\left\Vert g\right\Vert _{L_{t}^{2}L_{x}^{2}}.\label{eq:slowstrong}
		\end{equation}
		When $2^{-2j}u\gtrsim M_{j}$, we have to use the second decay estimate
		from \eqref{eq:Kjdecay}. Then one has the bound
		\begin{align}
			\int_{u\ge2^{2j}M_{j}}\frac{2^{3j}}{2^{j}u}w_{2}\left(s,h+z\right)\left|f\left(s,h+z\right)\overline{g}\left(s+u,h-b\right)\right|w_{1}\left(s+u,h-b\right)\,dhdzduds\nonumber \\
			\lesssim2^{j}M_{j}^{-\frac{1}{2}}\left\Vert f\right\Vert _{L_{t}^{2}L_{x}^{2}}\left\Vert g\right\Vert _{L_{t}^{2}L_{x}^{2}}.\label{eq:slownonstrong}
		\end{align}
		Finally, it remains to analyze the strong interaction region \eqref{eq:sizestronginter}.
		Here we only use the conservation of the $L^{2}$ norm. From the size estimate \eqref{eq:sizestronginter}, one obtains that
		\begin{align}
			\left|\int_{\mathfrak{S}\left(M_{j}\right)}K_{j}\left(t-s,x-y+b\right)w_{2}\left(s,y\right)f\left(s,y\right)w_{1}\left(t,x\right)\overline{g}\left(t,x\right)\,dsdtdxdy\right|\nonumber \\
			\lesssim M_{j}\left\Vert f\right\Vert _{L_{t}^{2}L_{x}^{2}}\left\Vert g\right\Vert _{L_{t}^{2}L_{x}^{2}}.\label{eq:stronginter}
		\end{align}
		Therefore, overall, putting \eqref{eq:stronginter}, \eqref{eq:slownonstrong},
		\eqref{eq:slowstrong}, \eqref{eq:slowbound1} and \eqref{eq:strongdecayregion}
		together, one has
		\begin{equation}
		\left|I_{j}^{M}\left(f,g\right)\right|\lesssim\max\left\{ 2^{2j}M^{\frac{3}{2}-\sigma},M_{j},2^{j}M_{j}^{-\frac{1}{2}},2^{-jN}M^{-2-N}\right\} \left\Vert f\right\Vert _{L_{t}^{2}L_{x}^{2}}\left\Vert g\right\Vert _{L_{t}^{2}L_{x}^{2}}.\label{eq:result1}
		\end{equation}
		To optimize the coefficients in terms of $M_{j}$, we can pick $M_{j}=2^{\frac{2}{3}j}$. 
		
		Moreover, from the uniform decay of the kernel, we also have
		\begin{align}
			\left|J_{j}^{M}\left(f,g\right)\right| & \lesssim2^{\frac{5}{2}j}\int_{u>M}u^{-\frac{3}{2}}\left\Vert w\left(s,x\right)f\left(s,x\right)\right\Vert _{L_{x}^{1}}\left\Vert w\left(s+u,h\right)f\left(s+u,h\right)\right\Vert _{L_{y}^{1}}\,dsdu\nonumber \\
			& \lesssim2^{\frac{5}{2}j}M^{-\frac{1}{2}}\left\Vert f\right\Vert _{L_{t}^{2}L_{x}^{2}}\left\Vert g\right\Vert _{L_{t}^{2}L_{x}^{2}}.\label{eq:result2}
		\end{align}
		Picking $\kappa<1$ and $\eta>0$ appropriately, for $j$ small, we
		use \eqref{eq:result2} and apply \eqref{eq:result1} for $j$ large.
		More precisely, let $\kappa$ be close to $1$ and $\eta$ be small. Then
		\[
		2^{\frac{2}{3}j}\leq2^{j\kappa}M^{-\eta}
		\]
		provided that $j\geq\frac{\eta}{\kappa-\frac{2}{3}}\log_{2}M$. Secondly,
		notice that
		\[
		2^{\frac{5}{2}j}M^{-\frac{1}{2}}\leq2^{\kappa}M^{-\eta}
		\]
		given that $2^{j\left(\frac{5}{2}-\kappa\right)}\leq M^{\frac{1}{2}-\eta}$
		which is ensured by $j\leq\frac{\frac{1}{2}-\eta}{\frac{5}{2}-\kappa}\log_{2}M$.
		Therefore, picking $\frac{2}{3}<\kappa<1$ and $\eta$ small enough
		such that $\frac{\frac{1}{2}-\eta}{\frac{5}{2}-\kappa}\geq\frac{\eta}{\kappa-\frac{2}{3}}$
		we can obtain that
		\[
		\left|I_{j}^{M}\left(f,g\right)\right|\lesssim\min\left\{ C\left(\sigma\right)2^{2j}M^{\frac{3}{2}-\sigma},2^{\kappa j}M^{-\eta}\right\} \left\Vert f\right\Vert _{L_{t}^{2}L_{x}^{2}}\left\Vert g\right\Vert _{L_{t}^{2}L_{x}^{2}}
		\]
		as desired.
	\end{proof}
	With the lemma above, noticing that $\kappa<1$, if we can gain one
	derivative from $f$, it will allow us to sum
	the Littlewood-Paley decomposition.
	\begin{lem}
		\label{lem:truncated}Let $1>\nu>0$ and $\sigma\geq\frac{1}{2}+\frac{2}{\nu}$.
		If $\left\Vert a\right\Vert _{L_{t}^{\infty}}\ll1$ then there is
		$\eta>0$ such that
		\begin{align}
			\left\Vert \left\langle \cdot-\beta_{1}t-c_{1}\left(t\right)\right\rangle ^{-\sigma}\int_{0}^{t-M}e^{\left(t-s\right)i\mathcal{D}}\mathrm{U}_{A}\left(t,s\right)\mathcal{D}^{-1}f\left(s\right)\,ds\right\Vert _{L_{t}^{2}L_{x}^{2}}\nonumber \\
			\lesssim M^{-\eta}\left\Vert \left\langle \cdot-\beta_{2}t+c_{2}\left(t\right)\right\rangle ^{\alpha}f\right\Vert _{L_{t}^{2}L_{x}^{2}}.\label{eq:firsttruncate}
		\end{align}
		Moreover, one also has
		\begin{align}
			\left\Vert \left\langle \cdot-\beta_{1}t-c_{1}\left(t\right)\right\rangle ^{-\sigma}\mathcal{D}^{-\frac{\nu}{2}}\int_{0}^{t-M}e^{\left(t-s\right)i\mathcal{D}}\mathrm{U}_{A}\left(t,s\right)\mathcal{D}^{-1}f\left(s\right)\,ds\right\Vert _{L_{t}^{2}L_{x}^{2}}\nonumber \\
			\lesssim M^{-\eta}\left\Vert \left\langle \cdot-\beta_{2}t+c_{2}\left(t\right)\right\rangle ^{\alpha}\mathcal{D}^{-\frac{\nu}{2}}f\right\Vert _{L_{t}^{2}L_{x}^{2}}.\label{eq:secondtruncate}
		\end{align}
	\end{lem}
	
	\begin{proof}
		We first use the Littlewood-Paley decomposition. Testing against $g$,
		it is reduced to analyze
		\[
		I_{j}^{M}\left(\mathcal{D}^{-1}f,g\right).
		\]
		In order to achieve the decay in terms of $M$, we further truncate
		the integration region to analyze $I_{j}^{\mathcal{M}_{j}}\left(\mathcal{D}^{-1}f,g\right)$
		with an lower bound for $t-s$ depending on $j$.
		
		Taking the difference and applying the trivial $L_{x}^{2}$ conservation,
		we have
		\begin{equation}
		\left|I_{j}^{M}\left(\mathcal{D}^{-1}f,g\right)-I_{j}^{\mathcal{M}_{j}}\left(\mathcal{D}^{-1}f,g\right)\right|\lesssim\left(\mathcal{M}_{j}-M\right)\left\Vert \mathcal{D}^{-1}f\right\Vert _{L_{t,x}^{2}}\left\Vert g\right\Vert _{L_{t,x}^{2}}.\label{eq:L2part}
		\end{equation}
		Now recall bounds from Lemma \ref{lem:interactionLP} with the parameter
		$\mathcal{M}_{j}$
		\[
		\min\left\{ 2^{2j}\mathcal{M}_{j}^{\frac{3}{2}-\sigma},2^{\kappa j}\mathcal{M}_{j}^{-\eta}\right\} .
		\]
		For any $1>\nu>0$, we choose $\mathcal{M}_{j}=2^{\nu j}$ and $\sigma$ large to ensure
		\[
		2^{2j}\mathcal{M}_{j}^{\frac{3}{2}-\sigma}\leq\mathcal{M}_{j}=2^{j\nu},\,2^{\kappa j}\mathcal{M}_{j}^{-\eta}\leq2^{j\nu}.
		\]
		Therefore, by Lemma \ref{lem:interactionLP}, with \eqref{eq:L2part},
		we have
		\begin{equation}
		\left|I_{j}^{\mathcal{M}_{j}}\left(\mathcal{D}^{-1}f,g\right)\right|\lesssim2^{\left(\nu-1\right)j}\left\Vert f\right\Vert _{L_{t,x}^{2}}\left\Vert g\right\Vert _{L_{t,x}^{2}}\label{eq:own-1}
		\end{equation}
		and
		\begin{equation}
		\left|I_{j}^{M}\left(\mathcal{D}^{-1}f,g\right)\right|\lesssim2^{\left(\nu-1\right)j}\left\Vert f\right\Vert _{L_{t,x}^{2}}\left\Vert g\right\Vert _{L_{t,x}^{2}}.\label{eq:IMJnu}
		\end{equation}
		By Lemma \ref{lem:interactionLP}, we also know
		\[
		\left|I_{j}^{M}\left(\mathcal{D}^{-1}f,g\right)\right|\lesssim\min\left\{ 2^{j}M^{\frac{3}{2}-\sigma},2^{\left(\kappa-1\right)j}M^{-\eta}\right\} \left\Vert f\right\Vert _{L_{t}^{2}L_{x}^{2}}\left\Vert g\right\Vert _{L_{t}^{2}L_{x}^{2}}.
		\]
		For the low-frequency part, we notice that $2^{\left(\kappa-1\right)j}M^{-\eta}$
		can give the desired smallness in terms of $M$ after summing over
		the Littlewood-Paley decomposition. On the other hand, we note that
		$2^{2j}M^{\frac{3}{2}-\sigma}\leq2^{\left(\nu-1\right)j}M^{-\eta}$
		which can be ensured by
		\[
		j\leq\frac{\sigma-\frac{3}{2}-\eta}{3-\nu}\log_{2}M.
		\]
		For the large-frequency part, we use \eqref{eq:IMJnu} and note that
		\[
		2^{\left(\nu-1\right)j}\leq M^{-\eta}2^{\left(\nu+\epsilon\left(\eta\right)-1\right)j}
		\]
		which is ensured by $j\geq\frac{\eta}{\epsilon\left(\eta\right)}\log_{2}M$.
		We can always pick $\sigma$ large enough such that
		\[
		\frac{\sigma-\frac{3}{2}-\eta}{3-\nu}\geq\frac{\eta}{\epsilon\left(\eta\right)}.
		\]
		Therefore, for all $j\in\mathbb{N}$, we have
		\[
		\left|I_{j}^{M}\left(\mathcal{D}^{-1}f,g\right)\right|\lesssim M^{-\eta}2^{-\mu j}
		\]
		where $\mu=\min\left\{ 1-\kappa,1-\nu-\epsilon\left(\eta\right)\right\} >0$.
		
		Hence we can sum up the Littlewood-Paley decomposition as
		\begin{align*}
			\left|\int\int_{0}^{t-M}\left\langle e^{i\left(t-s\right)\mathfrak{L}_{0}}\mathrm{U}_{1}\left(t,s\right)w_{2}\left(s\right)\mathcal{D}^{-1}f\left(s\right),w_{1}\left(t\right)g\left(t\right)\right\rangle \,dsdt\right|\ \\
			\leq\sum_{j=0}^{\infty}\left|I_{j}^{M}\left(\mathcal{D}^{-1}f,g\right)\right|\ \ \ \ \ \ \ \ \ \ \ \ \ \ \ \ \ \ \ \ \ \ \ \ \ \ \ \ \ \ \ \ \ \ \ \ \\
			\lesssim\sum_{j=0}^{\infty}\sum_{\left|j-k\right|+\left|j-\ell\right|\leq2}M^{-\eta}2^{-\mu j}\left\Vert \varDelta_{k}f\right\Vert _{L_{t,x}^{2}}\left\Vert \varDelta_{\ell}g\right\Vert _{L_{t,x}^{2}}+M^{-\eta}2^{\left(\nu-2\right)j}\left\Vert f\right\Vert _{L_{t,x}^{2}}\left\Vert g\right\Vert _{L_{t,x}^{2}}\\
			\lesssim M^{-\eta}\left\Vert f\right\Vert _{L_{t,x}^2}\left\Vert g\right\Vert _{L_{t,x}^{2}}
		\end{align*}
		which implies via duality,
		\begin{align}
			\left\Vert \left\langle \cdot-\beta_{1}t-c_{1}\left(t\right)\right\rangle ^{-\sigma}\int_{0}^{t-M}e^{i\left(t-s\right)\mathfrak{L}_{0}}\mathrm{U}_{A}\left(t,s\right)\mathcal{D}^{-1}f\left(s\right)\,ds\right\Vert _{L_{t}^{2}L_{x}^{2}}\nonumber \\
			\lesssim M^{-\eta}\left\Vert \left\langle \cdot-\beta_{2}t+c_{2}\left(t\right)\right\rangle ^{\alpha}f\right\Vert _{L_{t}^{2}L_{x}^{2}}.\label{eq:firsttruncate-1}
		\end{align}
		By a similar argument, we can also obtain the bound \eqref{eq:secondtruncate}.
		We are done.
	\end{proof}

	\section{Existence of wave operators}\label{sec:waveop}
	
	In this section, we establish the existence of wave operators restricted
	onto the central direction.  Recall that in this setting, we further assume that all $\phi^0_{j,m}$  from \eqref{eq:modescalar} are $0$. In other words, there are no zero modes.
	
	\subsection{Existence of wave operators}
	
	
	We directly work on the Hamiltonian formalism. Our goal of this section is to show the following:
	\begin{thm}\label{thm:exiWave}
		For any free data $\hm{\phi}_{0}\in\mathcal{H}$,
		there exists the corresponding data $\hm{u}\left(0\right)=\pi_{c}\left(0\right)\hm{u}\left(0\right)$
		such that the solution to the equation 
		\[
		\hm{u}_{t}=\mathfrak{L}\left(t\right)\hm{u}+\hm{F}
		\]
		where
		\begin{equation}
		\mathfrak{L}\left(t\right):=\left(\begin{array}{cc}
		0 & 1\\
		\Delta-1+\sum_{j=1}^{J}(V_{j})_{\beta_{j}(t)}(\cdot-y_{j}(t)) & 0
		\end{array}\right)
		\end{equation}
		with initial data $\hm{u}(0)$
		satisfying
		\[
		\left\Vert \hm{u}(t)-e^{t\mathfrak{L}_0}\hm{\phi}_{0}\right\Vert _{\mathcal{H}}\rightarrow0,\ t\rightarrow\infty.
		\]
		Here $\pi_{c}\left(t\right)$ is the projection on the central space.
	\end{thm}
	After projecting the equation by $\pi_{c}\left(t\right)$, it suffices
	to consider the problem
	\[
	\hm{u}_{t}=\mathfrak{L}\left(t\right)\hm{u}+\pi_{c}\left(t\right)\hm{F}
	\]
	with $\hm{u}=\pi_{c}\left(t\right)\hm{u}\in\pi_{c}\left(t\right)\mathcal{H}$.

	We first remove the inhomogeneous term by considering the following
	data
	\begin{equation}
	\hm{\vartheta}_{t}=\mathfrak{L}\left(t\right)\hm{\vartheta}+\pi_{c}\left(t\right)\hm{F},\:\hm{\vartheta}\left(0\right)=0.\label{eq:reinhomo}
	\end{equation}
	From our computations above and  Theorem \ref{thm:mainthmHam}, we know that there
	exists $\hm{\rho}_{0}\in\mathcal{H}$ such that
	\[
	\left\Vert \hm{\vartheta}\left(t\right)-e^{t\mathfrak{L}_{0}}\hm{\rho}_{0}\right\Vert _{\mathcal{H}}\rightarrow0,\ t\rightarrow\infty.
	\]
	Therefore by considering $\hm{\phi}_{0}-\hm{\rho}_{0}$ as the initial
	data, it reduces to find initial data for the homogeneous problem
	\begin{equation}
	\hm{u}_{t}=\mathfrak{L}\left(t\right)\hm{u},\ \hm{u}\left(t\right)=\pi_{c}\left(t\right)\hm{u}\left(t\right)\label{eq:homoredp}
	\end{equation}
	such that
	\[
	\left\Vert \hm{u}\left(t\right)-e^{t\mathfrak{L}_{0}}\hm{\psi}_{0}\right\Vert _{\mathcal{H}}\rightarrow0,\ t\rightarrow\infty
	\]
	where $\hm{\psi}_{0}=\hm{\phi}_{0}-\hm{\rho}_{0}$. Hereafter, we
	will only consider the homogeneous problem.
	
	\subsection{Basic estimates}

	Consider the following linear problem
	\[
	\hm{u}_{t}=\mathfrak{L}\left(t\right)\hm{u},\ \hm{u}\left(t\right)=\pi_{c}\left(t\right)\hm{u}\left(t\right)\in\pi_{c}\left(t\right)\mathcal{H}.
	\]
	From Theorem \ref{thm:mainthmHam}, for any given initial data $\hm{u}_{0}=\pi_{c}(0)\hm{u}_{0}$
	there exists $\hm{\phi}_{0}$ such that
	\begin{equation}
	\left\Vert \hm{u}(t)-e^{t\mathfrak{L}_{0}}\hm{\phi}_{0}\right\Vert _{\mathcal{H}}\rightarrow0,\ t\rightarrow\infty\label{eq:scatterAss}
	\end{equation}
	with
	\begin{equation}
	\left\Vert \hm{\phi}_{0}\right\Vert _{\mathcal{H}}\lesssim\left\Vert \hm{u}_{0}\right\Vert _{\mathcal{H}}\label{eq:linearscat1}
	\end{equation}
	with implicit constants independent of $\hm{u}$.
	
	This in particular defines a bounded linear map
	\begin{equation}
	\mathscr{L}_{S}:\,\mathcal{H}\rightarrow\mathcal{H},\ \ \mathscr{L}_{S}\left(\hm{u}_{0}\right)=\hm{\phi}_{0}.\label{eq:Ls}
	\end{equation}
	From \eqref{eq:linearscat1}, we know that
	\[
	\left\Vert \mathscr{L}_{S}\left(\hm{u}_{0}\right)\right\Vert _{\mathcal{H}}\lesssim\left\Vert \hm{u}_{0}\right\Vert _{\mathcal{H}}.
	\]
	We also claim the inverse inequality:
	\begin{lem}
		\label{lem:lowerLs}Using the notations above, one has the following
		estimate
		\begin{equation}
		\left\Vert \hm{u}_{0}\right\Vert _{\mathcal{H}}\lesssim\left\Vert \hm{\phi}_{0}\right\Vert _{\mathcal{H}}=\left\Vert \mathscr{L}_{S}\left(\hm{u}_{0}\right)\right\Vert _{\mathcal{H}}\label{eq:lowerLs}
		\end{equation}
		with implicit constants independent of $\hm{u}$.
	\end{lem}
	
	\begin{proof}
		The estimate above is a direct consequence of Strichartz estimates
		of the backwards flow. Assume that
		\begin{equation}
		\left\Vert \hm{u}(0)\right\Vert _{\mathcal{H}}\lesssim\left\Vert \hm{u}(s)\right\Vert _{\mathcal{H}}\label{eq:assums0}
		\end{equation}
		with constants independent of $s$ which will be proved in Theorem
		\ref{thm:backwardStri}.
		
		By the construction of $\mathscr{L}_{S}$, the scattering assumption
		implies that
		\[
		\left\Vert e^{-t\mathfrak{L}_{0}}\hm{u}(t)-\hm{\phi}_{0}\right\Vert _{\mathcal{H}}\rightarrow0,\ t\rightarrow\infty
		\]
		strongly in $\mathcal{H}$.
		
		From \eqref{eq:assums0}, it follows that
		\[
		\left\Vert \hm{u}(0)\right\Vert _{\mathcal{H}}\lesssim\left\Vert \hm{u}(s)\right\Vert _{\mathcal{H}}=\left\Vert e^{-s\mathfrak{L}_{0}}\hm{u}(s)\right\Vert _{\mathcal{H}}
		\]
		since $e^{s\mathfrak{L}_{0}}$ is an unitary operator. Passing to the
		limit in the strong sense on the RHS, we conclude
		\[
		\left\Vert \hm{u}_{0}\right\Vert _{\mathcal{H}}=\left\Vert \hm{u}(0)\right\Vert _{\mathcal{H}}\lesssim\lim_{s\rightarrow\infty}\left\Vert e^{-s\mathfrak{L}_{0}}\hm{u}(s)\right\Vert _{\mathcal{H}}=\left\Vert \hm{\phi}_{0}\right\Vert _{\mathcal{H}}
		\]
		as desired.
	\end{proof}
	Now we discuss the estimate \eqref{eq:assums0}. This actually is a
	direct consequence of a more general result on Strichartz estimates.
	\begin{thm}
		\label{thm:backwardStri}Using the notations above, we consider the
		equation
		\[
		\hm{u}_{t}=\mathfrak{L}\left(t\right)\hm{u}+\pi_{c}\left(t\right)\hm{F},\ \hm{u}(t)=\pi_{c}(t)\hm{u}\left(t\right).
		\]
		Using norms from Theorem \ref{thm:mainthmHam}, we have
		\[
		\left\Vert \hm{u}\right\Vert _{S_{\mathcal{H}}\bigcap\mathfrak{W}_{\mathcal{H}}}\lesssim\left\Vert \hm{u}\left(s\right)\right\Vert _{\mathcal{H}}+\left\Vert \hm{F}\right\Vert _{\left(\mathfrak{W}_{\mathcal{H}}\bigcap S_{\mathcal{H}}\right)^{*}}
		\]
		for any $s$. In particular, in the homogeneous case ($\hm{F}=0$) we have
		the backwards estimate
		\begin{equation}
		\left\Vert \hm{u}(0)\right\Vert _{\mathcal{H}}\lesssim\left\Vert \hm{u}(s)\right\Vert _{\mathcal{H}}.\label{eq:assums0-1}
		\end{equation}
	\end{thm}
	
	\begin{proof}
		As we discussed in the forward problem, we only need to restrict on
		the time interval $\left(-\infty,-T_{0}\right)\bigcup[T_{0,}\infty)$
		for some $T_{0}$ large enough so that all potentials are sufficiently
		separated. The estimates for the time internal $\left[-T_{0},T_{0}\right]$
		can be bounded with some constant large enough. As before, we decompose
		the space into $J+1$ channels. The free channel is the same as before.
		For the $j$-th channel, the $P_{c}(t)$ part analysis is again the
		same as the forward direction. 
		
		The only difference in the current setting from the earlier sections
		is the analysis of the stable and unstable spaces. We only sketch
		the difference here. 
		
		We still use the notations from Subsection \ref{subsec:analysisJchannel}. In the $j$-th channel. We denote
		$\tilde{\hm{v}}\left(t,y\right)=\hm{v}\left(t,y-c_{j}\left(t\right)\right)$.
		Then one has
		\[
		\tilde{\hm{v}}_{t}=\mathcal{L}_{j}\left(t\right)\tilde{\hm{v}}+\left[c_{j}'\left(t\right)\cdot\nabla\tilde{\hm{v}}+\hm{\tilde{F}}\right]
		\]
		where $\mathcal{L}_{j}\left(t\right)$ is defined via \eqref{eq:L2t}
		and
		\[
		\tilde{\hm{F}}=\left(\begin{array}{cc}
		0 & 0\\
		-\sum_{\rho\neq j}\left(V_{\rho}\right)_{\rho}\left(\cdot-\beta_{\rho}t+c_{\rho}\left(t\right)-c_{j}\left(t\right)\right) & 0
		\end{array}\right)\tilde{\hm{v}}+\hm{F}\left(t\right).
		\]
		Using the fact that $\left\Vert\tilde{\hm{v}}\right\Vert _{\mathcal{H}}$
		is bounded or grows subexponentially for $t\in\mathbb{R}$ from the
		energy estimate above since we are restricted onto the central direction,
		we obtain the following system
		\[
		\lambda_{j,k,+}\left(t\right)=-\int_{t}^{\infty}e^{\frac{\nu_{j.k}}{\gamma_{j}}\left(t-s\right)}\omega\left(a_{j}\left(t\right)\nabla\tilde{\hm{v}}+\hm{\tilde{F}},\mathcal{Y}_{j,k,\beta_j}^{-}\left(s\right)\right)\,ds
		\]
		\[
		\lambda_{j,k,-}\left(t\right)=-\int_{-\infty}^{t}e^{\frac{\nu_{j.k}}{\gamma_{j}}\left(s-t\right)}\omega\left(a_{j}\left(t\right)\tilde{\hm{v}}+\hm{\tilde{F}},\mathcal{Y}_{j,k,\beta_j}^{+}\left(s\right)\right)\,ds
		\]
		coupled with the equation in the stable space
		\[
		\left(P_{c}\left(t\right)\tilde{\hm{v}}\right)_{t}=\mathcal{L}_{j}\left(t\right)P_{c}\left(t\right)\tilde{\hm{v}}+P_{c}\left(t\right)\left[a_{j}\left(t\right)\nabla\tilde{\hm{v}}+\hm{\tilde{F}}\right].
		\]
		This system is different from the coupled system from before.
		
		To estimate $L_{t}^{2}\bigcap L_{t}^{\infty}$ norm of $\lambda_{j,k,\pm}$,
		we apply Young's inequality. Note that by the decay of eigenfunctions
		given by Agmon's estimate and the decay of the potential, we always
		have that
		\[
		\omega\left(\hm{V}_{\rho}\left(x-\beta_{\rho}t+c_{\rho}\left(t\right)-c_{j}\left(t\right)\right)\tilde{\hm{v}}\left(t\right),\mathcal{Y}_{j,k,\beta_j}^{\pm}\left(t\right)\right)\lesssim e^{-\alpha\left|t\right|}C_{\rho,T}
		\]
		for $\rho\neq j$ and some positive $\alpha$.  For discrete modes,
		one has{\small
			\begin{align}
				\sum_{k=1}^{K_{j}}\left(\left\Vert \lambda_{j,k,+}\left(t\right)\right\Vert _{L_{t}^{2}\bigcap L_{t}^{\infty}}+\left\Vert \lambda_{j,k,-}\left(t\right)\right\Vert _{L_{t}^{2}\bigcap L_{t}^{\infty}}\right) & \lesssim\left(e^{-\alpha B}+\delta\right)\left\Vert\tilde{\hm{v}}\right\Vert _{S_{\mathcal{H}}\bigcap\mathfrak{W}_{\mathcal{H}}}\nonumber +\left\Vert\hm{F}\right\Vert _{\left(\mathfrak{W}_{\mathcal{H}}\bigcap S_{\mathcal{H}}\right)^{*}}\label{eq:boud1-1-1}\\
				& \lesssim\sum_{\rho\neq j}\left(e^{-\alpha B}+\delta\right)C_{\rho,T}+\left\Vert \hm{F}\right\Vert _{\left(\mathfrak{W}_{\mathcal{H}}\bigcap S_{\mathcal{H}}\right)^{*}}.\nonumber 
		\end{align}}
		Therefore, we can apply the same bootstrap argument as before.
		
		Using any fixed $s\in\mathbb{R}$ as the initial data, we obtain that
		\[
		\left\Vert \hm{u}\right\Vert _{S_{\mathcal{H}}\bigcap\mathfrak{W}_{\mathcal{H}}}\lesssim\left\Vert \hm{u}\left(s\right)\right\Vert _{\mathcal{H}}+\left\Vert \hm{F}\right\Vert _{\left(\mathfrak{W}_{\mathcal{H}}\bigcap S_{\mathcal{H}}\right)^{*}}
		\]
		as desired. The estimate \eqref{eq:assums0-1} follow from the estimate
		above as special cases.
	\end{proof}
	
	\subsection{Construction of initial data}
	
	Let $\hm{u}_{t_{0}}\left(t\right)$ be the solution to \eqref{eq:homoredp}
	such that
	\begin{equation}
	\hm{u}_{t_{0}}\left(t_{0}\right)=\pi_{c}(t_{0})e^{t_{0}\mathfrak{L}_{0}}\hm{\phi}_{0}\label{eq:gammat0}
	\end{equation}
	where $\hm{\phi}_{0}\in\mathcal{H}$ is given.

	From our discussion on the scattering behavior the perturbed flow,
	there exists a free profile $\hm{\psi}_{t_{0}}$ such that
	\[
	\left\Vert \hm{u}_{t_{0}}\left(t\right)-e^{t\mathfrak{L}_{0}}\hm{\psi}_{t_{0}}\right\Vert _{\mathcal{H}}\rightarrow0,\ t\rightarrow\infty
	\]
	and
	\begin{equation}\label{eq:scattfinal}
	\left\Vert \hm{u}_{t_{0}}\left(t\right)\right\Vert _{S_{\mathcal{H}}\bigcap\mathfrak{W}_{\mathcal{H}}}\lesssim\left\Vert \pi_{c}(t_{0})e^{t_{0}\mathfrak{L}_{0}}\hm{\phi}_{0}\right\Vert _{\mathcal{H}}.
	\end{equation}
	The final preparation for the main proof the following estimate.
	\begin{lem}
		\label{lem:largeStri}Using that notations above, one has
		\begin{equation}
		\left\Vert \hm{u}_{t_{0}}\left(t\right)\right\Vert _{\left(S_{\mathcal{H}}\bigcap\mathfrak{W}_{\mathcal{H}}\right)_{t_{0}}}\rightarrow0,\,t_{0}\rightarrow\infty\label{eq:largeStr}
		\end{equation}
		where $\left(S_{\mathcal{H}}\bigcap\mathfrak{W}_{\mathcal{H}}\right)_{t_{0}}$
		is the subspace of $S_{\mathcal{H}}\bigcap\mathfrak{W}_{\mathcal{H}}$ from
		Theorem \ref{thm:backwardStri} restricted onto $[t_{0},\infty)$.
		
		From \eqref{eq:largeStr}, we also conclude that
		\begin{equation}
		\hm{\psi}_{t_{0}}\rightarrow\hm{\phi}_{0},\ t_{0}\rightarrow\infty\label{eq:convprofile}
		\end{equation}
		strongly in the $\mathcal{H}$ sense.
	\end{lem}
	
	\begin{proof}
		By Strichartz estimates from Theorem \ref{thm:backwardStri} for $F=0$,
		one has
		\[
		\left\Vert \hm{u}_{t_{0}}\left(t\right)\right\Vert _{\left(S_{\mathcal{H}}\bigcap\mathfrak{W}_{\mathcal{H}}\right)_{t_{0}}}\lesssim\left\Vert e^{t\mathfrak{L}_{0}}\pi_{c}(t_{0})e^{t_{0}\mathfrak{L}_{0}}\hm{\phi}_{0}\right\Vert _{\left(S_{\mathcal{H}}\bigcap\mathfrak{W}_{\mathcal{H}}\right)_{0}}.
		\]
		Taking the limit, we have
		\[
		\lim_{t_{0}\rightarrow\infty}\left\Vert \hm{u}_{t_{0}}\left(t\right)\right\Vert _{\left(S_{\mathcal{H}}\bigcap\mathfrak{W}_{\mathcal{H}}\right)_{t_{0}}}\lesssim\lim_{t_{0}\rightarrow\infty}\left\Vert e^{t\mathfrak{L}_{0}}\pi_{c}(t_{0})e^{t_{0}\mathfrak{L}_{0}}\hm{\phi}_{0}\right\Vert _{\left(S_{\mathcal{H}}\bigcap\mathfrak{W}_{\mathcal{H}}\right)_{0}}.
		\]
		We first note that $\pi_{c}(t_{0})e^{t_{0}\mathfrak{L}_{0}}\hm{\phi}_{0}\rightarrow e^{t_{0}\mathfrak{L}_{0}}\hm{\phi}_{0}$
		as $t_{0}\rightarrow\infty$ due to the pointwise decay of the free
		flow. It follows that
		\begin{align*}
			\lim_{t_{0}\rightarrow\infty}\left\Vert e^{t\mathfrak{L}_{0}}\pi_{c}(t_{0})e^{t_{0}\mathfrak{L}_{0}}\hm{\phi}_{0}\right\Vert _{\left(S_{\mathcal{H}}\bigcap\mathfrak{W}_{\mathcal{H}}\right)_{0}} & =\lim_{t_{0}\rightarrow\infty}\left\Vert e^{t\mathfrak{L}_{0}}e^{t_{0}\mathfrak{L}_{0}}\hm{\phi}_{0}\right\Vert _{\left(S_{\mathcal{H}}\bigcap\mathfrak{W}_{\mathcal{H}}\right)_{0}}\\
			& =\lim_{t_{0}\rightarrow\infty}\left\Vert e^{t\mathfrak{L}_{0}}\hm{\phi}_{0}\right\Vert _{\left(S_{\mathcal{H}}\bigcap\mathfrak{W}_{\mathcal{H}}\right)_{t_{0}}}.
		\end{align*}
		Therefore, by Strichartz estimates for the free flow, one has
		\[
		\lim_{t_{0}\rightarrow\infty}\left\Vert e^{t\mathfrak{L}_{0}}\hm{\phi}_{0}\right\Vert _{\left(S_{\mathcal{H}}\bigcap\mathfrak{W}_{\mathcal{H}}\right)_{t_{0}}}=0
		\]
		which implies
		\[
		\left\Vert e^{t\mathfrak{L}_{0}}\pi_{c}(t_{0})e^{t_{0}\mathfrak{L}_{0}}\hm{\phi}_{0}\right\Vert _{\left(\mathfrak{W}_{\mathcal{H}}\right)_{t_{0}}}\rightarrow0,\,t_{0}\rightarrow\infty,
		\]
		whence, by \eqref{eq:scattfinal}
		\[
		\lim_{t_{0}\rightarrow\infty}\left\Vert \hm{u}_{t_{0}}\left(t\right)\right\Vert _{\left(\mathfrak{W}_{\mathcal{H}}\right)_{t_{0}}}=0.
		\]
		Finally, we analyze \eqref{eq:convprofile}. By construction from the
		scattering and Strichartz estimates, one has
		\begin{align*}
			\left\Vert e^{-t\mathfrak{L}_{0}}\pi_{c}(t_{0})e^{t_{0}\mathfrak{L}_{0}}\hm{\phi}_{0}-\hm{\psi}_{t_{0}}\right\Vert _{\mathcal{H}} & =\left\Vert e^{-t\mathfrak{L}_{0}}\hm{u}_{t_{0}}\left(t_{0}\right)-\hm{\psi}_{t_{0}}\right\Vert _{\mathcal{H}}\\
			& \lesssim\left\Vert \hm{u}_{t_{0}}\left(t\right)\right\Vert _{\left(\mathfrak{W}_{\mathcal{H}}\right)_{t_{0}}}
		\end{align*}
		and
		\begin{align*}
			\lim_{t_{0}\rightarrow\infty}\left\Vert \hm{\phi}_{0}-\hm{\psi}_{t_{0}}\right\Vert _{\mathcal{H}} & =\lim_{t_{0}\rightarrow\infty}\left\Vert e^{-t_{0}\mathfrak{L}_{0}}\hm{u}_{t_{0}}\left(t_{0}\right)-\hm{\psi}_{t_{0}}\right\Vert _{\mathcal{H}}\\
			& \lesssim\lim_{t_{0}\rightarrow\infty}\left\Vert \hm{u}_{t_{0}}\left(t\right)\right\Vert _{\left(\mathfrak{W}_{\mathcal{H}}\right)_{t_{0}}}=0
		\end{align*}
		as desired.
	\end{proof}
	With preparations above, the final result follows naturally.
	\begin{proof}[Proof of Theorem \ref{thm:exiWave}]
		For any $\hm{\phi}_{0}\in\mathcal{H}$, we construct a sequence of
		solutions $\hm{u}_{t_{n}}(t)$ by \eqref{eq:gammat0} with $t_{n}\rightarrow\infty$.
		Using the notation of the scattering map \eqref{eq:Ls}, one has
		\[
		\mathscr{L}_{S}\left(\hm{u}_{t_{n}}(0)\right)=\hm{\psi}_{t_{n}}.
		\]
		From Lemma \ref{lem:largeStri},
		\[
		\left\Vert \hm{\psi}_{t_{n}}-\hm{\phi}_{0}\right\Vert _{\mathcal{H}}\rightarrow0,\ n\rightarrow\infty.
		\]
		Then by Lemma \ref{lem:lowerLs}, the convergence above implies that
		$\hm{u}_{t_{n}}(0)$ converges to some function $\hm{u}(0)\in\pi_{c}(0)\mathcal{H}$
		such that
		\[
		\mathscr{L}_{S}\left(\hm{u}(0)\right)=\phi_{0}.
		\]
		By construction, using $\hm{u}(0)$ as the initial data, the solution
		$\hm{u}(t)$ to the equation \eqref{eq:homoredp} satisfies
		\[
		\left\Vert \hm{u}(t)-e^{t\mathfrak{L}_{0}}\hm{\phi}_{0}\right\Vert _{\mathcal{H}}\rightarrow0,\ t\rightarrow\infty
		\]
		as desired.
	\end{proof}

	\bigskip
	
\end{document}